\documentclass[11pt, reqno, a4paper]{amsart}
\usepackage{fullpage} 
\usepackage{amssymb, amsmath, amscd}
\usepackage{amsthm}

\usepackage{verbatim}

\usepackage{color}

\usepackage{etex}
\usepackage{dsfont}
\usepackage[matrix, arrow, curve, frame, arc]{xy}
\usepackage{pgf}
\usepackage{tikz-cd}
\usetikzlibrary{arrows,shapes,automata,backgrounds,petri,calc}
\usepackage[square,sort,comma,numbers]{natbib}
 \usepackage{url}
 \usepackage{hyperref}
 \usepackage[shortlabels]{enumitem} %to index the enumeration

\newcommand{\tikzAngleOfLine}{\tikz@AngleOfLine}
\def\tikz@AngleOfLine(#1)(#2)#3{%
\pgfmathanglebetweenpoints{%
\pgfpointanchor{#1}{center}}{%
\pgfpointanchor{#2}{center}}
\pgfmathsetmacro{#3}{\pgfmathresult}%
}

\newcommand{\bS}{\mathbb{S}}

\newcommand{\bN}{\mathbb{N}}
\newcommand{\bM}{\mathbb{M}}
\newcommand{\bE}{\mathbb{E}}
\newcommand{\bA}{\mathbb{A}}

\newcommand{\bfQ}{{\bf Q}}
\newcommand{\bbQ}{\mathbb{Q}}

\newcommand{\bZ}{\mathbb{Z}}
\newcommand{\cB}{\mathcal{B}}
\newcommand{\bD}{\mathbb{D}}
\newcommand{\wt}{\widetilde}
\newcommand{\ba}{\bar{\alpha}}
\newcommand{\La}{\Lambda}
\newcommand{\vf}{\varphi}
\newcommand{\ve}{\varepsilon}

\newcommand{\cA}{\mathcal{A}}

\newcommand{\cK}{\mathcal{K}}

\newcommand{\rad}{\operatorname{rad}}
\newcommand{\soc}{\operatorname{soc}}
\renewcommand{\mod}{\operatorname{mod}}
\newcommand{\Hom}{\operatorname{Hom}} 
\newcommand{\End}{\operatorname{End}} 
\newcommand{\Adj}{\operatorname{Ad}}
 
\newcommand{\op}{\operatorname{op}} 
\newcommand{\proj}{\operatorname{proj}} 
\newcommand{\add}{\operatorname{add}} 
\newcommand{\ima}{\operatorname{im}} 
\newcommand{\cok}{\operatorname{coker}} 
\newcommand{\alg}{\operatorname{-alg}}

\def\vec#1{\left[\begin{smallmatrix}#1\end{smallmatrix}\right]}

\begin{document}

\newtheorem{defi}{Definition}[section]
\newtheorem{rem}[defi]{Remark}
\newtheorem{prop}[defi]{Proposition}
\newtheorem{ques}[defi]{Question}
\newtheorem{lemma}[defi]{Lemma}
\newtheorem{cor}[defi]{Corollary}
\newtheorem{thm}[defi]{Theorem}
\newtheorem{expl}[defi]{Example} 
\newtheorem*{prob}{Problem}
\newtheorem*{mthm}{Main Theorem}

\parindent0pt

\title[Small TSP4 algebras]{Tame symmetric algebras of period four with small Gabriel quivers} 

\author[K. Erdmann]{Karin Erdmann}
\address[Karin Erdmann]{Mathematical Institute, University of Oxford, ROQ, Oxford OX2 6GG, United Kingdom} 
\email{erdmann@maths.ox.ac.uk} 

\author[A. Jaworska-Pastuszak]{Alicja Jaworska-Pastuszak}
\address[Alicja Jaworska-Pastuszak]{Faculty of Mathematics and Computer Science, 
Nicolaus Copernicus University, Chopina 12/18, 87-100 Toru\'n, Poland}
\email{czogori@mat.umk.pl}

\author[A. Skowyrski]{Adam Skowyrski}
\address[Adam Skowyrski]{Faculty of Mathematics and Computer Science, 
Nicolaus Copernicus University, Chopina 12/18, 87-100 Toru\'n, Poland}
\email{skowyr@mat.umk.pl} 


\begin{abstract} The tame symmetric algebras of period four, TSP4 algebras for short, form an important class 
of algebras, with interesting links to various branches of modern algebra. The study of this class has been   
recently developed in two major directions. The first embraces new classes of examples of TSP4 algebras, 
such as virtual mutations \cite{VM} and generalized weighted surface algebras \cite{GWSA}, both extending the known class 
of weighted surface algebras, investigated in \cite{WSA,WSAG,WSAC}. The second provides new classifications of TSP4 
algebras (based on known results for $2$-regular case \cite{AGQT}), which handle algebras, whose Gabriel quivers 
satisfy more general properties; see \cite{EHS3} for the classification in biregular case (the biserial case is 
a current work in progress). An ongoing project \cite{PS1,PS2} sheds a new light on the combinatorics of such 
algebras, introducing a new useful tool for their classification, called periodicity shadows. In 
this paper, we attack the problem of classification of TSP4 algebras, from another perspective, namely, 
we give a classification of all TSP4 algebras with not too big Gabriel quivers, i.e. having 
at most $5$ vertices -- but with no restrictions on their structure, as it was the case for previous 
classifications. The result is based on the application of the notion of periodicity shadow, 
which allows computing all possible Gabriel quivers of such algebras (for small number of vertices), 
and recent results on iterated mutations of algebras with periodic simple modules \cite{mut}. The main result 
shows that TSP4 algebras with Gabriel quivers having at most $5$ vertices are generalized weighted surface algebras, 
confirming a general conjecture in this case. 
\end{abstract}

\maketitle

\section{Introduction}\label{sec:1} 

Throughout the paper, by an {\it algebra} we mean a finite-dimensional, basic and connected $K$-algebra over a fixed 
algebraically closed field. Given an algebra $\La$, we denote by $\mod\La$ the category of finitely generated 
right $\La$-modules. It is well known that any algebra $\La$ admits a {\it presentation}, that is, we 
have an isomorphism $\La\cong KQ/I$, where $KQ$ is the path algebra of a (finite) quiver $Q=(Q_0,Q_1,s,t)$, and $I$ 
is an admissible ideal of $KQ$. Such quiver $Q$ (denoted by $Q_\La$) is unique up to permutation of vertices and 
it is called the {\it Gabriel quiver} of $\La$. For more details about basic notions we refer the reader to Section 
\ref{sec:2}. For the basic background on the representation theory of algebras we refer the reader to books 
\cite{ASS,SkY}. \medskip 

Our main concern is the classification of TSP4 algebras, which reduces to a description of possible 
presentations $\La=KQ/I$ of TSP4 algebras based on the three properties: tameness, symmetricity and periodicity 
(period $4$). \smallskip 

We recall that algebras split into two disjoint classes of {\it tame} and {\it wild} algebras, 
due to the well-known Tame and Wild Theorem \cite{Dro}, where the algebra $\La$ is tame if and only if in each 
dimension $d$, the indecomposable modules occur in a finite number of discrete and a finite number of one-parameter 
families. We will not use this definition in the sequel, however, tameness implies some useful restrictions on the 
Gabriel quiver of the algebra (i.e. it has not 'too much' arrows), discussed in the first part of Section \ref{sec:2}. \smallskip 

The symmetric algebras form a prominent class of algebras, including classical examples such as the blocks of 
group algebras of finite groups or the Hecke algebras of finite Coxeter groups. Recall that an algebra 
is {\it symmetric} if and only if there exists an associative non-degenerate symmetric $K$-bilinear form 
on $\La$, or equivalently, $\La$ and $D(\La)$ are isomorphic as $\La$-bimodules, where $D=\Hom_K(-,K)$ is 
the standard duality on $\mod\La$. We note that symmetric algebras are examples of the {\it self-injective} 
algebras, i.e. algebras with projective $\La$-modules being injective. In particular, if $\La=KQ/I$ is symmetric, then 
$P_i\simeq I_i$, for any $i\in Q_0$, where $P_i$ (respectively, $I_i$), for $i\in Q_0$, form a complete set 
of indecomposable projective (injective) $\La$-modules, indexed by vertices of the Gabriel quiver $Q=Q_\La$ 
of $\La$. \smallskip 

For an algebra $\La$, we denote by $\Omega:=\Omega_\La$, the {\it syzygy} operator, which assigns to a 
module $X$ in $\mod\La$, the kernel $\Omega(X):=\ker d$ of an arbitrary projective cover $d=d_0:P\to X$ of 
$X$ in $\mod\La$. Taking all iterations $\Omega^k(X)$ one obtains the projective resolution of $X$, 
determined by a sequence of projective covers $d_{k}:P_k\to\Omega^k(X)$, $k\geqslant 0$. If for some 
$d\geqslant 1$, we have $\Omega^d(X)\simeq X$, then we call $X$ a {\it periodic} module. The smallest 
such a number $d$ is called {\it period} of $X$, and a periodic module of period $d$ is simply called a 
{\it $d$-periodic} module. 

We note that an algebra $\La$ is called {\it periodic} (of period $d$), if it is $d$-periodic as a $\La$-bimodule, 
or equivalently, as a (right) module over its enveloping algebra $\La^e=\La^{\op}\otimes_K\La$. 
These two notions are related: periodicity of the algebra $\La$ implies periodicity of all non-projective 
modules in $\mod\La$. \smallskip 

The class of tame symmetric algebras of period $4$ (TSP4) is of special interest for us. Classical examples 
of such algebras come from the modular representation theory, especially the seminal classification of 
blocks of group algebras with dihedral, semidihedral or quaternion defect groups \cite{E90}. Motivated by known 
properties of blocks with generalized quaternion defect group Erdmann introduced the so called {\it algebras of 
quaternion type}, which are by definition, the tame symmetric algebras of infinite representation type with 
non-singular Cartan matrix, whose indecomposable non-projective modules are periodic of period dividing $4$. 
A natural extension of this class leads to the so called {\it algebras of generalized quaternion type} \cite{AGQT} 
(GQT algebras for short), that is, the tame symmetric algebras of infinite representation type with all simple modules 
periodic of period $4$. We note that every TSP4 algebra of infinite type is automatically GQT, and there are no 
known examples of GQT algebras which are not TSP4. It is conjectured that these two classes coincide and most of 
the results of this paper are true for all GQT (so TSP4) algebras, however at some point we need TSP4, because GQT is 
too weak assumption (see proofs of Propositions \ref{prop:6.1}-\ref{prop:6.3}). \smallskip   

The classification of TSP4 algebras in general is yet not so easy, but we can formulate the 
following problem, whose solution seems within reach.  

\begin{prob} Classify up to isomorphism all TSP4 (relatively, GQT) algebras $\La=KQ/I$ with $n=|Q_0|\leqslant 6$, 
i.e. with the set $Q_0$ having at most $6$ vertices. \end{prob} \medskip 

In this paper, we will provide a solution for $n\leqslant 5$. Cases $n\leqslant 3$, called very small, are treated 
separately, since they are rather direct, and require no new methods (see Sections \ref{subs:6.1}-\ref{subs:6.2}). 
For $n\geqslant 3$, the classification is accessible thanks to the notion of periodicity shadow \cite{PS1,PS2}, which 
is used to get a rigorous bound for the shapes of Gabriel quivers, and ultimately, to get a full description 
of those (for $n=3$ this can be still done by hand, but shadows give a direct answer quicker). For $n=3,4$ and 
$5$ computations performed in \cite{PS2} left us with relatively small number of cases to consider, and this is the 
first natural step. For $n=6$ the number is too big for a case-by-case inspection, 
and this case is planned as a next step, however needing more conceptual approach based on restrictions obtained 
in cases $n\leqslant 5$. \medskip 

Generally speaking, the notion of a periodicity shadow is strictly related to study the `shadow invariant', 
that is, the following function 
$$\bS_\bullet:K\alg(n)\to\bM^*_n,$$ 
between the set $K\alg(n)$ of all algebras $\La$ with $n$-vertex Gabriel quivers and the set $\bM^*_n$ of all 
skew-symmetric integer matrices, which associates to any algebra $\La\in K\alg(n)$ its {\it shadow}, i.e. the 
signed adjacency matrix 
$$\bS_\La=\Adj_{Q_\La}$$ 
of its Gabriel quiver, where for a quiver $Q$, the matrix $\Adj_Q=[a_{ij}]$ has entries $a_{ij}$ equal to the 
difference between the number of arrows $i\to j$ and the number of arrows $j\to i$ in $Q$. It is known that 
skew-symmetric matrices are in one-to-one correspondence with $2$-acyclic quivers, and one can view the 
shadow $\bS_\La$ of $\La$ as the skew-symmetric matrix associated to the reduced Gabriel quiver $Q_\La^\times$, 
where for a quiver $Q$, the quiver $Q^\times$ is obtained from $Q$ by deleting all loops and $2$-cycles. \smallskip 

The shadow can be defined for arbitrary algebra, but it is especially interesting notion in case of TSP4 algebras. 
Namely, in \cite{PS1}, we defined a set $\bS(n)$ of so called {\it tame periodicity shadows}, which contains the 
set $\mathbb{TSP}4:=\bS_{TSP4}$ of all shadows of TSP4 algebras. 
The technical definition can be found in Section \ref{sec:4}, here we only mention that the set $\bS(n)$ is 
computable for small $n$. Moreover, in \cite{PS2} we presented an algorithm allowing to compute the set $\bS(n)$, 
and introduced a bit stronger notion of the {\it essential shadow}, or in other words, we studied a smaller 
superset $\bE(n)$ of $\mathbb{TSP}4$: $\mathbb{TSP}4\subset\bE(n)\subset\bS(n)$. We also computed the 
sets $\bS(3),\bS(4)$ and $\bE(5)$, which are briefly recalled in Section \ref{sec:4}, since they are crucial for the 
classification. Having the set $\bE(n)$ of all essential shadows, for some $n$, we can reconstruct all 
Gabriel quivers of TSP4 algebras $\La$ with $n$ vertices. Indeed, for a fixed $A\in\bE(n)$, we first 
find all possible quivers $Q=Q_\La$, $\La\in$TSP4, obtained from the $2$-acyclic quiver 
identified with $A$ by adding $2$-cycles or loops. This is done using a non-trivial result from \cite{PS1}, 
called the Reconstruction Theorem (see Theorem \ref{thm:recon}), which gives very strong conditions restricting 
the position of $2$-cycles in Gabriel quivers of TSP4 (GQT) algebras. Position of loops is also enough controlled to get a full 
description of possible Gabriel quivers $Q=Q_\La$ of TSP4 (GQT) algebras $\La$ for small $n$ (see Section 
\ref{sec:4}). In other words, for any $A\in\bE(n)$ we can describe the Gabriel quivers of TSP4 algebras 
in the fibre $\bS_\bullet^{-1}(A)$. In some cases (see Sections \ref{subsec:n=4} and \ref{subsec:n=5}), 
we obtain empty intersection $\bS_\bullet^{-1}(A)\cap$TSP4, so that the shadow 
$A$ can be excluded, that is, there are no TSP4 (GQT) algebras $\La$ with $\bS_\La=A$. 
After such analysis for all shadows $A$ in $\bE(4)$ or $\bE(5)$, we obtain the list 
of all possible Gabriel quivers of TSP4 (or GQT) algebras in Propositions \ref{prop:Ga4} and \ref{prop:Ga5}. The 
case $n=3$ is rather direct, and requires no case study; in cases $n=1$ or $2$, the classification of all tame 
symmetric algebras is known, and the problem reduces to check which of them are GQT (or TSP4) (see Sections 
\ref{subs:6.1}-\ref{subs:6.2}). 
\medskip 

The main result of this article shows that all TSP4 algebras with Gabriel quivers having $n\in\{3,4,5\}$ vertices 
are the so called {\it generalized weighted surface algebras}, GWSA for short, which form a central class of algebras 
for this paper. Such algebras were introduced and investigated in \cite{GWSA}, where a detailed description 
of their presentation was also given. We only mention that this class appeared naturally as a generalization of the 
{\it weighted surface algebras} \cite{WSA,WSAC,WSAG}, defined by triangulation quivers (see Section \ref{subs:3.2}), 
and their virtual mutations \cite{VM}. We want to avoid technicalities concerning generators of admissible ideals $I$ 
determining presentation $\La=KQ/I$ of a GWSA (these are not really important for the paper), so we only 
discuss the necessary details needed for the description of the Gabriel quivers appearing in the classification 
(and some general remarks on the shape of generators of admissible ideals $I$). For the complete list of 
generators, see \cite[Definition 4.1]{GWSA}. \smallskip 

By a {\it generalized weighted surface algebra} we mean a quotient algebra $\La=KQ/\wt{I}$, given 
by so called generalized triangulation quiver $Q$, and an ideal $\wt{I}$, which is not always 
admissible, but its generators are uniquely determined by the shape of $Q$. A quiver $Q$ is called 
a {\it generalized triangulation quiver} if it is a glueing of a finite number $B_1,\dots,B_m$ of of blocks of 
the following five types 
$$
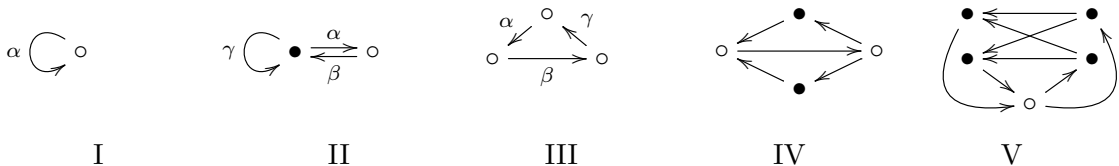
\xymatrix@C=0.3cm@R=0.2cm{\\ \ar@(lu, ld)[]_{\alpha}\circ} \qquad \qquad
\xymatrix@C=0.6cm@R=0.2cm{&\\ \ar@(lu, ld)[]_{\gamma}\bullet\ar@<.35ex>[r]^{\alpha}&\ar@<.35ex>[l]^{\beta}\circ} 
\quad \qquad 
\xymatrix@C=0.3cm@R=0.2cm{&\circ\ar[ld]_{\alpha}&\\ \circ\ar[rr]_{\beta}&&\circ\ar[lu]_{\gamma}}\quad \qquad 
\xymatrix@C=0.6cm@R=0.1cm{&\bullet\ar[ld] &\\ \circ\ar[rr] &&\circ\ar[ld]\ar[lu] \\&\bullet\ar[lu]&} \qquad 
\xymatrix@C=0.4cm@R=0.2cm{ \bullet \ar@/_23pt/[rdd] && \ar[ll] \bullet \ar[lld] \\ 
\bullet \ar[rd] && \bullet \ar[ll] \ar[llu] \\ 
& \circ \ar[ru] \ar@/_23pt/[ruu] & }$$ 
$$\mbox{ I }\qquad\quad\qquad\quad\quad\mbox{ II }\qquad\qquad\qquad\mbox{ III }\qquad\qquad\qquad\mbox{ IV }
\qquad\qquad\qquad\mbox{ V}$$
where by a glueing we mean that each vertex $\circ$ in every block $B_i$ is glued with exactly one 
vertex $\circ$ in a block $B_j$, for $j\neq i$. Then the ideal $\wt{I}$ is precisely determined by the blocks 
$B_1,\dots,B_m$ defining $Q$ and collections of integer weights $m_\bullet\geqslant 1$ and parameters 
$c_\bullet\in K^*$ associated to arrows of $Q$. Actually, weights and parameters are constant on $g$-orbits of 
some permutation $g$ of arrows of $Q$, and the $g$-orbits together with weights $m_\bullet$ determine the 
Gabriel quiver of $\La$. Namely, $Q_\La\subset Q$ is obtained from $Q$ by deleting so called {\it virtual arrows}, 
which are arrows $\alpha$ in $Q_1$ satisfying $m_\alpha n_\alpha=2$, where $n_\alpha$ is the length of $g$-orbit 
of $\alpha$. Then we may read off the presentation of $\La=KQ_\La/I$, where generators of $I$ are obtained from 
generators of $\wt{I}$ by substituting virtual arrows (a bit more details concerning generators in Section 
\ref{subs:3.2}). \medskip 

Now, the main result of this paper is stated as follows. 

\begin{mthm} Let $\La$ be a TSP4 algebra with Gabriel quiver $Q=Q_\La$ having $n=3,4$ or $5$ vertices. Then 
$\La$ is socle equivalent to a generalized weighted surface algebra. \end{mthm}
  
The proof of the above theorem is divided into two major parts. The first part is devoted to determine all 
possible Gabriel quivers of TSP4 (even GQT) algebras with $n=3,4$ or $5$ vertices, and this is done in Section \ref{sec:6}. 
In particular, we confirm that the form of the Gabriel quivers is exactly as expected, i.e. these are the Gabriel 
quivers of GWSA's, which are generalized triangulation quivers modulo virtual arrows. 
With this, the second part deals with describing all possible TSP4 algebra structures $\La=KQ/I$ given by quivers $Q$ 
computed in the first, and this is done in Section \ref{sec:algstr}. More precisely, we show there that 
for each allowed $Q$, the only TSP4 algebras $\La$ with $Q_\La=Q$ are the GWSA's with this quiver. \smallskip 

The rest part of the article provides necessary preparation for the proof. In Section \ref{sec:2} we recall a few 
basic notions from the representation theory, including derived equivalence and related notion of mutation of 
a symmetric algebra at a vertex of its Gabriel quiver. Section \ref{sec:3} is devoted to present some preparatory 
results on TSP4 (GQT) algebras, among others, the results concerning position of non-regular vertices in biserial 
Gabriel quivers of GQT algebras, which are crucial for analysis in Section \ref{sec:6}. 
The remaining Section \ref{sec:4} contains a quick introduction to the notion of periodicity shadow, where in particular, 
we present necessary computations needed for description of the Gabriel quivers in Section \ref{sec:6}. \medskip 

\section*{Acknowledgements} The third named author has been supported from the research grant no. 2023/51/D/ST1/01214 
of the Polish National Science Center. 

\section{Preliminaries}\label{sec:2} 

By a quiver we mean a quadruple $Q=(Q_0,Q_1,s,t)$, where $Q_0$ is a finite set of vertices, $Q_1$ a finite set 
of arrows and $s,t:Q_1\to Q_0$ functions assigning to every arrow $\alpha$ its source $s(\alpha)$ and target 
$t(\alpha)$. For a quiver $Q$, we denote by $KQ$ the {\it path algebra} of $Q$, whose $K$-basis is given by all 
paths of length $\geqslant 0$ in $Q$. Recall that the Jacobson radical of $KQ$ is the ideal $R_Q$ of $KQ$ generated 
by all paths of length $\geqslant 1$, and an ideal $I$ of $KQ$ is called {\it admissible}, provided that 
$R_Q^m\subseteq I \subseteq R_Q^2$, for some $m\geqslant 2$. Moreover, the trivial paths $\ve_i$ (of length $0$) 
at vertices $i\in Q_0$ form a complete set of (pairwise orthogonal) primitive idempotents of $KQ$ with 
$\sum_{i\in Q_0}\ve_i$ being the identity of $KQ$. We assume further that the vertices of a quiver $Q$ are always 
labeled with numbers $1,\dots,n$, i.e. $Q_0=\{1,\dots,n\}$. \medskip 

If $Q$ is a quiver and $I$ an admissible ideal $I$ of $KQ$, then $(Q,I)$ is said to be a {\it bound quiver}, 
and the associated quotient algebra $KQ/I$ is called a {\it bound quiver algebra}. Any algebra $\Lambda$ over an 
algebraically closed field admits a presentation as a bound quiver algebra $\Lambda\cong KQ/I$. In this case, the cosets 
$e_i=\varepsilon_i+I\in \La$ form a complete set of primitive orthogonal idempotents of $\La$ and $\sum_{i\in Q_0}e_i$ 
is the identity of $\Lambda$. For a path $w$ in $Q$, we will write $w\prec I$, if $w$ appears as a summand 
in one of the relations $\rho_1,\dots,\rho_n$ forming a set of minimal generators of $I$. \smallskip 

We recall that for any algebra $\La$ and a complete set of primitive orthogonal idempotents $e_i$, $i\in\{1,\dots,n\}$, 
with $\sum e_i=1_\La$, the {\it Gabriel quiver} of $\Lambda$ (sometimes also called the {\it ordinary quiver}) is the 
quiver $Q_\La$ with vertex set $Q_0=\{1,\dots,n\}$ and the arrows $i\to j$ in a bijective correspondence with 
a basis of the $K$-vector space $e_i(J/J^2)e_j$, where $J$ is the Jacobson radical $J=J_\La=R_Q+I$ of 
$\Lambda$. This quiver is uniquely determined by the algebra (up to permutation of vertices), and does not depend 
on the choice of a complete set of primitive othogonal idempotents \cite[see II.3]{ASS}. Moreover, for any presentation 
$\Lambda=KQ/I$, the Gabriel quiver of $\La$ is $Q_\La=Q$, and $Q_\Lambda$ is connected if and only if $\La$ is 
indecomposable as an algebra. \smallskip 

We note that the modules $P_i:=e_i\La$, for $i\in\{1,\dots,n\}$, form a complete set of indecomposable projective 
modules in $\mod\La$, and the modules $I_i:=D(\La e_i)$, $i\in\{1,\dots,n\}$, form a complete set of indecomposable 
injective modules in $\mod\La$. Then the modules $S_i=e_i\La/e_iJ$, $i\in\{1,\dots,n\}$, form 
a complete set of simple modules in $\mod\La$. We assume without mentioning that all algebras are self-injective 
(even symmetric), and we will only work with projective modules. 

We use notation $p_i$ for a dimension vector of the projective $P_i$ associated to a vertex $i\in Q_0$, that is, 
vector $x=[x_s]^T\in\bN^n$, where $x_s=\dim_K e_i\La e_s$. Additionally, for a vertex $i$, we denote by $P_i^+$ 
the direct sum $\bigoplus_{i\to j} P_j$ of projectives corresponding to targets of arrows starting from $i$, and 
by $p_i^+$ the dimension vector of $P_i^+$. Similarly, $P_i^-$ denotes the direct sum $\bigoplus_{j\to i} P_j$ of 
projectives corresponding to sources of arrows ending at $i$, and $p_i^-$ its dimension vector. \medskip

For a vertex $i$ of a quiver $Q$, we write $i^+$ for the set $\{\alpha\in Q_1| \ s(\alpha)=i \}$ of arrows 
starting at $i$, and $i^-$ for the set $\{\alpha\in Q_1| \ t(\alpha)=i\}$ of arrows ending at $i$ (including loops 
in $i^-\cap i^+$). Then $i$ is called a {\it $(p,q)$-vertex}, or sometimes {\it $(p,q)$-regular}, provided that 
$|i^-|=p$ and $|i^+|=q$. By an $r$-vertex (or an $r$-regular vertex) we mean an $(r,r)$-vertex. Moreover, a vertex 
is said to be {\it at most (at least) $r$-regular}, if it is a $(p,q)$-vertex with $p,q\leqslant r$ ($\geqslant r$). 
Vertices which are not $r$-regular, for any $r\geqslant 1$, are called {\it non-regular}. \smallskip 
 
We will frequently use quivers called {\it $2$-regular}, {\it biregular} or {\it biserial}. By these we mean the 
quivers $Q$ such that any vertex $i\in Q_0$ is, $2$-regular, $1$- or $2$-regular or at most $2$-regular, respectively. 
We say that an algebra is {\it $2$-regular}, {\it biregular} or {\it biserial}, if so is its Gabriel quiver. \medskip

Recall that for any vertex $i$ such that $S_i$ is $4$-periodic, we have the following 
exact sequence 
$$0\to S_i\to P_i \stackrel{d_3}\to  P_i^- \stackrel{d_2}\to  P_i^+ \stackrel{d_1}\to P_i \to S_i\to 0$$ 
with $Im(d_k)\cong\Omega^k(S_i)$, for $k\in\{1,2,3\}$. One can assume that $d_1$ is given by arrows 
starting from $i$ and $d_3$ by arrows ending at $i$ (see \cite[Lemma 4.1 and Proposition 4.3]{AGQT}). Moreover, 
comparing dimension vectors, we get that $p_i^-=p_i^+$, for any $i\in Q_0$, which gives rise to the crucial 
equation involving the Cartan matrix of $\La$, and leading to the notion of shadow (see condition (PS3), Section 
\ref{sec:4}). \medskip

Algebras considered later have their Gabriel quivers built from some smaller `blocks', as mentioned briefly in 
the Introduction; for more details see 
Section \ref{subs:3.2}. Here we only mention that {\it a block} (in $Q$) is a subquiver $\Gamma=(\Gamma_0,\Gamma_1)$ of 
$Q$, such that its set of vertices $\Gamma_0$ is a disjoint union $\Gamma_0=B\cup W$, where any arrow 
$\alpha\in Q_1\setminus\Gamma_1$ has a source and target in $Q_0\setminus B$. 
The vertices of $\Gamma$ contained in $B$ are labelled by $\bullet$, whereas vertices in $W$, by $\circ$. If $\Gamma$ is 
a block in $Q$, then $Q$ is a `glueing' of $\Gamma$ with the remaining part of $Q$ consisting 
of arrows between vertices in $Q_0\setminus B$. \smallskip 

For two quivers $Q$ and $Q'$ having the same vertex set $Q_0=Q_0'$, we 
denote by $Q\sqcup Q'$ the quiver $(Q_0,Q_1\sqcup Q'_1)$ on the same vertex set, whose set of arrows 
is a disjoint union of the sets of arrows $Q_1$ and $Q'_1$. This will be 
sometimes called a {\it disjoint union of quivers}, which is different from glueing of blocks, in the above 
intuitive sense. Although, we can view glueing as a disjoint union. Namely, if $\Gamma=(\Gamma_0,\Gamma_1)$ 
is a block in $Q=(Q_0,Q_1)$, then we have two associated subquivers $\Gamma_{*}=(Q_0,\Gamma_1)$ and 
$Q'=(Q_0,Q_1')$, where $Q_1'$ consists of all arrows between vertices in $Q_0\setminus B$. Then $Q$ is a 
glueing of $(Q_0\setminus B,Q_1')$ and $\Gamma$ (identifying common vertices in $W$), or a disjoint union 
$Q=\Gamma_{*}\sqcup Q'$ of quivers on the same set of vertices. \bigskip 

We only mention that all algebras in this paper are assumed to be tame (see the Introduction). We shall present now 
some consequences of this assumption on the Gabriel quivers of algebras. We omit mostly elementary arguments; 
more details can be found in \cite[Section 2]{PS1}. \medskip 

First, every Gabriel quiver $Q=Q_\La$ of a tame algebra has at most $2$ arrows between any pair of vertices, 
which is indicated in the condition (T1) from the definition of a tame shadow (see Section \ref{sec:4}). Clearly, 
also any vertex $i\in Q_0$ admits at most $1$ loop (if $|Q_0|\geqslant 2$) and at most $4$ distinct arrows starting 
(respectively, ending) at $i$; see also condition (T3) in the definition of a tame shadow. Similarily, if $Q$ 
contains a subquiver of the form 
$$K_2^+=\xymatrix@C=0.3cm{\circ&&\ar[ll]\circ\ar@<+0.4ex>[rr]\ar@<-0.4ex>[rr] && \circ} \ \ \mbox{ or } \ \ 
K_2^-=\xymatrix@C=0.3cm{\circ\ar[rr]&& \circ && \ar@<+0.4ex>[ll]\ar@<-0.4ex>[ll] \circ} $$  
then its path algebra $A$ is a wild specialization of $\La$ in the sense of \cite[see 1.2]{Rin} (obtained 
from $\La$ by deleting all vertices not in the subquiver). In particular, it follows that $\La$ is wild, 
hence $Q$ cannot contain such a subquiver. This restriction is reflected in the condition (T2). \medskip  

Secondly, we say that $\Delta$ is a subquiver {\it of type} $K_2^*$ if $\Delta$ is one of the following two quivers 
$$ \xymatrix@C=0.6cm{\circ \ar@<+0.4ex>[r]^{\ba}\ar@<-0.4ex>[r]_{\alpha}&\circ\ar[r]^{\beta}&\circ} \ \mbox{ or } \ 
\xymatrix@C=0.6cm{\circ \ar[r]^{\beta}&\circ \ar@<+0.4ex>[r]^{\ba}\ar@<-0.4ex>[r]_{\alpha}&\circ}$$
and both $\alpha\beta, \bar{\alpha}\beta\nprec I$ (or both $\beta\alpha,\beta\bar{\alpha}\nprec I$). Note that 
if $\La$ is tame, its Gabriel quiver $Q=Q_\La$ does not contain a subquiver of type $K_2^*$. Indeed, it is 
sufficient to see that if $\Delta$ is a subquiver in $Q$ of type $K_2^*$, then $\La$ admits a specialization 
to the path algebra $K\Delta$, which is a wild hereditary algebra (obtained by deleting all vertices and arrows 
not in $\Delta$). Some of the subquivers of type $K_2^*$ are excluded in the definition of the essential shadow; 
see condition (PS5). \medskip 

Finally, we will use the following abbreviation. We say that an algebra $\La$ {\it has a subcategory $A$ 
in covering}, if there is a factor algebra $B=\La/N$ of $\La$ (more generally, a specialization) such that $B$ 
admits a Galois covering $F:\tilde{B}\to B$ and $A$ is a full subcategory of $\tilde{B}$. In particular, it 
follows from the well known results (see \cite[Proposition 2]{DS1} and \cite[Theorem]{DS2}) that anytime $\La$ 
has $A$ as a subcategory in covering and $A$ is a wild algebra, then also $\La$ is wild. In the sequel, we will 
frequently apply this fact without mentioning. 

\bigskip 

We finish this section with recalling a few basic facts on derived equivalences of algebras and related 
notion of the mutation at vertex. Let $\La$ be an algebra. We denote by $\cK^b(\mod\La)$ the homotopy category of 
bounded complexes of modules. We will use 
notation $\cK^b_\La$ for the subcategory $\cK^b(\proj\La)$ formed by bounded complexes of projective modules. 
The derived category $D^b(\mod\La)$ of $\La$ is the localization of $\cK^b(\mod\La)$ with respect to quasi-isomorphisms, 
and admits the structure of a triangulated category, where the suspension functor is given by left shift $(-)[1]$ 
(see \cite{Hap}). Two algebras $\La$ and $\La'$ are called {\it derived equivalent} provided their derived categories 
$D^b(\mod\La)$ and $D^b(\mod\La')$ are equivalent as triangulated categories. Recall that a complex $T\in\cK^b_\La$ is 
called a {\it tilting complex} \cite{Rick1}, if the following conditions are satisfied: \begin{enumerate} 
\item[(1)] $\Hom(T, T[i])=0$, for all integers $i\neq 0$,
\item[(2)] $\add T$ generates $\cK^b_\La$ as triangulated category. \smallskip 
\end{enumerate} 

We have the following well known criterion for verifying derived equivalence of algebras \cite[see Theorem 6.4]{Rick1}.

\begin{thm}\label{der_eq} Two algebras $\La$ and $\La'$ are derived equivalent if and only if there exists a 
tilting complex $T\in\cK^b_\La$ such that $\End_{\cK^b_\La}(T)\cong \La'$.   
\end{thm} \medskip 

We recall also the following two theorems (see \cite[Corollary 5.3]{Rick3} and \cite[Theorem 2.9]{ES08}) 
showing that symmetricity and periodicity are also preserved under derived equivalences. 

\begin{thm}\label{der_sym} Let $\La$ and $\La'$ be derived equivalent algebras. Then 
$\La$ is symmetric if and only if $\La'$ is symmetric. \end{thm} \medskip 

\begin{thm}\label{der_per} Let $\La$ and $\La'$ be derived equivalent algebras. Then $\La$ is periodic if and only 
if $\La'$ is periodic. Moreover, if this is the case, then both have the same period. \end{thm} \medskip 

Representation type of two derived equivalent algebras is a bit more subtle issue, it is not true in general, 
that it is preserved. However, in the class of self-injective algebras, we have the following result, 
whose short proof can be found in \cite[see Theorem 2.4]{VM}.  

\begin{thm}\label{der_type} Let $\La$ and $\La'$ be derived equivalent self-injective algebras. Then 
$\La$ is tame if and only if $\La'$ is tame. \end{thm} 

This can be extended also to finite representation type, because finite type is preserved by stable 
equivalences of Morita type.  

\begin{rem}\label{der_inf} \normalfont Without loss of generality, we can assume that all tame symmetric 
algebras under consideration are of infinite representation type. It is a consequence of the assumption on the size of 
quivers defining algebras. Indeed, all quivers have $n=|Q_0|\geqslant 3$ vertices. This implies that any such 
algebra must be representation-infinite, since the symmetric algebras of finite representation type satisfy 
$n\leqslant 2$, due to \cite[see Lemma 3.1]{E26}. \end{rem} 
\medskip  

In this paper, we will consider a special kind of derived equivalence, induced from silting mutation. 
Recall that the {\it mutation} of a symmetric algebra $\La=KQ/I=\oplus_{i\in Q_0} P_i$ {\it at vertex} 
$i\in Q_0$ is the following algebra 
$$\mu_i(\La):=\End_{\cK^b_\La}(T),$$ 
where $T=\oplus_{j\in Q_0} T_j$ is a complex in $\cK^b_\La$ with summands $T_j=(P_j)$ concentrated in 
degree $0$, for $j\neq i$, and 
$$T_i=(\xymatrix{P_{i} \ar[r]^{f} & P^-_i})$$ 
is concentrated in degrees $-1$ and $0$ with $f=[\alpha_1\ \cdots \alpha_s]^T$ defined by the arrows in  
$i^-=\{\alpha_1,\dots,\alpha_s\}$. In other words, $f$ is a left $\add Q$-approximation 
of $P_i$, where $Q$ is a projective module such that $\La=P_i\oplus Q$, and $T$ is a silting mutation of $\La$ 
with respect to the indecomposable direct summand $P_i$ (cf. \cite{AI}; see also \cite[]{mut}). It is known from 
general theory that $T$ is a silting complex, because $\La$ is silting, and in our setup ($\La$ is symmetric), 
silting complexes coincide with tilting complexes. Therefore, the mutation $\mu_i(\La)$ is derived equivalent with 
$\La$, by Theorem \ref{der_eq}, and moreover, $\mu_i(\La)$ is again symmetric and periodic, if $\La$ is (of 
the same period, if this is the case), hence mutations of TSP4 algebras are again TSP4. This property will be 
crucial in the proof of the Main Theorem. Namely, it is used in Section \ref{sec:algstr}, where we prove that the only 
possible TSP4 algebra structures on the Gabriel quivers computed in Section \ref{sec:6}, are the generalized 
weighted surface algebras. \smallskip 

Besides, we will also need the following property \cite[Corollaries 1 and 4]{mut}. \medskip 

\begin{thm}\label{mut:per} Let $\La$ be a tame symmetric algebra, and $i$ a vertex in $Q=Q_\La$ such that $S_i$ 
is $4$-periodic. If $Q$ has no loop at $i$, then $\mu_i^2(\La)$ and $\La$ are socle equivalent. In case $\La$ is 
a weighted surface algebra, socle equivalence can be replaced by isomorphism. \end{thm} \smallskip 

The idea of using the above theorem is as follows. For a fixed quiver $Q$ we want to classify 
all TSP4 algebras $\La=KQ/I$ given by $Q=Q_\La$. Take any $\La=KQ/I$ which is TSP4, and let $i\in Q_0$ be a vertex 
without loop, so that the mutation $\La'=\mu_i(\La)$ is again TSP4 and $\La\cong \mu_i^2(\La)=\mu_i(\La')$ is 
a mutation of $\La'$. We will use appropriate vertex $i$ to get $\La'=KQ'/I'$ with $Q'=Q_{\La'}$ being 
$2$-regular or biregular, where the classifications are known (see Theorem \ref{class:bireg}). Then we know all TSP4 
structures on $Q'$, and we can `lift' them to TSP4 structures on $Q$ via mutation, modulo socle equivalence (see Section \ref{sec:algstr} for details). 

\section{Algebras of generalized quaternion type}\label{sec:3} 

In this section we recall some preparatory results on GQT algebras, needed in the classification. 

\subsection{General properties}\label{subs:3.1} Fix a GQT algebra $\La=KQ/I$. We have the following consequences of infinite type 
\cite[see Lemmas 2.1 and 2.4]{EHS1}. 

\begin{lemma}\label{lem:3.1}  For any vertex $i\in Q_0$, we have $\hat{p}_i\neq p_i$, where $\hat{p}_i$ denotes 
the vector $p_i^-=p_i^+$ corresponding to vertex $i$. \end{lemma} 

\begin{lemma}\label{lem:3.2} There is no arrow $\alpha: i\to j$ with $i^+ = \{ \alpha\} = j^-$. \end{lemma} 

The next result gives a useful tool for constructing triangles in $Q$ (i.e. cycles of length $3$) induced from 
relations defining $\La=KQ/I$ \cite[see Proposition 4.1]{EHS1}. 

\begin{lemma}\label{lem:3.3} Assume $\alpha: i\to j$ and $\beta: j\to k$ are arrows such that $\alpha\beta \prec I$. 
Then there is an arrow in $Q$ from $k$ to $i$, so that  $\alpha$ and $\beta$ are part of a triangle in $Q$. \end{lemma} \medskip 

The following result from \cite[see Lemma 4.3]{EHS1} shows how relations propagate in triangles. 

\begin{lemma}\label{lem:3.4} Assume $Q$ contains  a triangle
        
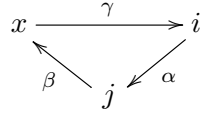
\[
 \xymatrix@R=1.pc@C=1.8pc{
%  \xymatrix@C=.8pc{
    x
    \ar[rr]^{\gamma}
    && i
    \ar@<.35ex>[ld]^{\alpha}
    \\
    & j
    \ar@<.35ex>[lu]^{\beta}
  }
\]
with $\alpha\beta\prec I$. If $\gamma$ is the unique arrow $x\to i$, then $\gamma\alpha\prec I$ and 
$\beta\gamma\prec I$. If we have double arrows $\gamma,\bar{\gamma}: x\to i$, then there is one  
$\delta\in\{\gamma,\bar{\gamma}\}$ such that $\delta\alpha\prec I$ and $\beta\delta\prec I$.

\end{lemma} \smallskip 

Note that, though some of the above results may be formulated in the original paper \cite{EHS1} in 
a bit different way (for biserial algebras), their proofs can be easily rewritten for arbitrary GQT algebra. \medskip 

In biregular case, we know more about position of $1$-vertices in $Q$. Namely, we have the following 
result proved in a separate paper \cite{EHS2}. 

\begin{thm}\label{bireg1} If the Gabriel quiver $Q=Q_\La$ is biregular, then every 
$1$-regular vertex $i$ in $Q$ is a vertex $\bullet$ in a block of the form 
$$
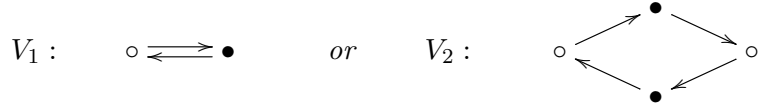
\xymatrix@R=.7ex{&&& & \bullet\ar[rd] & \\ 
V_1: \qquad \circ\ar@<+.4ex>[r]&\bullet\ar@<+.4ex>[l]&\mbox{ or }\qquad V_2:& \circ\ar[ru] && \circ\ar[ld]  \\ 
&&& &\ar[lu]\bullet&}$$
\end{thm}

\subsection{Weighted surface algebras and generalizations}\label{subs:3.2} Now, let us briefly recall the construction of weighted 
surface algebras, known class of examples of TSP4 (or GQT) algebras. By a {\it triangulation 
quiver}, we mean a ($2$-regular) quiver $Q$ which is a glueing of finite number of blocks of types I-III 
(see the Introduction); for an equivalent definition we refer to \cite[Definition 4.2]{WSA}. Then the set of 
arrows of $Q$ admits a permutation $f:Q_1\to Q_1$ which fixes a loop in each block of type I, and otherwise, 
$f$ has an orbit of the form $(\alpha \ \beta \ \gamma)$ (in the notation from the Introduction). In particular, 
$f^3=1$ and $Q$ is $2$-regular, so we have an involution 
$\overline{(-)}:Q_1\to Q_1$, sending an arrow $\alpha$ to the second arrow $\ba\neq \alpha$ starting at the source 
$s(\alpha)=s(\ba)$. In particular, one can consider the another permutation $g:Q_1\to Q_1$ given as 
$g(\alpha)=\overline{f(\alpha)}$. For any collection of integers $m_\alpha$ and parameters $c_\alpha\in K\setminus\{0\}$, 
$\alpha\in Q_1$, which are constant on $g$-orbits, we define paths $A_{\alpha}:=\alpha g(\alpha)\cdots 
g^{m_\alpha n_\alpha-2}(\alpha)$, where $n_\alpha$ is the length of the $g$-orbit of $\alpha$, and then the 
{\it weighted surface algebra} (WSA) \cite{WSA} is a quotient $KQ/I$, where $I$ is generated by the following 
two types of relations 
\begin{enumerate}
\item[(Q)] $\alpha f(\alpha)-c_{\ba}A_{\ba}$, for all arrows $\alpha\in Q_1$, and 
\item[(Z)] $\alpha f(\alpha) g(f(\alpha))$ and $\alpha g(\alpha) f(g(\alpha))$, for some arrows $\alpha\in Q_1$.  
\end{enumerate}  

For more details, see papers \cite{WSA,WSAG,WSAC}. We omit technical conditions involved in the definition 
of relations of type (Z), since in this paper we will not deal with relations directly. Instead, we will use 
known classification results (see Theorem \ref{class:bireg} below), to get a structure of WSA for a mutation $\La'$ 
of $\La$. Note that WSA's are TSP4 (GQT) algebras, except few cases, the so called {\it exceptional algebras}. There 
are four families of the exceptional algebras, called singular disc, triangle, spherical or tetrahedral algebras 
\cite[see Section 3]{WSAG}. \medskip 

Let us only mention the notion of virtual arrows, which can be helpful in understanding furhter 
combinatorial details. For any WSA $\La$, we assume that $m_\alpha n_\alpha\geqslant 2$, for all 
arrows $\alpha\in Q_1$, since this is necessary to get the path $A_\alpha$ of length $\geqslant 1$. 
One can see that arrows $\alpha$, for which $m_\alpha n_\alpha=2$ are involved in the relations of 
type (Q), i.e. $\ba f(\ba)-c_\alpha\alpha\in I$, since then $A_\alpha=\alpha$. It means that $\alpha=c_\alpha^{-1} 
\ba f(\ba)\in J^2$, $J=J_\La$, so $\alpha$ is not an arrow of the Gabriel quiver of $\La$. Hence the 
Gabriel quiver of $\La$ consists of arrows $\alpha\in Q_1$ with $m_\alpha n_\alpha\geqslant 3$. The 
remaining arrows of $Q$ are called the {\it virtual arrows}. \smallskip 

Note that an arrow $\alpha$ of a triangulation quiver $Q$ can be virtual, only in two cases: if 
$n_\alpha=1$ and $m_\alpha=2$, or if $n_\alpha=2$ and $m_\alpha=1$. These correspond to a 
$g$-orbit $(\alpha)$ of length $n_\alpha=1$, which can happen only when $\alpha$ is a loop in a 
block of type II (of weight $2$), or a $g$-orbit of length $n_\alpha=2$ (of weight $1$), which appears 
only when $\alpha$ lies in a (dotted) $2$-cycle obtained from glueing of two triangles as follows.     
$$
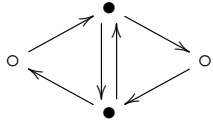
\xymatrix@R0.3cm{ & \bullet \ar[rd] \ar@{.>}@<-0.4ex>[dd] & \\ 
\circ \ar[ru]   && \circ \ar[ld]\\ 
& \bullet \ar@{.>}@<-0.4ex>[uu] \ar[lu] & }$$
Concluding, after removing the virtual loops and $2$-cycles, the Gabriel quiver of a WSA is a 
glueing of a finite number of blocks of types I-III and blocks of types $V_1,V_2$, depicted in 
Theorem \ref{bireg1}. \medskip 

We mention the following classification, which summarizes the results of \cite[Main Theorem]{AGQT} and 
\cite[Main Theorem]{EHS3}. This shows that WSA's exhaust almost all TSP4 algebras with biregular 
Gabriel quiver. The remaining algebras form two exotic families of TSP4 algebras, the so called {\it higher 
tetrahedral algebras} \cite{HTA} and {\it higher spherical algebras} \cite{HSA}, which are not WSA's but their 
Gabriel quivers are the same as Gabriel quivers of WSA's.  

\begin{thm}\label{class:bireg} Let $\La$ be a TSP4 (equivalently, GQT) algebra, whose Gabriel quiver 
has at least three vertices. Then the following conditions hold. 
\begin{enumerate}
\item[(1)] If $Q_\La$ is $2$-regular, then $\La$ is isomorphic to a (non-exceptional) WSA or to a 
higher tetrahedral algebra. 
\item[(2)] If $Q_\La$ is biregular, but not $2$-regular, then $\La$ is isomorphic to a (non-exceptional) 
WSA or to a higher spherical algebra. 
\end{enumerate} \end{thm}  

Note that in the first case, the quiver $Q_\La$ is a triangulation quiver without virtual arrows, and in the 
second, it contains at least one virtual arrow. Moreover, the exotic families consist of algebras with 
$n=6$ vertices, so these will not appear in this paper. \medskip 

Let $\La=KQ/I$ be a weighted surface algebra, and assume $Q$ has virtual arrows. If $i$ 
is a $1$-vertex in $Q_\La$ lying in a block of type $V_2$, then the mutation $\mu_i(\La)$ 
is the so called {\it virtual mutation of} $\La$ (VM, for short). The Gabriel quiver of the virtual mutation 
$\mu_i(\La)$ is obtained from $Q_\La$ by replacing a block of type $V_2$ containing $i$, by 
a block of type IV (see the Introduction). \smallskip 

Virtual mutations of weighted 
surface algebras were introduced and studied in \cite{VM} (including presentation details), where we considered 
the general case, allowing several mutations at vertices in different blocks of type $V_2$. It has been proven 
that virtual mutations are given by quivers and relations similar to weighted surface algebra. 
Namely, every virtual mutation $A$ has the form $A=KQ/I$, where $Q$ is a glueing of a finite number 
of blocks of types I-IV, and $I$ is generated by two types of relations, analogous to (Q) and (Z). 
In case of virtual mutations, we also have two permutations $f$ and $g=\bar{f}$, but defined for `most' 
of the arrows (remaining arrows appear in additional zero relations in blocks of type IV). For technical 
details we refer to \cite[see Section 4]{VM}. \smallskip 

Similarily as for WSA's, in case of virtual mutations we can have virtual arrows in $Q$, which are not 
in the Gabriel quiver. Namely, if $\La=KQ/I$ is a virtual mutation, then we can have virtual arrows 
of two types (loops and $2$-cycles) as for WSA, and two other types of virtual arrows forming (dotted) 
$2$-cycles contained in one of the following blocks of $Q$.  
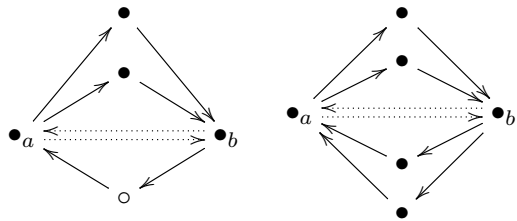
$$\xymatrix@R=0.4cm{ &\bullet \ar[rdd] & \\ & \bullet\ar[rd]& \\ 
\bullet_a \ar[ru] \ar[ruu] \ar@{.>}@<-0.35ex>[rr]&& \bullet_b \ar[ld] \ar@{.>}@<-0.35ex>[ll]\\ & \circ \ar[lu] & } \quad 
\xymatrix@R=0.2cm{\\ \\ \mbox{ or } \\ \\ }\quad \xymatrix@R=0.2cm{& \bullet \ar[rdd] & \\ & \bullet \ar[rd] & \\ 
\bullet_a \ar[ru]\ar[ruu] \ar@{.>}@<-0.35ex>[rr] && \ar[ld]\ar[ldd] \ar@{.>}@<-0.35ex>[ll] \bullet_b \\ 
& \bullet \ar[lu] & \\ & \ar[luu] \bullet & }$$ 
This is a glueing of a triangle with block of type IV, in the first case, or two blocks of type IV, in the second. 
Note that the second case may be excluded from our considerations, since then the block is the whole 
$Q$, and it has at least $n\geqslant 6$ vertices. After removing the virtual arrows, we obtain the following two 
types blocks in $Q_\La$. \medskip 
$$\xymatrix@R=0.2cm{\\ \\ V_3: \\ \\ }\quad \xymatrix@R=0.4cm{ &\bullet \ar[rdd] & \\ & \bullet\ar[rd]& \\ 
\bullet_a \ar[ru] \ar[ruu] && \bullet_b \ar[ld] \\ & \circ \ar[lu] & } \quad 
\xymatrix@R=0.2cm{\\ \\ \mbox{ or } \\ \\ }\quad \xymatrix@R=0.2cm{\\ \\ V_4: \\ \\ }\quad \xymatrix@R=0.2cm{& \bullet \ar[rdd] & \\ & \bullet \ar[rd] & \\ 
\bullet_a \ar[ru]\ar[ruu]  && \ar[ld]\ar[ldd] \bullet_b \\ 
& \bullet \ar[lu] & \\ & \ar[luu] \bullet & }$$ 

Finally, in paper \cite{GWSA}, we discovered the largest known class of TSP4 algebras containing both 
WSA's and VM's. Basically, we considered a virtual mutation $\La=KQ/I$, such that $Q_\La$ is containing 
a block of type $V_3$. We observed that if $i$ is the unique $(1,2)$-vertex in $V_3$, then mutation 
$A=\mu_i(\La)$ induces another block. Namely, the Gabriel quiver $Q_A$ of $A$ is obtained from $Q_\La$ 
by replacing the block $V_3$, by the block of type V (see the Introduction; this is also shown directly 
in the proof of Proposition \ref{prop:6.3}). Such algebras are called the 
{\it generalized weighted surface algebras} (GWSA), and they were described by quivers and relations in \cite{GWSA}.  
In general, we allow arbitrary number of mutations at $(1,2)$-vertices from blocks of type $V_3$, but for 
this paper, it is sufficient to consider a mutation at single vertex. Any GWSA has analogous form $\La=KQ/I$, where 
$Q$ is a glueing of a finite number of blocks of types I-V, and $I$ is generated by two types of relations, 
similar to (Q) and (Z) (and additional zero or commutativity relations in new blocks IV or V).

\subsection{Non-regular vertices in biserial case}\label{subs:3.3}

In the remaining part of this section we will give some technical results concerning position of 
non-regular vertices, mainly in case $Q=Q_\La$ is biserial. First, we have a general fact pertaining 
tame symmetric algebras. 

\begin{lemma}\label{lem:TWD} If $A=KQ/I$ is a tame symmetric algebra, then $Q=Q_A$ does not contain a 
block of the form $\Gamma$ 
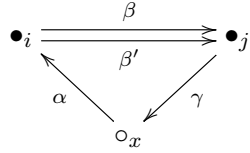
$$\xymatrix{\bullet_i \ar@<+0.44ex>[rr]^{\beta} \ar@<-0.44ex>[rr]_{\beta'} && \bullet_j \ar[ld]^{\gamma}\\ 
& \circ_x \ar[lu]^{\alpha} & }$$ 
with $\gamma\alpha\prec I$. \end{lemma} 

\begin{proof} Suppose to the contrary that $A=KQ/I$ is tame symmetric and $Q_A$ contains a block as in the 
statement. Since $A$ is tame, we must have $\alpha\beta\prec I$ or $\alpha\beta'\prec I$ (otherwise, we get 
a subquiver of type $K_2^*$). Let $\alpha\beta'\prec I$. 
Every path in $Q_A$ ending at $j$ is of the form $\cdots\alpha\beta$ or $\cdots\alpha\beta'$, because $j^-=\{\beta,\beta'\}$ 
and $i^-=\{\alpha\}$. Hence, the path $\alpha\beta'$ is involved in a minimal relation   
$$\alpha\beta'+z_0\alpha\beta+z_1\alpha\beta'=0,$$ 
where $z_1\in J_A$. Premultiplying both sides by the inverse $u^{-1}$ of the unit $u=1+z_1$, we obtain a 
relation of type: $\alpha\beta'+z_0\alpha\beta=0$, $z_0\in A$. Then $\alpha\beta\nprec I$. 
Indeed, otherwise we have a relation of the form $\alpha\beta+z_0'\alpha\beta'=0$, and hence, 
we obtain $\alpha\beta'=-z_0\alpha\beta=z_0z_0'\alpha\beta'=(z_0z_0')^2\alpha\beta'=\dots$. In this case, 
we cannot have $z_0\in J_A$ or $z_0'\in J_A$, because then we would get $\alpha\beta'=0=\alpha\beta$, 
which is impossible for a symmetric algebra (we would get an arrow $\alpha$ in the right socle of $A$). Consequently, 
both $z_0,z_0'$ are units (mutually inverse), so $\alpha\beta+\lambda\alpha\beta'=0$ for some non-zero $\lambda\in K$. 
But then after adjusting arrows $\beta:=\beta$ and $\beta':=\beta+\lambda\beta'$, we obtain a new presentation of $A$ 
with $\alpha\beta'=0$, which leads to a contradiction, since then the arrow $\beta'$ belongs to the 
left socle of $A$. As a result, we may assume $\alpha\beta\nprec I$ and $\alpha\beta'\prec I$. \medskip 

Now, consider the idempotent algebra $\bar{A}=eAe$, where $e=e_i+e_j+e_x$. Then $\bar{A}=K\bar{Q}/\bar{I}$ 
is also a tame symmetric algebra, the block $\Gamma$ remains a block in $\bar{Q}=Q_{\bar{A}}$ and $\bar{Q}$ consists 
of the arrows $\alpha,\beta,\beta',\gamma$ and possibly a loop at $x$. \smallskip 

Observe that there is a loop at $x$ in $\bar{Q}$. Indeed, by the assumption $\gamma\alpha\prec I$, 
so we have a relation of the form $\gamma\alpha+\gamma z_0=0$ with $z_0\in J_A^2$ (every path in $Q$ 
starting from $j$ begins with $\gamma$). Adjusting $\alpha:=\alpha+z_0$, we can assume $\gamma\alpha=0$ in $A$. 
Therefore, we have also $\gamma\alpha=0$ in $\bar{A}$. So there must be a loop at $x$, since otherwise 
there are no non-zero paths from $j$ to $i$, which is impossible for a symmetric algebra $\bar{A}$ 
(we have arrows $i\to j$). As a result, the quiver $\bar{Q}$ is a glueing of the block $\Gamma$ with a loop 
$\sigma$ (at $x$). Because $\alpha$ is the unique arrow in $\bar{Q}$ ending at $i$, using similar 
arguments one can show that $\gamma\sigma\nprec\bar{I}$ and $\sigma\alpha\nprec \bar{I}$ (otherwise, 
adjusting $\sigma$ we get $\gamma\sigma=0$ or $\sigma\alpha=0$, and then $e_j\bar{A}e_i=0$). Moreover, 
we have $\alpha\beta\nprec\bar{I}$, $\alpha\beta'\prec\bar{I}$ and $\alpha\beta'=a\alpha\beta$, for some 
$a\in J_{\bar{A}}$. \smallskip 

Next, note that $\sigma^2\prec \bar{I}$, because otherwise we have the following wild subcategory in covering. 
$$\xymatrix@R=0.5cm{&& j &&& \\ && i\ar[u]^{\beta} &&& \\ 
j\ar[r]^{\gamma} & x & \ar[l]_{\sigma} x \ar[u]^{\alpha} & \ar[l]_{\sigma} x \ar[r]^{\alpha} & i\ar[r]^{\beta} & j}$$
Indeed, all paths of length $2$ are not involved in minimal relations of $\bar{I}$, hence it remains to 
see that $\sigma\alpha\beta\nprec\bar{I}$. If this was not the case, we would obtain a relation 
of the form $\sigma\alpha\beta+ z\alpha\beta=0$ (in $\bar{A}$), with $z\in J^2_{\bar{A}}$, because 
every path ending with $\alpha\beta'\prec \bar{I}$ can be replaced by a path ending with $\alpha\beta$. 
Now, we can adjust $\sigma:=\sigma+z$, and then $\sigma\alpha\beta=0$ in $\bar{A}$, which contradicts 
symmetricity of $\bar{A}$, since then the path $\alpha\beta$ belongs to the left socle of $\bar{A}$. \smallskip 

Finally, we will show another contradiction with symmetricity. Indeed, it is easy to check that 
$\sigma^2\prec\bar{I}$ and other relations in $\bar{I}$ imply that each radical quotient 
$e_j\bar{J}^k/\bar{J}^{k+1}$, where $\bar{J}=J_{\bar{A}}$, is spanned by the coset of the path $[W_\gamma]_k$, 
which is an initial submonomial of length $k$ of the following cycle 
$$W_\gamma=(\gamma\sigma\alpha\beta)^m,$$ 
$m\geqslant 1$, generating the socle of $e_j\bar{A}$. The independence of these paths is clear, so $e_j\bar{A}$ 
has a basis consisting of initial submonomials of $W_\gamma$. Moreover, the socle of $e_i\bar{A}$ is spanned by the rotation $W_\beta=(\beta\gamma\sigma\alpha)^m$ 
of $W_\gamma$. One can similarly show that $e_i\bar{A}$ has a basis consisting of initial submonomials of 
$W_\beta$ together with at least one path starting with $\beta'$. But then the basis of $e_j\bar{A}e_i$ is given by paths 
$(\gamma\sigma\alpha\beta)^k\gamma\sigma\alpha$, for $k\in\{0,\dots,m-1\}$, whereas the basis of $e_i\bar{A}e_j$, by 
paths $(\beta\gamma\sigma\alpha)^k\beta$, for $k\in\{0,\dots,m-1\}$, together with at least $\beta'$. Consequently, we have 
$\dim_K e_i\bar{A}e_j>m=\dim_K e_j\bar{A}e_i$, which is not possible for symmetric algebras. This finishes 
the proof. \end{proof} \medskip 

\begin{cor}\label{12wD} If $\La=KQ/I$ is a GQT algebra, then $Q=Q_\La$ does not contain a $(1,2)$-vertex $i$ such 
that $i^+$ consists of double arrows. Dually, we have no $(2,1)$-vertices $i$ with $i^-$ consisting of double 
arrows. \end{cor}

\begin{proof} Assume $\La$ is a GQT algebra and let $i$ be a $(1,2)$-vertex with $i^-=\{\alpha:x\to i\}$ and 
double arrows in $i^+=\{\beta,\bar{\beta}:i\to j\}$. Since algebra is tame (no type $K_2^{*}$ subquiver), we 
must have $\alpha\beta\prec I$ or $\alpha\bar{\beta}\prec I$, hence by Lemma \ref{lem:3.3}, there is an arrow 
$\gamma:j\to x$. We claim that $j$ is a $(2,1)$-vertex. Indeed, if this is not the case, then there is an arrow 
$\bar{\gamma}:j\to j'$ different from $\gamma$, and moreover, we get $j'=x$, because otherwise both 
$\beta\bar{\gamma},\bar{\beta}\bar{\gamma}\nprec I$, due to Lemma \ref{lem:3.3}, and then $\La$ is wild (has a 
type $K_2^{*}$ subquiver). So we have another double arrows $\gamma,\bar{\gamma}:j\to x$. Now, if $x$ 
is a $(2,1)$-vertex, then $x^+=\{\alpha\}=i^-$, which contradicts Lemma \ref{lem:3.2}. Thus $|x^+|\geqslant 2$ and 
there is an arrow $\ba:x\to x' $ with $x'\neq i$, since $|i^-|=1$. Consequently, there is no arrow $x'\to j$, 
so applying Lemma \ref{lem:3.3} again, we conclude that $\gamma\ba,\bar{\gamma}\ba\nprec I$, which gives a type $K_2^{*}$ 
subquiver. As a result, we proved that $j$ is a $(2,1)$-vertex, and hence 
$\alpha,\beta,\bar{\beta},\gamma$ exhaust all arrows in $Q$ that start or end at vertices $i,j$. But then passing 
to idempotent algebra $A=e\La e$, where $e=e_i+e_j+e_x$, we get a tame symmetric algebra $A=KQ_A/I_A$ such that 
arrows $\alpha,\beta,\bar{\beta},\gamma$ remain arrows in $Q_A$, and they also exhaust all arrows in $Q_A$, which 
start or end at $i,j$. It follows that $Q_A$ contains a block forbidden by the previous lemma, and we are done 
(note: $\alpha\beta\prec I$, so $\gamma\alpha\prec I$, by Lemma \ref{lem:3.4}, and hence also $\gamma\alpha\prec I_A$).  
\end{proof}

\medskip 

Now, we will prove two necessary results concerning non-regular vertices in biserial case (Propositions 
\ref{prop:typeN} and \ref{prop:biser}). If the Gabriel quiver of $\La$ is biserial, but not biregular, we need to 
understand at least general facts about position of non-regular vertices, i.e. $(1,2)$- or $(2,1)$-vertices. \smallskip 

Assume that $i$ is a $(1,2)$-vertex 
$$\xymatrix@R=0.35cm{&&j\\ x\ar[r]^{\alpha} & i \ar[ru]^{\beta}\ar[rd]_{\bar{\beta}} & \\ &&k}
\leqno{\xymatrix@R=0.35cm{\\(*)\\}}$$

It follows from the previous corollary that always $j\neq k$. We call $i$ a {\it vertex of type R} (respectively, 
{\it of type N}), provided that $\alpha\beta,\alpha\bar{\beta}\prec I$ (respectively, 
$\alpha\beta,\alpha\bar{\beta}\nprec I$). Recall an immediate observation from \cite[see Lemma 5.2]{EHS1}. 

\begin{lemma}\label{typeR} If $Q=Q_\La$ is biserial, then there are no non-regular vertices of type R. \end{lemma} 

\medskip 

We will need the following elementary observation. 

\begin{lemma}\label{lem:exseq} Let $\La$ be a self-injective algebra and assume that there is an exact 
sequence 
$$\xymatrix{0 \ar[r] & Ker(\pi) \ar[r] & P \ar[r]^{\pi} & M \ar[r] & 0}$$
with $P$ a projective module and $M$ indecomposable non-projective. Then $\soc(P)\subset Ker(\pi)$. 
\end{lemma} 

\begin{proof} It is sufficient to prove that every simple submodule $S$ of $P$ is contained in $Ker(\pi)$. 
Suppose to the contrary that there is a simple module $S\subset P$ with $\pi(S)\neq 0$. Then there exists 
a decomposition $P=P_1\oplus P_2$ such that $P_1$ indecomposable with $\soc(P_1)=S$. It follows that the induced 
map $\pi_1=\pi_{|P_1}:P_1\to M$ is a monomorphism. Indeed, if $Ker(\pi_1)\neq 0$, then it contains at 
least one simple submodule, but $S=\soc(P_1)$ is the unique simple submodule of $P_1$, and we would get $\pi(S)=0$, 
a contradiction. As a result, $\pi_1$ is a split mono, so $M\simeq P_1$, because $M$ is indecomposable, 
and therefore $M\in\proj\La$, which contradicts the assumptions. \end{proof} 

\medskip 

\begin{prop}\label{prop:typeN} If $Q=Q_\La$ is biserial, then every non-regular vertex is of type N. \end{prop} 

\begin{proof} Fix a non-regular vertex, say $i$ is a $(1,2)$-vertex with arrows denoted as in $(*)$ above. We know that 
$j\neq k$ and $i$, by Lemma \ref{12wD}, and $i$ not of type R, due to Lemma \ref{typeR}. This means that we 
have at most one of $\alpha\beta,\alpha\bar{\beta}$ involved in a minimal relation of $I$. We shall prove that there is 
no such a path, i.e $i$ is of type N. \medskip 

Suppose to the contrary that $i$ is not of type N, say $\alpha\beta\prec I$ but $\alpha\bar{\beta}\nprec I$. 
We will show that this leads to a contradiction in two steps. We have $\alpha\beta\prec I$, hence by Lemma \ref{lem:3.3}, 
there is a triangle $(\alpha \ \beta \ \gamma)$ in $Q=Q_\La$, and we can take $\gamma$ such that $\beta\gamma\prec I$ and 
$\gamma\alpha\prec I$ (see Lemma \ref{lem:3.4}). \smallskip 

First, note that any symmetric algebra $A$ satisfy the following property. \begin{enumerate}
\item[(*)] {\it There is no arrow $\sigma:a\to b$ in $Q_A$ with $a^+=\{\sigma\}=b^-$ and $\sigma\delta\prec I_A$, 
for some arrow $\delta\in b^+$.} \end{enumerate}
Indeed, if $\sigma$ is the unique arrow starting from $a$, then $\sigma\delta\prec I$ forces $\sigma\delta=0$, 
after adjusting arrow $\delta:b\to c$, so we get an arrow $\delta$ in the left socle of $A$, since $\sigma$ is also 
the unique arrow in $Q_A$ ending at $b$, which is impossible for symmetric $A$. \medskip 

Consider now the idempotent algebra $A=e\La e$, where $e=e_x+e_i+e_j$. If $A=KQ_A/I_A$, then clearly $Q_A$ contains 
arrows $\alpha:x\to i$, $\beta:i\to j$ and $\gamma:j\to x$, and moreover, we have $\gamma\alpha=0$ in $A$, 
because $i^-=\{\alpha\}$ in $Q_A$ (possibly adjusting $\gamma$; see also Lemma \ref{lem:TWD}). As a result, we also have 
$\alpha\beta\prec I_A$ and $\beta\gamma\prec I_A$. \medskip  

Now, applying \cite[Lemma IV.2.4]{E90}, we deduce that exactly one of the following conditions hold. 
\begin{enumerate}
\item[A)] $\beta$ is the unique arrow in $Q_A$ starting from $i$, and then $Q_A$ must be one of the 
following quivers 
$$
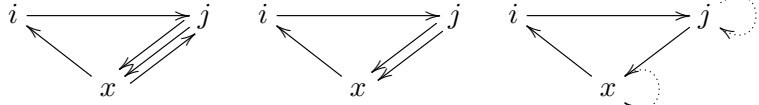
\xymatrix@R=.5cm{ i \ar[rr]&& j\ar@<-.7ex>[ld]\ar[ld] \\ 
& \ar[lu] x \ar@<-.7ex>[ru] &} \quad 
\xymatrix@R=.5cm{ i \ar[rr]&& j\ar@<-.4ex>[ld]\ar@<.4ex>[ld] \\ 
& \ar[lu] x &} \quad 
\xymatrix@R=.5cm{ i \ar[rr]&& j\ar[ld] \ar@(ru,rd)@{.>}[] \\ 
& \ar[lu] x \ar@(ru,rd)@{.>}[] &}$$ 

\item[B)] There exists an arrow $\beta':i\to j$ in $Q_A$, different from $\beta$, and then $Q_A$ is one 
of the following quivers 
$$
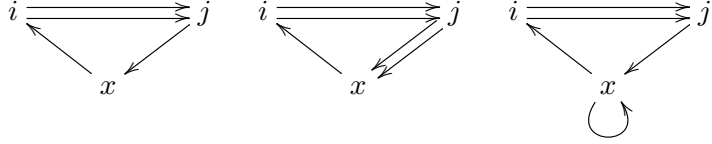
\xymatrix@R=.5cm{i \ar@<+0.44ex>[rr] \ar@<-0.44ex>[rr] && j \ar[ld]\\ 
& x \ar[lu] & } \quad 
\xymatrix@R=.5cm{i \ar@<+0.44ex>[rr] \ar@<-0.44ex>[rr] && j \ar@<-.4ex>[ld]\ar@<.4ex>[ld]\\ 
& x \ar[lu] & } \quad 
\xymatrix@R=.5cm{i \ar@<+0.44ex>[rr] \ar@<-0.44ex>[rr] && j \ar[ld]\\ 
& x \ar[lu] \ar@(dl,dr)[] & } \vspace*{.5cm}$$ 

\item[C)] There is an arrow $\beta':i\to x$ in $Q_A$, and then it is one of the quivers 
$$
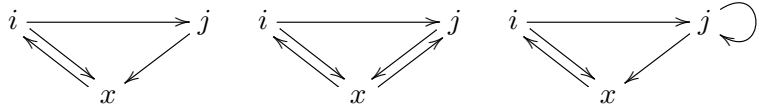
\xymatrix@R=.5cm{ i \ar[rr]\ar@<.4ex>[rd] && j\ar[ld] \\ 
& \ar@<.4ex>[lu] x &} \quad 
\xymatrix@R=.5cm{ i \ar[rr]\ar@<.4ex>[rd] && j\ar@<-.4ex>[ld] \\ 
& \ar@<.4ex>[lu] x \ar@<-.4ex>[ru] &} \quad 
\xymatrix@R=.5cm{ i \ar[rr]\ar@<.4ex>[rd] && j\ar[ld] \ar@(ru,rd)[]\\ 
& \ar@<.4ex>[lu] x &}$$
\end{enumerate}

The second quiver in A) and the first and third quiver in B) can be immediately excluded, by Lemma \ref{lem:TWD}. 
The second quiver in B) is not allowed, due to the property $(*)$, since then $\alpha\beta\prec I$, for an arrow 
$\alpha:x\to i$ with $\{\alpha\}=x^+=i^-$. Moreover, one can also show that $Q_A$ cannot be any of the quivers from 
C). Indeed, if this 
was the case, then any path in $Q_A$ from $j$ to $i$ can be written as $u\gamma\alpha$ (in $KQ_A$), for some path 
$u\in KQ_A$. But $\gamma\alpha=0$ in $A$, hence $e_jAe_i=0$, which is impossible for a symmetric algebra, because we 
have an arrow $\beta:i\to j$. \medskip 

Summing up, we have proven that $Q_A$ is the first or the third quiver in A). In both cases, $\beta$ is the unique 
arrow in $Q_A$ starting at $i$.  

\medskip 

Next, $i$ is a $(1,2)$-vertex in $Q$, hence the exact sequence for $S_i$ in $\mod\La$ has the form: 
$$\xymatrix{0 \ar[r] & S_i \ar[r] & P_i \ar[r]^{\alpha} & P_x\ar[r]^(.4){d_2} & P_j\oplus P_k 
\ar[r]^(.6){[\beta \ \bar{\beta}]} & P_i\ar[r] & S_i \ar[r] & 0}.$$ 
Since $\gamma\alpha=0$ in $\La$, we can take $\gamma$ as the first row of $d_2=\vec{\gamma \\ \delta}$ 
(see \cite[Lemma 3.2(i)]{EHS1}). Moreover, we have $\delta\in J^2$, since otherwise $\delta$ is an arrow $k\to x$, 
and then $[\beta \ \bar{\beta}]\cdot d_2=0$, forces $\bar{\beta}\delta\prec I$, so we have a triangle 
$(\alpha \ \bar{\beta} \ \delta)$ with $\bar{\beta}\delta\prec I$ but $\alpha\bar{\beta}\nprec I$, 
a contradiction with Lemma \ref{lem:3.4}. \medskip 

Now, recall that $\beta$ is the unique arrow starting from $i$ in the quiver $Q_A$ of the 
idempotent algebra $A$. Consequently, every path in $e_i\La e_x$ belongs 
to $\beta\La$ (modulo $I_A$). For example, we can write $\bar{\beta}\delta=\beta z$ in $A$, for some 
$z\in J$, and then we have $\beta\gamma+\bar{\beta}\delta=\beta(\gamma+z)=0$ also in $\La$. Hence, we get 
$\beta\gamma=0$, after adjusting $\gamma:=\gamma+z$, if $z\in J^2$; otherwise, $z\in K\gamma$). \smallskip 

Finally, note that $\beta\gamma=0$ implies $[\beta \ \bar{\beta}]\cdot\vec{\gamma \\ 0}=0$, and 
the column $\vec{\gamma \\ 0}$ is not in $J^2$, hence using \cite[Lemma 3.2(ii)]{EHS1}, we conclude 
that $d_2$ can be taken as $d_2=\vec{\gamma \\ 0}$. In particular, there is an exact sequence in 
$\mod\La$ of the form 
$$\xymatrix{0\ar[r] & K\ar[r] & P_j\oplus P_k\ar[r]^{\pi} & M\ar[r] & 0}$$ 
where $\pi=d_1=[\beta \ \bar{\beta}]$, $K=Ker(\pi)=Im(\vec{\gamma \\ 0})=\gamma\La\oplus 0$ and 
$M=Im(d_1)=\rad P_i$ is indecomposable non-projective. But $K\subset P_j$ does not contain the socle 
$\soc(P_j\oplus P_k)=S_j\oplus S_k$, hence we get a contradiction with Lemma \ref{lem:exseq}. \end{proof} 

\medskip 

\begin{lemma}\label{precnonreg} If $Q=Q_\La$ is biserial, then every successor of a $(1,2)$-vertex 
is not a $(1,2)$-vertex. Dually, predecessors of $(2,1)$-vertices are not $(2,1)$-vertices. \end{lemma} 

\begin{proof} By Proposition \ref{prop:typeN}, we may assume that each non-regular vertex $i$ is 
of type N (without double arrows starting or ending from $i$). Let $i$ be a fixed $(1,2)$-vertex 
with $i^+=\{\beta,\bar{\beta}\}$ and two successors $j=t(\beta)$, $k=t(\bar{\beta})$, and $i^-=\{\alpha\}$ 
with one predecessor $x=s(\alpha)$. Using Lemma \ref{lem:3.2} ($\La$ is of infinite type), we deduce 
that $|x^+|=2$, so there is an arrow $\ba:x\to u$, $\ba\neq\alpha$. \medskip 

Suppose to the contrary that one of $j,k$ is again a $(1,2)$-vertex. Without loss of generality, suppose 
$j$ is a $(1,2)$-vertex with $j^+=\{\gamma,\bar{\gamma}\}$ and successors $a=t(\gamma)$ and $b=t(\bar{\gamma})$. 
Note that $j^-=\{\beta\}$ and $i=s(\beta)$ is the unique predecessor of $j$. \smallskip 

1) First, observe that exactly one of the paths $\alpha\beta\gamma,\alpha\beta\bar{\gamma}$ is involved 
in a minimal relation of $I$. Indeed, suppose both $\alpha\beta\gamma,\alpha\beta\bar{\gamma}\nprec I$. 
Then we have the following wild subcategory in covering 
$$\xymatrix@R=0.4cm{&& k & b & \\ 
u & \ar[l]_{\ba} x \ar[r]^{\alpha} & i \ar[r]^{\beta} \ar[u]^{\bar{\beta}} 
& j \ar[r]^{\gamma} \ar[u]^{\bar{\gamma}} & a }$$ \smallskip 

On the other hand, if both $\alpha\beta\gamma,\alpha\beta\bar{\gamma}\prec I$, then applying 
\cite[Proposition 4.5]{EHS1} we conclude that there are arrows $\sigma:a\to x$ and $\sigma':b\to x$ 
(and $x^-=\{\sigma,\sigma'\}$). In this case, we have 
$$p_i+p_u=p_x^+=p_x^-=p_a+p_b=p_j^+=p_j^-=p_i,$$ 
so $p_u=0$, a contradiction. This proves 1). We may assume that $\alpha\beta\gamma\prec I$ but 
$\alpha\beta\bar{\gamma}\nprec I$. In particular, it follows that there is an arrow $\sigma:a\to x$. \medskip 

2) Next, we claim that $|a^-|=2$. Assume that this is not the case. Then either $a$ is a $1$-vertex 
with $a^+=\{\sigma\}$ or $a$ is a $(1,2)$-vertex with $a^+=\{\sigma,\bar{\sigma}\}$. In the first case, 
we have $p_j=p_a^-=p_a^+=p_x$, which is impossible, since $p_x=p_i^-=p_i^+=p_j+p_k$. In the second, 
we get $p_j=p_a^-=p_a^+=p_x+p_{t(\bar{\sigma})}=p_j+p_k+p_{t(\bar{\sigma})}$, a contradiction again. \medskip 

3) Now, we will prove that $|k^-|=2$, or $|k^-|=1$ and then for any $\delta\in k^+$, we have $\bar{\beta}\delta\nprec I$. 
Let $|k^-|=1$. If $|k^+|=1$, i.e. $k$ is a $1$-vertex, then the unique $\delta:k\to k'$ must satisfy 
$\bar{\beta}\delta\nprec I$, because otherwise, there is an arrow $k'\to s(\bar{\beta})=i$, by Lemma \ref{lem:3.3}, 
but then $k'=x$, so $p_i^-=p_x=p_{k'}=p_k^+=p_k^-=p_i$, and we obtain a contradiction with Lemma \ref{lem:3.1}. 
If $|k^+|=2$, then $k$ is a $(1,2)$-vertex, hence of type N, and the claim follows. \medskip 

4) In the last step, we show that $|b^-|=2$, or $|b^-|=1$ and then there is $\ve\in b^+$ such that 
$\bar{\gamma}\ve\nprec I$ and $\beta\bar{\gamma}\ve\nprec I$. Let $|b^-|=1$. If $|b^+|=2$, then $b$ is 
a $(1,2)$-vertex of type N, so $\bar{\gamma}\ve\nprec I$, for any $\ve\in b^+$, and the second part 
follows from 1) applied to vertex $j$. Suppose now $|b^+|=1$. 
Then $b$ is a $1$-vertex and the unique arrow $\ve:b\to b'$ satisfies $\bar{\gamma}\ve\nprec I$, since 
otherwise, we have an arrow $b'\to s(\bar{\gamma})=j$, due to Lemma \ref{lem:3.3}, but then $b'=i$ 
(unique predecessor of $j$), so $b=x$ is a $2$-vertex, which is an absurd. Further, we have 
$\beta\bar{\gamma}\ve\nprec I$. Indeed, if this is not the case, then by \cite[Proposition 4.5]{EHS1}, we get an 
arrow $b'\to s(\beta)=i$, hence $b'=x$, and we obtain $p_x=p_{b'}=p_b^+=p_b^-=p_j$, a contradiction as in 2). \medskip 

Finally, summing up the above conditions 2)-4), we conclude that there exists a wild subcategory of 
the following form 
$$\xymatrix@R=0.4cm{&&& \circ\ar[d] && \\ &&& a && \\ 
k'\ar@{-}[r]^{\delta}&k& \ar[l]_{\bar{\beta}}i\ar[r]^{\beta} & j\ar[u]^{\gamma}\ar[r]^{\bar{\gamma}} & b 
&\ar@{-}[l]_{\ve} b' }$$ 
Note that $\delta$ is either an arrow $\delta:k'\to k$, $\delta\neq\bar{\beta}$, if $|k^-|=2$, or it is any 
$\delta:k\to k'$, otherwise, so we have $\bar{\beta}\delta\nprec I$, due to part 3). The arrow $\ve:b'\to b$ 
or $\ve:b\to b'$ is defined in a similar way, but using 4), and in this case $\bar{\gamma}\ve,\beta\bar{\gamma}\ve\nprec I$, 
if the paths exist. The second arrow ending at $a$ different from $\gamma$ exists by 2). \medskip 

This completes the proof in case of $(1,2)$-vertices. The proof for $(2,1)$-vertices follows from dual 
arguments. \end{proof}

Finally, we prove the following proposition on neighbours of non-regular vertices. 

\begin{prop}\label{prop:biser} Assume $Q=Q_\La$ is biserial and $i$ is a $(1,2)$-vertex with 
$i^+=\{\beta,\bar{\beta}\}$ and $j=t(\beta)$, $k=t(\bar{\beta})$. Then at least one of $j,k$ is a $1$-vertex. 
Dually, at least one of the predecessors of a $(2,1)$-vertex is a $1$-vertex. \end{prop} 

\begin{proof} Assume to the contrary that both $j,k$ are not $1$-regular. Since $j,k$ cannot be $(1,2)$-vertices, 
due to the previous proposition, we obtain that $|j^-|=|k^-|=2$. Let $\beta':v\to j$ and $\beta'':w\to k$ 
be the arrows in $Q$ such that $j^-=\{\beta,\beta'\}$ and $k^-=\{\bar{\beta},\beta''\}$. Moreover, 
we have $i^-=\{\alpha\}$, so using Lemma \ref{lem:3.2}, we conclude that $x^+=\{\alpha,\ba\}$, for 
some arrow $\ba:x\to u$. \smallskip 

1) First, we claim that $|j^+|=|k^+|=2$, i.e. both $j,k$ are $2$-regular. Suppose that this is not the case, 
say $|j^+|=1$, $j^+=\{\gamma:j\to a\}$. Then $j$ is a $(2,1)$-vertex, so by Proposition \ref{precnonreg}, 
its predecessor $v=s(\beta')$ is not a $(2,1)$-vertex. If $v$ is a $1$-vertex, then the unique arrow 
$\sigma:v'\to v$ satisfies $\sigma\beta'\nprec I$, because otherwise (Lemma \ref{lem:3.3}), we get an arrow 
$j\to v'$, so $v'=a$, and we have a triangle $(\beta' \ \gamma \ \sigma)$ with $\sigma\beta'\prec I$ but 
$\beta'\gamma\nprec I$ ($j$ is a non-regular vertex, hence of type N), a contradiction with Lemma \ref{lem:3.4}. 
If $v$ is not a $1$-vertex, then it is a $2$- or $(1,2)$-vertex, so $|v^+|=2$ and then we put 
$\sigma:=\overline{\beta'}:v\to v'$. 
In both cases, we obtain the following wild subquiver 
$$\xymatrix@R=0.5cm{&&u&&& \\ && \ar[u]^{\ba}x\ar[d]_{\alpha} &&& \\ 
w\ar[r]^{\beta''} &k& \ar[l]_{\bar{\beta}}i\ar[r]^{\beta} &j& \ar[l]_{\beta'}v\ar@{-}[r]^{\sigma} &v'}$$ 

Therefore, it has been proven that vertex $j$ is $2$-regular. The same arguments work for $k$. Moreover, we can 
assume that $|v^+|=|w^+|=1$. Let $\gamma:j\to a,\bar{\gamma}:j\to b$ and $\delta:k\to c,\bar{\delta}:k\to d$ be 
the arrows starting from $j$ and $k$, respectively. \medskip 

2) We have both $\beta\gamma,\beta\bar{\gamma}\nprec I$ (and dually 
$\bar{\beta}\delta,\bar{\beta}\bar{\delta}\nprec I$). Indeed, otherwise by Lemma \ref{lem:3.3} we get an arrow 
$a\to i$ or $b\to i$, so $x=a$ or $b$, because $i$ has the unique predecessor $x$. But then we obtain a 
triangle $(\alpha \ \beta \ \rho)$, where $\rho=\gamma$ or $\bar{\gamma}$, such that $\alpha\beta\nprec I$ 
($i$ is of type N) but $\beta\rho\prec I$, which contradicts Lemma \ref{lem:3.4}, and we are done. In other 
words, both paths starting from $\beta$ are not involved in minimal relations of $I$. \medskip 

3) Next, observe that any vertex $z\in\{a,b,c,d\}$ satisfies $|z^-|=1$. In fact, by 2), it is easy to see  
that there is a (tame) hereditary subcategory of the form  
$$\xymatrix@R=0.5cm{&b&&d& \\ 
a & \ar[l]_{\gamma}j\ar[u]^{\bar{\gamma}} & \ar[l]_{\beta}i\ar[r]^{\bar{\beta}} & 
\ar[r]^{\delta}k\ar[u]^{\bar{\delta}} & c}$$ 
Attaching an arrow to any of the leaves $a,b,c,d$ of the above tree, we get a wild subcategory, so the 
claim follows. \medskip 

4) Eventually, observe that vertices $a,b,c,d$ must be pairwise different, because otherwise, there is 
at least one $z\in\{a,b,c,d\}$ with $|z^-|=2$, which is not possible, by 3). We know that $x$ has 
exactly two predecessors, so there is at least one $z\in\{a,b,c,d\}$, which is not a 
predecessor of $x$. Without loss of generality, assume $z=a$ is not a predecessor of $x$. Then using \cite[Proposition 4.5]{EHS1}, we 
deduce that $\alpha\beta\gamma\nprec I$, and therefore, we get a wild subcategory of the form 
$$\xymatrix@R=0.5cm{&&w\ar[d]_{\beta''}&& \\ &&k& v\ar[d]_{\beta'} & \\ 
u& \ar[l]_{\ba}x\ar[r]^{\alpha} & i\ar[u]^{\bar{\beta}}\ar[r]^{\beta} & j\ar[r]^{\gamma} & a}$$
Note that this is isomorphic to a wild hereditary algebra or to a wild one-relation algebra 
\cite[see Theorem 2 in 1.5]{Rin}, depending on $\beta'\gamma\nprec I$ or $\beta'\gamma\prec I$. Similar wild 
algebras can be constructed in the remaining cases $z=b,c$ or $d$. This completes the proof. \end{proof}

\section{Periodicity shadows}\label{sec:4} 

This section gives a quick recap of the notion of periodicity shadow \cite{PS1,PS2}. We will first discuss 
the definition and the so-called Reconstruction Theorem, and then give examples for $n=3,4,5$. We finish with 
some helpful lemmas concerning the position of loops in the Gabriel quivers of GQT algebras. \medskip 

The notion of peridicity shadow was motivated by a few conditions, which have to 
be satisfied by shadows $\bS_\La$ of TSP4 (in general, GQT) algebras $\La$. Following \cite[see Definition 4.2]{PS1}, 
we say that a skew-symmetric matrix $A\in\bM_n(\bZ)$ is a {\it periodicity shadow} if $A$ satisfies the following 
conditions (see also \cite[Theorem 2.2]{PS1}): 
\begin{enumerate}
\item[(PS1)] $A$ is singular, 
\item[(PS2)] $A$ does not admit a non-zero row containing only integers of the same sign,  
\item[(PS3)] there exists a symmetric matrix $C\in\bM_n(\bN)$ with non-zero columns such that $AC=0$. 
\end{enumerate} 
Such a matrix $A$ is called {\it tame}, provided that: 
\begin{enumerate}
\item[(T1)] $A$ has entries $|a_{ij}|\leqslant 2$, 
\item[(T2)] no row (equivalently, column) contains $a_{ij}=2,a_{ik}\geqslant 1$ or  $a_{ij}=-2,a_{ik}\leqslant -1$, 
\item[(T3)] each row (column) contains at most four positive and at most four negative entries. 
\end{enumerate} 

We will often identify a skew-symmetric matrix $A$ with the unique $2$-acyclic quiver $\bfQ_A$, whose 
adjacency matrix is $\Adj_{\bfQ_A}=A$. \smallskip 

The tame periodicity shadows were the main object of interest in the first paper \cite{PS1}, whereas 
in the second part \cite{PS2}, we presented an algorithm allowing us to generate the set 
$\bS(n)$ of all (basic) tame periodicity shadows of size $n$, and we provided successful computations 
for $n\leqslant 6$. By basic, we mean that each tame periodicity shadow is obtained from a shadow $A\in\bS(n)$ 
by permutations of rows and columns, or taking $-A$. In other words, we take a set $\bS(n)$ of tame 
periodicity shadows, which consists of all representatives of orbits under the action of permutation or taking 
negative matrix. At the level of associated quivers, this means permutation of vertices or taking the opposite 
quiver. Actually, we are interested only in the so-called {\it essential shadows} \cite[see Section 2]{PS2}, 
which are shadows $A\in\bS(n)$ satisfying two additional conditions: 
\begin{enumerate}
\item[(PS4)] each row of $A$ does not contain both $2$ and $-2$ (except $A$ is the Markov shadow 
$A=\bS_1\in\bS(3)$ from Example \ref{exm:4.1}), 
\item[(PS5)] for any $i,j,$ and $k$ such that $a_{ij}=2$ and $a_{jk}=1$, we have $a_{ki}>0$; for any $a_{ij}=-2$ and 
$a_{jk}=-1$, we have $a_{ki}<0$. 
\end{enumerate} 

Note that for any GQT algebra $\La$ its shadow $\bS_\La$ is essential \cite[see Section 2]{PS2}. As a result, computing 
the set $\bE(n)$ of all essential shadows (called just shadows, or TSP4-shadows), gives us access to all possible shadows 
$\bS_\La$ of GQT algebras. In terms of quivers, we conclude that for any GQT algebra $\La=KQ/I$, $Q=Q_\La$, the associated 
quiver $\bfQ_{\bS_\La}$, 
which is exactly the quiver $Q^\times$ obtained from $Q$ by deleting all $2$-cycles and loops, is one of 
the quivers $\bfQ_A$ identified with essential shadows $A\in\bE(n)$ (up to permutation or taking the opposite 
quiver). \medskip 

We can say more. Namely, one of major results of the paper \cite[see Section 5, especially Theorem 5.6]{PS1} proves 
that the loop-free part $Q^\circ$ of $Q$ (obtained from $Q$ by deleting loops) has the form of a disjoint union 
$$Q^\circ=Q^\times \sqcup E,$$ 
where $E$ is a disjoint union of $2$-cycles (in \cite{PS2}, we used a different notation for $E$, showing its 
underlying graph, whose edges encode the $2$-cycles in $E$). In other words, any Gabriel quiver $Q=Q_\La$ of 
a GQT algebra $\La$ is, up to permutation or taking the opposite quiver, obtained from its shadow $Q^\times$ 
in $\bE(n)$, by adding a finite number of disjoint $2$-cycles, and a finite number of loops. Moreover, the position 
of $2$-cycles is severly restricted by the following theorem, which summarizes the results of 
\cite[see the Main Theorem]{PS1}, called sometimes The Reconstruction Theorem.  

\begin{thm}\label{thm:recon} Let $n\geqslant 1$ be a natural number. Then there is a finite set $\bE(n)\subset\bS(n)$ 
of TSP4-shadows $n\times n$ such that every GQT algebra $\Lambda$ with Gabriel quiver $Q=Q_\La$ having $n$ 
vertices satisfies the following conditions. 
\begin{enumerate} 
\item[(a)] The subquiver $Q^\times$ obtained by removing all loops and $2$-cycles has $\Adj_{Q^\times}\in\bE(n)$ 
(up to relabelling of vertices or taking the opposite quiver).  
\item[(b)] If $Q^\times$ is non-empty, then the $2$-cycles $\xymatrix{a_i\ar@<-0.4ex>[r] & b_i \ar@<-0.4ex>[l]}$ in $E$ 
are pairwise disjoint, i.e. $\{a_i,b_i\}$ and $\{a_j,b_j\}$ are disjoint, for $i\neq j$. Moreover, any of the $2$-cycles 
$a\leftrightarrows b$ is contained in one of the following blocks in $Q$. \end{enumerate} 
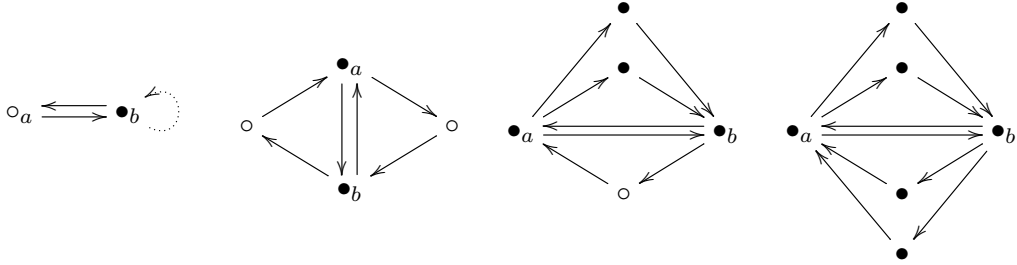
$$\xymatrix@R=0.3cm{\\ \\ \circ_a \ar@<-0.45ex>[r] & \bullet_b \ar@<-0.45ex>[l] \ar@(rd,ru)@{..>}[]& }
\xymatrix@R0.3cm{\\ & \bullet_a \ar[rd] \ar@<-0.1cm>[dd] & \\ 
\circ \ar[ru]   && \circ \ar[ld]\\ 
& \bullet_b \ar@<-0.1cm>[uu] \ar[lu] & } \quad 
\xymatrix@R=0.3cm{ &\bullet \ar[rdd] & \\ & \bullet\ar[rd]& \\ 
\bullet_a \ar[ru] \ar[ruu] \ar@<-0.35ex>[rr]&& \bullet_b \ar[ld] \ar@<-0.35ex>[ll]\\ & \circ \ar[lu] & } \quad 
\xymatrix@R=0.2cm{& \bullet \ar[rdd] & \\ & \bullet \ar[rd] & \\ 
\bullet_a \ar[ru]\ar[ruu] \ar@<-0.35ex>[rr] && \ar[ld]\ar[ldd] \ar@<-0.35ex>[ll] \bullet_b \\ 
& \bullet \ar[lu] & \\ & \ar[luu] \bullet & }$$ \end{thm} 

The above theorem was proved before we knew that any shadow of a GQT algebra is essential, and the original 
statement in \cite{PS1} was a bit weaker. This gives the general strategy for computing 
all Gabriel quivers of GQT algebras $\La$: 
\begin{itemize}
\item first, compute the set $\bE(n)$, 
\item then for any $A\in\bE(n)$, compute all quivers obtained from $\bfQ_A$, by adding a disjoint union of $2$-cycles 
according to the rules given in Theorem \ref{thm:recon}(b), 
\item and finally, consider all possible quivers obtained from the ones in previous step by adding loops.  
\end{itemize} 

The position of loops is partially restricted, as we will see later in Lemma \ref{lem:loops} (and Corollaries \ref{rmkloops}, \ref{nregloops}), but this is enough to get relatively small number of cases. In consequence, we obtain a set containing 
all Gabriel quivers of GQT algebras with $n$ vertices. We will see that already in cases $n=3,4,5$, the set can contain 
quivers, which are not Gabriel quivers of GQT algebras, and those will be excluded, obtaining a full description 
in Propositions \ref{prop:Ga3}, \ref{prop:Ga4} and \ref{prop:Ga5}, for $n=3,4$ or $5$, respectively. \smallskip 

In the original statement of the Reconstruction Theorem, it was not mentioned that we consider non-empty shadows, 
because it was clear from the proof. Actually, we can have zero shadow only for $n\leqslant 3$ (see Corollary \ref{coro1} 
further), or equivalently, if $n\geqslant 4$, then all shadows are assumed non-zero (i.e. there is no GQT algebra with 
zero shadow). It shall not lead to confusions, since the zero shadow is considered separately for $n\leqslant 3$ without 
using The Reconstruction Theorem. Note that the quiver identified with the zero shadow is the empty quiver with $n$-vertices.  

The proof of The Reconstruction Theorem is a series of lemmas covering all possibilities, case by case. Though it is 
technical, we would like to highlight its most important ingredients, which might be helpful in following 
analysis in the next section. \smallskip 

Let $i,j$ be vertices in $Q^\times$ which are connected by a $2$-cycle in $E$. Then the following rules apply 
\cite[see Lemmas 5.7-5.10]{PS1}. 
\begin{enumerate}
\item[(R1)] Vertices $i,j$ are at most $2$-regular in $Q^\times$ (and none is a source or target of double arrows).  
\item[(R2)] If one of $i,j$ is $2$-regular in $Q^\times$, then the second one is either isolated or also 
$2$-regular. In these cases, quiver $Q$ has $6$ vertices and it is, respectively, of the form 
$$
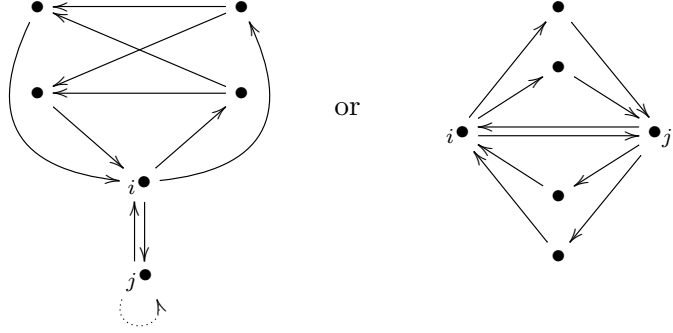
\xymatrix@R=0.55cm{ \bullet \ar@/_30pt/[rdd] && \ar[ll] \bullet \ar[lld] \\ 
\bullet \ar[rd] && \bullet \ar[ll] \ar[llu] \\ 
& _{i}\bullet \ar@<0.4ex>[d] \ar[ru] \ar@/_30pt/[ruu] & \\ & _{j}\bullet\ar@<0.4ex>[u] \ar@{.>}@(dl,dr)[]& } 
\qquad\xymatrix@R=0.4cm{\\ \\ \mbox{ or } \\ \\ }\qquad \xymatrix@R=0.4cm{& \bullet \ar[rdd] & \\ & \bullet \ar[rd] & \\ 
{}_{i}\bullet \ar[ru]\ar[ruu] \ar@<-0.35ex>[rr] && \ar[ld]\ar[ldd] \ar@<-0.35ex>[ll] \bullet_j \\ 
& \bullet \ar[lu] & \\ & \ar[luu] \bullet & }$$ 
\item[(R3)] If one of $i,j$ is a $(1,2)$-vertex in $Q^\times$, then the second one is either isolated, or 
a $(2,1)$-vertex, and then there is a block in $Q$ of the form  
$$
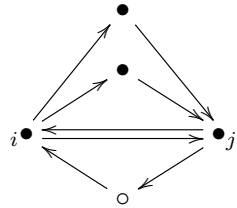
\xymatrix@R=0.4cm{ &\bullet \ar[rdd] & \\ & \bullet\ar[rd]& \\ 
_{i}\bullet \ar[ru] \ar[ruu] \ar@<-0.35ex>[rr]&& \bullet_{j} \ar[ld] \ar@<-0.35ex>[ll]\\ & \circ \ar[lu] & }$$ 
\item[(R4)] If one of $i,j$ is $1$-regular in $Q^\times$, then the second is either isolated or also $1$-regular, and then $i,j$ 
are contained in the following block of $Q$ (see \cite[Lemma 5.10]{PS1}): 
$$\xymatrix@R0.4cm{ & _{i}\bullet \ar[rd] \ar@<-0.1cm>[dd] & \\ 
\circ \ar[ru]   && \circ \ar[ld]\\ 
& \bullet_{j} \ar@<-0.1cm>[uu] \ar[lu] & }$$   
\end{enumerate}

In particular, it follows from (R2) that when $n\leqslant 5$, we have no $2$-vertex of $Q^\times$ connected with other 
$2$-vertex by a $2$-cycle in $E$. Therefore, all $2$-cycles in $E$ between non-isolated vertices are either connecting 
a pair of non-regular vertices in a block presented in (R3) or a pair of $1$-vertices as in (R4). \medskip 

\begin{rem}\label{noedge}\normalfont We cannot have a $2$-cycle $i\leftrightarrows j$ in $E$, if $i$ and $j$ 
are connected by an arrow in $Q$. Indeed, if this is the case, then $Q$ has a subquiver of the form 
$K_2^*$, say $\xymatrix@C=0.4cm{a \ar[r]^{\beta}& b \ar@<+0.4ex>[r]^{\ba}\ar@<-0.4ex>[r]_{\alpha}&c}$, with 
$a=c$ and $\{b,c\}=\{i,j\}$, and moreover, we have $\beta\alpha,\beta\ba\nprec I$, due to Lemma \ref{lem:3.3}, 
since there are no loops at vertices $i,j$, by rule (R1). \end{rem}

In the next part, we show some examples of TSP4-shadows in the smallest possible cases (used later in the 
classification of Gabriel quivers of GQT algebras). More precisely, the next three examples are devoted 
to describe a particular set $\bE(n)$ of essential shadows, for $n=3,4$ and $5$; in cases $n=3,4$, we 
additionally give the set $\bS(n)$ of all (basic) tame periodicity shadows, which in case $n=5$ is omitted, 
due to its size (besides it is not needed in the paper). This part is based on computations provided in the paper \cite{PS2}. 

\begin{expl}\label{exm:4.1} \normalfont The first interesting non-trivial case is $n=3$. Following \cite[Section 3]{PS2} 
we have $5$ (basic) tame periodicity shadows $\bS_1,\dots,\bS_5$ identified with the corresponding quivers 
$\bbQ_1,\dots,\bbQ_5$, $\bbQ_i=\bfQ_{\bS_i}$, given as follows. 
$$
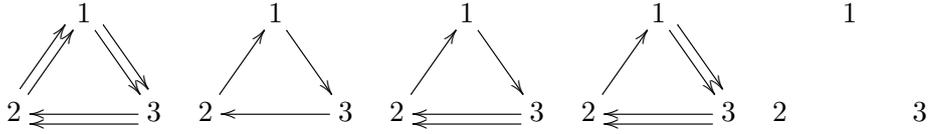
\xymatrix@C=0.5cm{ &1 \ar[rd]\ar@<+0.13cm>[rd] & \\ 
2 \ar[ru]\ar@<+0.13cm>[ru] && 3 \ar[ll]\ar@<+0.13cm>[ll]} \ \ 
\xymatrix@C=0.5cm{ &1 \ar[rd] & \\ 2 \ar[ru] && 3 \ar[ll]} \ \ 
\xymatrix@C=0.5cm{ &1 \ar[rd] & \\ 
2 \ar[ru] && 3 \ar[ll]\ar@<+0.13cm>[ll]} \ \ 
\xymatrix@C=0.5cm{ &1 \ar[rd]\ar@<+0.13cm>[rd] & \\ 
2 \ar[ru] && 3 \ar[ll]\ar@<+0.13cm>[ll]} \ \ 
\xymatrix@C=0.5cm{ & 1  & \\ 
2 && 3}$$ 
In this case, the shadow $\bS_4$ is not essential, since it does not satisfy condition (PS4). Note that $\bS_1$ 
is the Markov shadow, which does not satisfy the first part of (PS4), because it has $2$ and $-2$ in a 
row/column, or equivalently, the quiver $\bbQ_1$ has a consecutive double arrows. Except this one case,  
it is not allowed for an essential shadow. Hence we have four 
TSP4-shadows: $3$ non-trivial $\bbQ_1,\bbQ_2,\bbQ_3$ and the empty one $\bbQ_5$ (equivalently, $\bS_5=0$). 
Algebras with zero shadow can appear only for $n\leqslant 3$, and for $n=3$ their Gabriel quivers are known, 
due to the following result \cite[see Corollary 5.3]{PS1}. 
\end{expl} 

\begin{cor}\label{coro1} Let $\La$ be a GQT algebra with zero shadow $\bS=\bS_\La=0$. Then $n\leqslant 3$. 
In case $n=3$, the loop free part $Q^\circ$ of $Q$ is one of the following two quivers 
$$
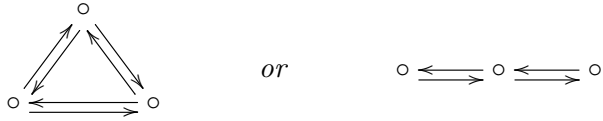
\xymatrix@R=0.4cm@C=0.5cm{ &\circ \ar[rd]\ar@<+0.13cm>[ld] & \\ \circ \ar[ru]\ar@<-0.13cm>[rr] && \circ \ar[ll]\ar@<+0.13cm>[lu]} \ 
\qquad\xymatrix@R=0.5cm{ \\ \mbox{ or } \\} 
\xymatrix@R=0.2cm{&&&&\\ & \circ \ar@<-0.13cm>[r] & \circ \ar[l] \ar@<-0.13cm>[r] & \circ \ar[l] & \\ }$$ 
Moreover, there are no loops in the first case, and at most one loop at each of the $1$-vertices of $Q^\circ$, 
in the second. \end{cor}  

\begin{expl}\label{exm:4.2} \normalfont Let now $n=4$. Then it follows from \cite[Section 3]{PS2} that we have $12$ (basic) 
tame periodicity shadows in $\bS(4)=\{\bS_1,\dots,\bS_{12}\}$ with the corresponding quivers $\bbQ_1,\dots,\bbQ_{12}$, 
$\bbQ_i=\bfQ_{\bS_i}$, given as follows.  
$$
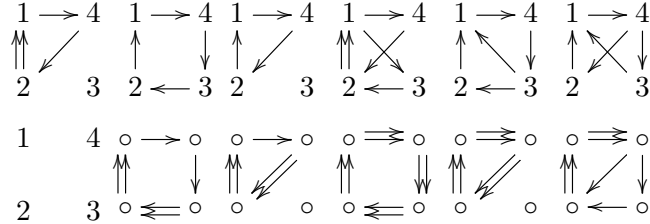
\xymatrix@C=0.5cm@R=0.5cm{1 \ar[r] & 4 \ar[ld] \\ 
2 \ar[u]\ar@<+0.1cm>[u]& 3 } \ 
\xymatrix@C=0.5cm@R=0.5cm{1 \ar[r] & 4 \ar[d] \\ 
2 \ar[u] & 3 \ar[l] } 
\xymatrix@C=0.5cm@R=0.5cm{1 \ar[r] & 4 \ar[ld] \\ 
2 \ar[u] & 3 } \ 
\xymatrix@C=0.5cm@R=0.5cm{1 \ar[r]\ar[rd] & 4 \ar[ld] \\ 
2 \ar[u]\ar@<+0.1cm>[u] & 3 \ar[l] } \ 
\xymatrix@C=0.5cm@R=0.5cm{1 \ar[r] & 4 \ar[d] \\ 
2 \ar[u] & 3 \ar[l]\ar[lu] } \ 
\xymatrix@C=0.5cm@R=0.5cm{1 \ar[r] & 4 \ar[d]\ar[ld] \\ 
2 \ar[u] & 3 \ar[lu] }$$ 
$$\xymatrix@C=0.5cm@R=0.5cm{1 & 4 \\ 
2 & 3}
\xymatrix@C=0.5cm@R=0.5cm{\circ \ar[r] & \circ \ar[d] \\ 
\circ \ar[u]\ar@<+0.1cm>[u]& \circ \ar[l]\ar@<+0.1cm>[l]} \ 
\xymatrix@C=0.5cm@R=0.5cm{\circ \ar[r] & \circ \ar[ld]\ar@<+0.1cm>[ld] \\ 
\circ \ar[u]\ar@<+0.1cm>[u]& \circ } \ \xymatrix@C=0.5cm@R=0.5cm{\circ \ar[r]\ar@<+0.1cm>[r] & \circ \ar[d]\ar@<+0.1cm>[d] \\ 
\circ \ar[u]\ar@<+0.1cm>[u]& \circ \ar[l]\ar@<+0.1cm>[l]} \ 
\xymatrix@C=0.5cm@R=0.5cm{\circ \ar[r]\ar@<+0.1cm>[r] & \circ \ar[ld]\ar@<+0.1cm>[ld] \\ 
\circ \ar[u]\ar@<+0.1cm>[u]& \circ} \ 
\xymatrix@C=0.5cm@R=0.5cm{\circ \ar[r]\ar@<+0.1cm>[r] & \circ \ar[d]\ar[ld] \\ 
\circ \ar[u]\ar@<+0.1cm>[u]& \circ \ar[l]}$$
It is easy to see that shadows $\bS_i$, for $i\geqslant 8$ do not satisfy (PS4). Additionally, the shadows 
$\bS_8$ and $\bS_{10}$ do not satisfy (PS5). Thus we get $7$ essential shadows in $\bE(4)$, which correspond to seven 
essential quivers $\bbQ_1,\dots,\bbQ_{7}$ (including the empty one $\bbQ_7=\emptyset$). 
\end{expl} 

\begin{expl}\label{exm:4.3} \normalfont Finally, according to the list of all essential shadows of size $n=5$ 
presented in \cite[Section 4]{PS2}, we have the following quivers associated to the $26$ essential shadows 
in $\bE(5)=\{\bS_1,\dots,\bS_{26}\}$. \smallskip 

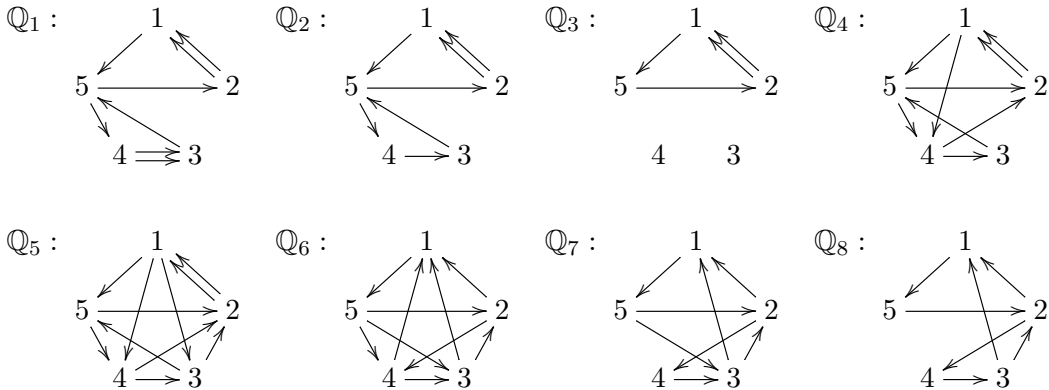
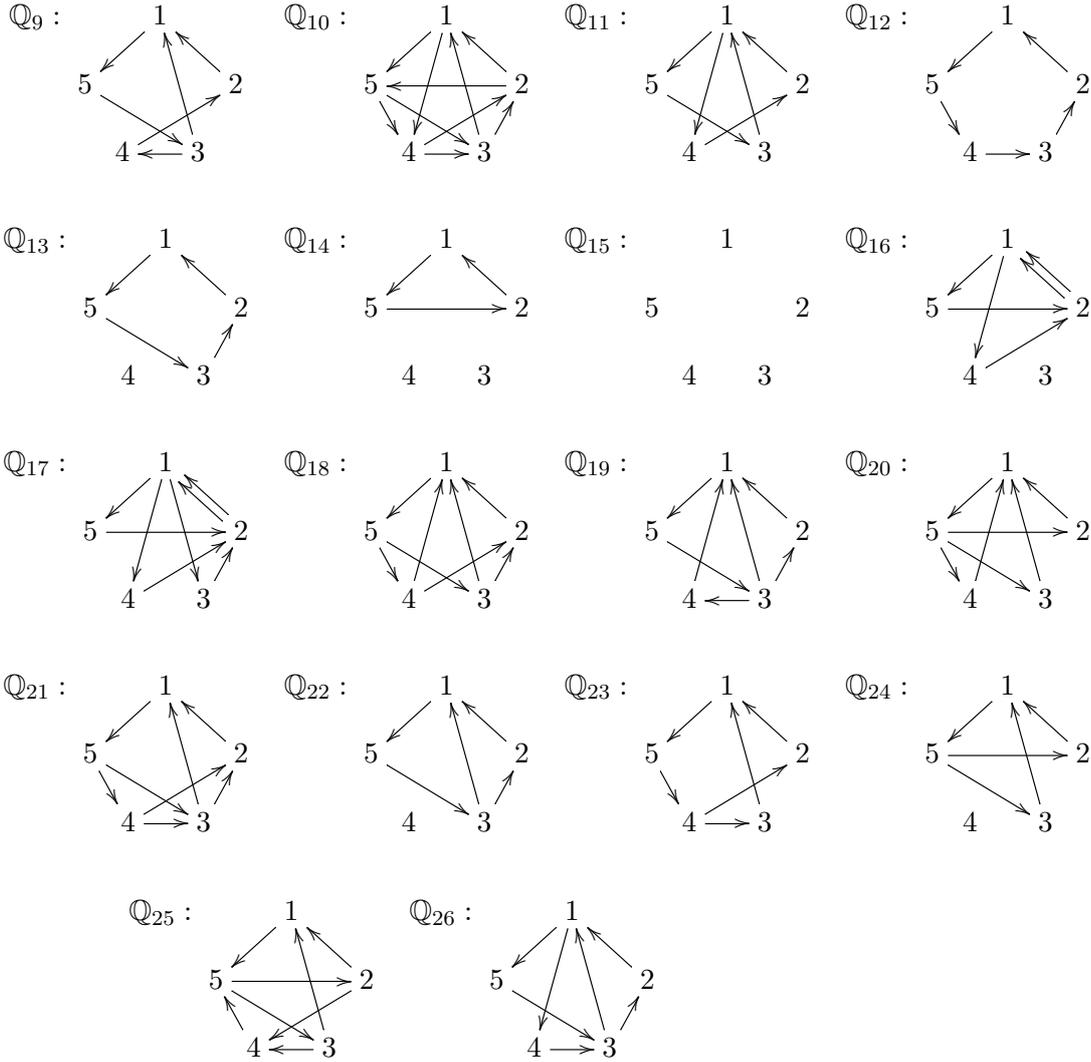
\begin{tabular}{cccc}
\(\bbQ_1: \xymatrix@R=0.1cm@C=0.075cm{& & 1 \ar[ddll]  & & \\
&&&&\\
5 \ar[rrrr] \ar[ddr]  & & & & 2 \ar@<0.35ex>[uull] \ar@<-0.35ex>[uull]\\
&&&&\\
& 4 \ar@<0.35ex>[rr] \ar@<-0.35ex>[rr] & & 3 \ar[uulll] &}\) 

&

\(\bbQ_2: \xymatrix@R=0.1cm@C=0.075cm{& & 1 \ar[ddll]  & & \\
&&&&\\
5 \ar[rrrr] \ar[ddr]  & & & & 2 \ar@<0.35ex>[uull] \ar@<-0.35ex>[uull]\\
&&&&\\
& 4 \ar[rr] & & 3 \ar[uulll] &}\)

&

\(\bbQ_3: \xymatrix@R=0.1cm@C=0.075cm{& & 1 \ar[ddll]  & & \\
&&&&\\
5 \ar[rrrr]   & & & & 2 \ar@<0.35ex>[uull] \ar@<-0.35ex>[uull]\\
&&&&\\
& 4  & & 3  &}\)

&

\(\bbQ_4: \xymatrix@R=0.1cm@C=0.075cm{& & 1 \ar[ddll] \ar[ddddl] & & \\
&&&&\\
5 \ar[rrrr] \ar[ddr]  & & & & 2 \ar@<0.35ex>[uull] \ar@<-0.35ex>[uull]\\
&&&&\\
& 4 \ar[uurrr] \ar[rr]  & & 3 \ar[uulll] &}\)

\\[70pt]

\(\bbQ_5: \xymatrix@R=0.1cm@C=0.075cm{& & 1 \ar[ddll] \ar[ddddl] \ar[ddddr] & & \\
&&&&\\
5 \ar[rrrr] \ar[ddr]  & & & & 2 \ar@<0.35ex>[uull] \ar@<-0.35ex>[uull]\\
&&&&\\
& 4 \ar[uurrr] \ar[rr]  & & 3 \ar[uur] \ar[uulll] &}\)

&

\(\bbQ_6: \xymatrix@R=0.1cm@C=0.075cm{& & 1 \ar[ddll]  & & \\
&&&&\\
5 \ar[rrrr] \ar[ddrrr] \ar[ddr]  & & & & 2 \ar[uull] \ar[ddlll]\\
&&&&\\
& 4 \ar[uuuur] \ar[rr]  & & 3 \ar[uuuul] \ar[uur] &}\)

&

\(\bbQ_7: \xymatrix@R=0.1cm@C=0.075cm{& & 1 \ar[ddll]  & & \\
&&&&\\
5 \ar[rrrr] \ar[ddrrr]  & & & & 2 \ar[uull] \ar[ddlll]\\
&&&&\\
& 4 \ar[rr]  & & 3 \ar[uuuul] \ar[uur] &}\)

&

\(\bbQ_8: \xymatrix@R=0.1cm@C=0.075cm{& & 1 \ar[ddll]  & & \\
&&&&\\
5 \ar[rrrr]   & & & & 2 \ar[uull] \ar[ddlll]\\
&&&&\\
& 4 \ar[rr]  & & 3 \ar[uuuul] \ar[uur] &}\)

\\[70pt]

\(\bbQ_9: \xymatrix@R=0.1cm@C=0.075cm{& & 1 \ar[ddll]  & & \\
&&&&\\
5 \ar[ddrrr]  & & & & 2 \ar[uull] \\
&&&&\\
& 4 \ar[uurrr]  & & 3 \ar[uuuul] \ar[ll] &}\)

&

\(\bbQ_{10}: \xymatrix@R=0.1cm@C=0.075cm{& & 1 \ar[ddddl] \ar[ddll]  & & \\
&&&&\\
5 \ar[ddrrr] \ar[ddr]  & & & & 2 \ar[uull] \ar[llll]\\
&&&&\\
& 4 \ar[uurrr] \ar[rr] & & 3 \ar[uuuul] \ar[uur] &}\)

&

\(\bbQ_{11}: \xymatrix@R=0.1cm@C=0.075cm{& & 1 \ar[ddddl] \ar[ddll] & & \\
&&&&\\
5 \ar[ddrrr]  & & & & 2 \ar[uull]\\
&&&&\\
& 4 \ar[uurrr]  & & 3 \ar[uuuul] &}\)

&

\(\bbQ_{12}: \xymatrix@R=0.1cm@C=0.075cm{& & 1 \ar[ddll]  & & \\
&&&&\\
5 \ar[ddr]   & & & & 2 \ar[uull] \\
&&&&\\
& 4 \ar[rr]  & & 3  \ar[uur] &}\)

\\[70pt]
\end{tabular} 

\begin{tabular}{cccc}

\(\bbQ_{13}: \xymatrix@R=0.1cm@C=0.075cm{& & 1 \ar[ddll]  & & \\
&&&&\\
5 \ar[ddrrr]  & & & & 2 \ar[uull] \\
&&&&\\
& 4   & & 3 \ar[uur] &}\)

&

\(\bbQ_{14}: \xymatrix@R=0.1cm@C=0.075cm{& & 1 \ar[ddll]  & & \\
&&&&\\
5 \ar[rrrr] & & & & 2 \ar[uull] \\
&&&&\\
& 4  & & 3 &}\)

&

\(\bbQ_{15}: \xymatrix@R=0.1cm@C=0.075cm{& & 1  & & \\
&&&&\\
5   & & & & 2 \\
&&&&\\
& 4   & & 3  &}\)

&

\(\bbQ_{16}: \xymatrix@R=0.1cm@C=0.075cm{& & 1 \ar[ddddl] \ar[ddll]  & & \\
&&&&\\
5 \ar[rrrr]   & & & & 2 \ar@<0.35ex>[uull] \ar@<-0.35ex>[uull] \\
&&&&\\
& 4 \ar[uurrr]  & & 3  &}\)

\\[70pt]

\(\bbQ_{17}: \xymatrix@R=0.1cm@C=0.075cm{& & 1 \ar[ddddr] \ar[ddddl] \ar[ddll]  & & \\
&&&&\\
5 \ar[rrrr]  & & & & 2 \ar@<0.35ex>[uull] \ar@<-0.35ex>[uull] \\
&&&&\\
& 4 \ar[uurrr]  & & 3 \ar[uur] &}\)

&

\(\bbQ_{18}: \xymatrix@R=0.1cm@C=0.075cm{& & 1 \ar[ddll]  & & \\
&&&&\\
5 \ar[ddrrr] \ar[ddr]  & & & & 2 \ar[uull] \\
&&&&\\
& 4 \ar[uuuur] \ar[uurrr] & & 3 \ar[uuuul] \ar[uur] &}\)

&

\(\bbQ_{19}: \xymatrix@R=0.1cm@C=0.075cm{& & 1 \ar[ddll] & & \\
&&&&\\
5 \ar[ddrrr]  & & & & 2 \ar[uull]\\
&&&&\\
& 4 \ar[uuuur]  & & 3 \ar[uuuul] \ar[uur] \ar[ll] &}\)

&

\(\bbQ_{20}: \xymatrix@R=0.1cm@C=0.075cm{& & 1 \ar[ddll]  & & \\
&&&&\\
5 \ar[rrrr] \ar[ddrrr] \ar[ddr]   & & & & 2 \ar[uull] \\
&&&&\\
& 4 \ar[uuuur]  & & 3  \ar[uuuul] &}\)

\\[70pt]

\(\bbQ_{21}: \xymatrix@R=0.1cm@C=0.075cm{& & 1 \ar[ddll]  & & \\
&&&&\\
5 \ar[ddrrr] \ar[ddr] & & & & 2 \ar[uull] \\
&&&&\\
& 4 \ar[uurrr] \ar[rr]  & & 3 \ar[uuuul] \ar[uur] &}\)

&

\(\bbQ_{22}: \xymatrix@R=0.1cm@C=0.075cm{& & 1 \ar[ddll]  & & \\
&&&&\\
5 \ar[ddrrr]  & & & & 2 \ar[uull] \\
&&&&\\
& 4  & & 3 \ar[uuuul] \ar[uur] &}\)

&

\(\bbQ_{23}: \xymatrix@R=0.1cm@C=0.075cm{& & 1 \ar[ddll] & & \\
&&&&\\
5 \ar[ddr]  & & & & 2 \ar[uull]\\
&&&&\\
& 4 \ar[uurrr] \ar[rr] & & 3 \ar[uuuul] &}\)

&

\(\bbQ_{24}: \xymatrix@R=0.1cm@C=0.075cm{& & 1 \ar[ddll]  & & \\
&&&&\\
5 \ar[rrrr] \ar[ddrrr]  & & & & 2 \ar[uull] \\
&&&&\\
& 4  & & 3  \ar[uuuul] &}\)

\\[70pt] 

\end{tabular} 

\begin{tabular}{cccc}
\( \xymatrix{ & }\)
&  

\(\bbQ_{25}: \xymatrix@R=0.1cm@C=0.075cm{& & 1 \ar[ddll]  & & \\
&&&&\\
5 \ar[rrrr] \ar[ddrrr] & & & & 2 \ar[uull] \ar[ddlll] \\
&&&&\\
& 4 \ar[uul]  & & 3 \ar[uuuul] \ar[ll] &}\)

&

\(\bbQ_{26}: \xymatrix@R=0.1cm@C=0.075cm{& & 1 \ar[ddll] \ar[ddddl] & & \\
&&&&\\
5 \ar[ddrrr] & & & & 2 \ar[uull] \\
&&&&\\
& 4 \ar[rr] & & 3 \ar[uuuul] \ar[uur] &}\)

& 
\( \xymatrix{ & }\) \\ 
\end{tabular} \end{expl} 

Final part of this section is devoted to present some partial results concerning the position of loops in 
quivers $Q=Q_\La$, $Q^\circ=\bbQ\sqcup E$, coming from GQT algebras. We will mostly use tameness, but 
periodicity is also necessary.  

\begin{rem}\label{loops:rk0} \normalfont It is easy to check that a loop at vertex $i$ of $Q^\circ$ with $|i^+|\geqslant 3$ 
or $|i^-|\geqslant 3$ induces a wild subcategory of type $\wt{\wt{\bD}}_4$. It means that {\it there may be a loop 
at vertex $i$ of $Q^\circ$ only when $i$ is at most $2$-regular.} Similarly, as it was the case in rule (R1) for 
$2$-cycles. \end{rem} \medskip 

Next, we will analize when an at most $2$-vertex admits a loop. We begin with the $2$-regular 
case, which involves only tameness. 

\begin{lemma}\label{lem:loops} Let $i$ be a $2$-vertex of $Q^\circ$. If one of neighbours of $i$ 
is a $2$-vertex (in $Q$), then there is no loop at $i$. \end{lemma} 

\begin{proof} Assume $i$ is a $2$-vertex of $Q^\circ$, so that $Q^\circ$ has a 
subquiver of the form 
$$
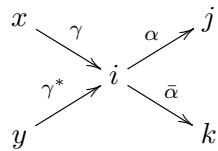
\xymatrix@R=0.25cm{ x \ar[rd]^{\gamma} && j \\ &i \ar[ru]^{\alpha} \ar[rd]^{\ba}& \\ y \ar[ru]^{\gamma^*} & & k}$$ 
Let one of the successors $s\in\{j,k\}$ of $i$ satisfy $|s^-|\geqslant 2$. If $\rho:i\to i$ is a loop, then 
we get the following wild subcategory in covering 
$$\xymatrix@R=0.4cm{ & y \ar[d] & \circ &  & \\ x \ar[r] & i & \ar[l]_{\rho} i \ar[u] \ar[r] & s & \ar[l] s'}$$ 
Similar argument works if one of the predecessors $p\in\{x,y\}$ has $|p^+|\geqslant 2$, and the claim follows. 
\end{proof} \smallskip 

We have the following immediate consequence of the proof. 

\begin{cor}\label{rmkloops} Let $i$ be a $2$-vertex of $Q^\circ$ satisfying one of the following conditions  
\begin{itemize}
\item $|j^-|$ or $|k^-|\geqslant 2$, or there is an arrow $\sigma\in j^+$ with $\alpha\sigma\nprec I$ 
(or $\sigma\in k^+$ with $\ba\sigma\nprec I$); 
\item $|x^+|$ or $|y^+|\geqslant 2$, or there is an arrow $\sigma\in x^-$ with $\sigma\gamma\nprec I$ 
(or $\sigma\in y^-$ with $\sigma\gamma^*\nprec I$). 
\end{itemize}  
Then there is no loop at $i$ in $Q$. \end{cor} \smallskip 

For a $(1,2)$-vertex $i$, we can prove analogous property. Indeed, let $i$ be a $(1,2)$-vertex of $Q^\circ$, and 
suppose there is a loop $\rho$ at $s(\rho)=i$. It follows that $\rho\alpha,\rho\ba \nprec I$. Indeed, otherwise, 
by Lemma \ref{lem:3.3}, we would get an arrow $j\to i$ or $k\to i$, and then $j=x$ or $k=x$, which is impossible, 
because $p_i^-=p_i^+$ forces $p_x=p_j+p_k$. As a result, we have the following (tame) hereditary subcategory in covering  
$$\xymatrix@R=0.4cm{ & k & k &   \\ j & \ar[l]_{\alpha}\ar[u]^{\ba} i & \ar[l]_{\rho} i \ar[u]^{\ba} \ar[r]^{\alpha} & j }$$ 
It can be extended to a wild subcategory if the first condition in Corollary \ref{rmkloops} holds. Analogous 
arguments work with $(2,1)$-vertex. Consequently, we get the following. 

\begin{cor}\label{nregloops} (a) If $i$ is a $(1,2)$-vertex of $Q^\circ$ such that $|j^-|$ or $|k^-|\geqslant 2$, or there is an arrow 
$\sigma\in j^+$ with $\alpha\sigma\nprec I$ (or $\sigma\in k^+$ with $\ba\sigma\nprec I$), then there is no loop 
at $i$ in $Q$. \\ 
(b) If $i$ is a $(2,1)$-vertex of $Q^\circ$ such that $|x^+|$ or $|y^+|\geqslant 2$, or there is an arrow 
$\sigma\in x^-$ with $\sigma\gamma\nprec I$ (or $\sigma\in y^-$ with $\sigma\gamma^*\nprec I$), then there is no 
loop at $i$ in $Q$. \end{cor} \medskip

\section{Gabriel quivers}\label{sec:6} 

In this section we will present a full classification (up to permutation, and taking $Q^{\op}$) of all 
possible Gabriel quivers $Q=Q_\Lambda$ of GQT algebras in case $n=|Q_0|\leqslant 5$. \medskip 

Fix an indecomposable GQT algebra $\La=KQ/I$, whose Gabriel quiver $Q=Q_\La$ has at most $5$ vertices. We split our 
considerations into three steps: in the first, we recall what is known for very small size (up to $3$), and in the 
next two, we deal with sizes $n=4$ and $n=5$, respectively. \medskip 

We only mention that classification of all self-injective algebras of finite representation type is known, 
due to results of Riedtmann \cite{Rd1,Rd2} (see also \cite{Rogg,Was}). Moreover, Dugas proved in \cite{Du} 
that all of them are periodic, so we will focus on infinite representation type (see also Remark \ref{der_inf}). 
For an elegant survey about self-injective algebras, we refer to \cite{SkoCont}. \medskip

\subsection{Quivers with at most 3 vertices}\label{subs:5.1}

First, note that the structure of local ($n=1$) tame symmetric algebras is well understood \cite{E90}. 
Their Gabriel quivers consist of one vertex and at most two loops (for relations we refer the reader to 
\cite[Theorem III.1]{E90}). \medskip 

Tame symmetric algebras with $n=2$ vertices were classified by Donovan \cite[Section 2]{Don} (see also 
\cite[VI.8 and VII.7]{E90}). According to \cite[Proposition 4.3]{PS1}, those which are GQT are given by one of the 
following four quivers. 
$\newline$
$$\xymatrix@C=0.4cm{\circ \ar@<+0.1cm>@/^10pt/[rr] \ar@/^8pt/[rr] & & \circ \ar@<+0.1cm>@/^10pt/[ll]\ar@/^8pt/[ll] } 
\xymatrix@C=0.4cm{& \circ \ar@/^8pt/[rr] && \ar@/^8pt/[ll] \circ & } 
 \xymatrix@C=0.4cm{&\ar@(lu,ld)[] \ar@/^8pt/[rr] \circ && \ar@/^8pt/[ll] \circ \ar@(ru,rd)[] & } 
\xymatrix@C=0.4cm{&\ar@/^8pt/[rr] \circ && \ar@/^8pt/[ll] \ar@(ru,rd)[] \circ & }$$ 
$\newline$ 
Now, let $n=3$. Following Section \ref{sec:4}, we deduce that the Gabriel quiver $Q$ of $\Lambda$ 
is given by the formula $Q^\circ = \bbQ\sqcup E$, where $\bbQ=Q^\times$ is one of the TSP4-shadows 
$\bbQ=\bfQ_{\bS}$, for $\bS\in\bE(n)$, and $E$ a disjoint union of $2$-cycles, if $\bS \neq 0$. \smallskip 

In this case, it follows from Example \ref{exm:4.1} that $\bbQ$ is one of the essential 
shadows $\bbQ\in\bE(3)=\{\bbQ_1,\bbQ_2,\bbQ_3,\bbQ_5\}$, where $\bbQ_5=\emptyset$ is the empty one. \smallskip 

If $\bbQ=\bbQ_5=\emptyset$, then it follows from Corollary \ref{coro1} that $Q^\circ=E$ is of one 
the following two quivers 
$$
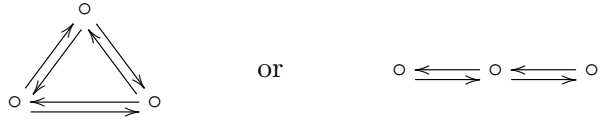
\xymatrix@C=0.5cm{ &\circ \ar[rd]\ar@<+0.13cm>[ld] & \\ \circ \ar[ru]\ar@<-0.13cm>[rr] && \circ \ar[ll]\ar@<+0.13cm>[lu]} \ 
\qquad\xymatrix@R=0.5cm{ \\ \mbox{ or } \\} 
\xymatrix@R=0.5cm{&&&& \\ 
& \circ \ar@<-0.13cm>[r] & \circ \ar[l] \ar@<-0.13cm>[r] & \circ \ar[l] & \\ 
&&&& \\}$$ 
where $Q=Q^\circ$ in the first case, and in the second, we have at most one loop in $Q$ at each of the 
$1$-vertices of $Q^\circ$. This describes the Gabriel quivers $Q$ with zero shadow. \medskip 

It remains to find all $Q$ with $\bbQ=\bbQ_i$, with $i\in\{1,2,3\}$. \smallskip 

If $\bbQ=\bbQ_1$, then $E$ is empty and $Q$ has no loops (otherwise, $\Lambda$ is wild). In this case, we 
have $Q=Q^\circ=\bbQ$, so $Q$ is the Markov quiver 
$$
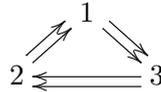
\xymatrix@C=0.5cm@R=0.35cm{ & 1 \ar[rd]\ar@<+0.13cm>[rd] & \\ 
2 \ar[ru]\ar@<+0.13cm>[ru] && 3 \ar[ll]\ar@<+0.13cm>[ll]}$$ 

Consider a triangle $\bbQ=\bbQ_2$. Then $E$ must be empty, because otherwise, there is a $2$-cycle in $Q$, 
not allowed by the rule (R4) (see also Remark \ref{noedge}). Hence 
$Q^\circ=\bbQ$ is a triangle, and $Q$ is obtained by adding at most $3$ loops. If there are three loops, we get a 
$2$-regular quiver $Q$. The remaining cases are excluded. Indeed, if there are at most two loops, then $Q$ admits at 
least one $1$-regular vertex $i\in Q_0=\{1,2,3\}$, and we have equal dimension vectors of all projectives: $p_1=p_2=p_3$, 
due to the identities $p_k^-=p_k^+$. But this is impossible, since then $p_i^+=p_i$, which contradicts Lemma 
\ref{lem:3.1}. \medskip 

Finally, consider the case $\bbQ=\bbQ_3$. Here as above $E$ must be empty (see Remark \ref{noedge}), so $Q^\circ=\bbQ$ 
and $Q$ admits at most one loop (at the unique $1$-regular vertex of $\bbQ$). But then $Q$ has a $(1,2)$-vertex forbidden 
in Corollary \ref{12wD}, hence the shadow $\bbQ_3$ is excluded. \smallskip

Summing up, we have proved the following proposition. 

\begin{prop}\label{prop:Ga3} If $\Lambda$ is a GQT algebra such that $Q=Q_\Lambda$ has $n=3$ vertices, then $Q$ has one of the 
following forms: 
$$
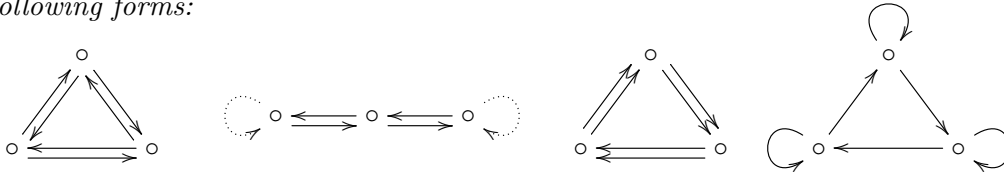
\xymatrix@C=0.5cm{ &\circ \ar[rd]\ar@<+0.13cm>[ld] & \\ \circ \ar[ru]\ar@<-0.13cm>[rr] && \circ \ar[ll]\ar@<+0.13cm>[lu]} \ 
\xymatrix@R=0.5cm{&&& \\ 
& \ar@{.>}@(lu,ld)[]\circ \ar@<-0.13cm>[r] & \circ \ar[l] \ar@<-0.13cm>[r] & \circ \ar[l] \ar@{.>}@(ru,rd)[]  \\ 
&&& \\} \ \ \qquad 
\xymatrix@C=0.5cm{ &\circ \ar[rd]\ar@<+0.13cm>[rd] & \\ 
\circ \ar[ru]\ar@<+0.13cm>[ru] && \circ \ar[ll]\ar@<+0.13cm>[ll]} \ 
\xymatrix@C=0.5cm{& &\circ \ar[rd] \ar@(ul,ur)[]& &\\ & \circ\ar@(lu,ld)[] \ar[ru] && \circ \ar[ll]\ar@(ru,rd)[]&}$$ 
where dotted loops indicate that there may be at most two loops. \end{prop}

\subsection{Quivers with 4 vertices}\label{subsec:n=4}

Now, we will describe quivers $Q=Q_\Lambda$ with $n=4$ vertices. It follows from Example \ref{exm:4.2} that 
$Q^\circ=\bbQ\sqcup E$, where $\bbQ=Q^\times$ is one of the six essential shadows $\bbQ_1,\dots,\bbQ_6$ 
(the case of $\bbQ_{7}=\emptyset$ is omitted, due to Corollary \ref{coro1}). Let $e$ denote the number of 
(disjoint) $2$-cycles in $E$. \medskip 

First, as in case $\bbQ_3$ for $n=3$, we can easily exclude shadow $\bbQ=\bbQ_1$, due to Corollary \ref{12wD}. \medskip 

Next, let $\bbQ=\bbQ_2$. Then $E$ contains at most one $2$-cycle connecting a pair of (opposite) $1$-regular vertices 
of $\bbQ$. Indeed, if it was not the case, then $\bbQ$ is $2$-regular, and hence $Q=Q^\circ$, by Lemma 
\ref{lem:loops}. But then applying \cite[Main Theorem]{AGQT}, we deduce that $\La$ is a weighted surface algebra, 
which is impossible, since $Q$ is not a triangulation quiver (it is not a glueing of blocks I-III). Therefore, we have 
$e\leqslant 1$, and hence $Q^\circ$ has the following form (up to permutation of vertices)  
$$
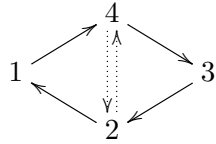
\xymatrix@R=0.2cm{& 4 \ar[rd]\ar@{.>}@<-.37ex>[dd]& \\ 
1 \ar[ru] && 3\ar[ld]\\ & 2\ar[lu] \ar@{.>}@<-.37ex>[uu] & }$$ 
Observe also that there must be two loops at vertices $1$ and $3$. Indeed, all vertices of $\bbQ$ are 
$1$-regular, so we have identities: $p_1=p_3$ and $p_2=p_4$. If $e=1$, i.e. there is a $2$-cycle $2 \leftrightarrows 4$, then we have no loops at vertices $2$ and $4$, 
due to Lemma \ref{lem:loops}, and $Q$ is a biregular quiver. 
In this case, there are two loops at $1$ and $3$, because otherwise, there exists a $1$-vertex in $Q$ not 
lying in a block presented in Theorem \ref{bireg1}. If $e=0$, that is $Q^\circ=\bbQ$, we must have exactly two loops 
at opposite vertices. Indeed, if this is not the case, then there is an arrow $\alpha:i\to j\in Q_1$ with 
$i^+=\{\alpha\}=j^-$, which contradicts Lemma \ref{lem:3.2}. \smallskip 

As a result, there are two Gabriel quivers $Q$ with $\bbQ=\bbQ_2$, and these are of the form 
$$
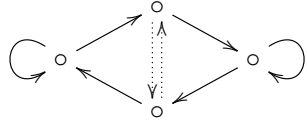
\xymatrix@R=0.2cm{&\circ\ar[rd]\ar@{.>}@<-.37ex>[dd]& \\ 
\ar@(lu,ld)[]\circ\ar[ru] && \circ\ar[ld]\ar@(ru,rd)[]\\ & \circ\ar[lu] \ar@{.>}@<-.37ex>[uu] & }$$
with or without the dotted $2$-cycle, denoted by $Q^{(2)}$ and $Q^{(1)}$, respectively. \medskip 

Suppose now that $\bbQ=\bbQ_{3}$. Then $E$ has exactly one $2$-cycle connecting an isolated vertex with one of the 
$1$-regular vertices in a triangle. There is at least one, since $Q$ is connected, but no more, because $\bbQ$ 
has no pairs of $1$-vertices allowed by the rule (R4).  

Moreover, as in the case of a triangle $\bbQ_2$ for $n=3$, we can show that there are loops at both 
remaining vertices (using Lemma \ref{lem:3.1}). As a result, for $\bbQ=\bbQ_{3}$, the quiver $Q=Q_\Lambda$ has the form 
$$
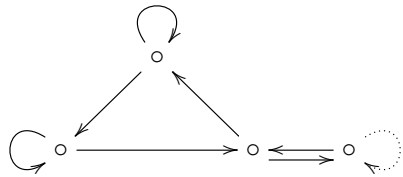
\xymatrix{& \circ\ar@(lu,ru)[] \ar[ld] && \\ 
\ar@(ul,dl)[] \circ \ar[rr] & & \circ\ar[lu] \ar@<-0.13cm>[r] & \circ \ar[l] \ar@{.>}@(ru,rd)[]}$$ 
The quiver without dotted loop is denoted by $Q^{(3)}$, and the quiver with dotted loop, by $Q^{(4)}$.

Now, consider $\bbQ=\bbQ_4$. As above, $\bbQ$ is a biregular quiver, and $Q$ has no loops at vertices 
$1$ and $2$, by the rule (R1). Moreover, there is no block of type $V_2$ in $\bbQ$, thus by the 
Reconstruction Theorem, we obtain that $E=\emptyset$. Hence $Q$ 
is also biregular, and it is obtained from $Q^\circ=\bbQ$, by adding loops at vertices $3$ or $4$. Actually, 
$Q$ is biregular, so we can use Theorem \ref{bireg1} to see (as for $\bbQ_2$) that there must be two loops 
in $Q$, since otherwise, the quiver $Q$ contains a forbidden $1$-vertex. Therefore, the unique 
$Q=Q^{(5)}$ with $\bbQ=\bbQ_4$ is of the form 
$$
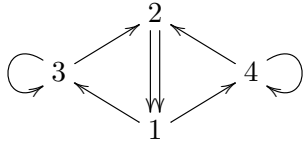
\xymatrix@R=0.3cm{&2\ar@<-.37ex>[dd]\ar@<.37ex>[dd]& \\ 
\ar@(lu,ld)[]3\ar[ru] && 4\ar[lu]\ar@(ru,rd)[]\\ & 1\ar[lu]\ar[ru]  & }  $$   

Further, we exclude the case $\bbQ=\bbQ_5$. Indeed, suppose to the contrary that there is a GQT algebra $\La=KQ/I$ 
with shadow $\bbQ=\bbQ_5$ of the form 
$$
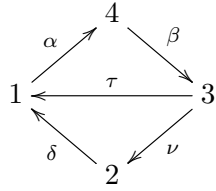
\xymatrix@R=0.6cm{& 4 \ar[rd]^{\beta} & \\ 1\ar[ru]^{\alpha} && 3\ar[ll]_{\tau} \ar[ld]^{\nu} \\ 
&2\ar[lu]^{\delta}& }$$ 
First, observe that $E=\emptyset$. Indeed, otherwise $E$ contains exactly one $2$-cycle between the two $1$-vertices 
of $\bbQ$, and there are no loops at vertices $2$ and $4$, due to Lemma \ref{lem:loops}. Then we get 
$p_1+p_2=p_4^-=p_4^+=p_3+p_2$, but $p_1^-=p_1^+$ gives $p_2+p_3=p_4$ (whenever there is a loop 
at $1$ or not), so we obtain $p_4^+=p_4$, which is impossible, due to Lemma \ref{lem:3.1}. This shows 
that $E$ is empty, so we have $Q^\circ=\bbQ$. \smallskip 

Observe that we cannot have a loop in $Q$ at vertices $1$ or $3$, by Corollary \ref{nregloops}. 
Hence $Q$ is a biserial quiver, and therefore, both non-regular vertices $1$ and $3$ in $Q$ are of type N 
(see Lemma \ref{prop:typeN}). It follows also that there must be a loop $\rho:4\to 4$ in $Q$, because otherwise 
$4^+=\{\beta\}=3^-$, and we get a contradiction with Lemma \ref{lem:3.2}. Consequently, we have the following wild 
subcategory in covering 
$$
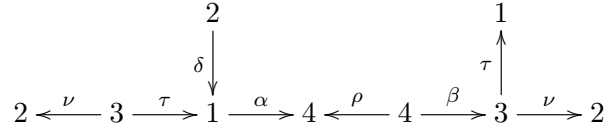
\xymatrix{&&2\ar[d]_{\delta} &&& 1& \\
2 & \ar[l]_{\nu} 3 \ar[r]^{\tau} & 1 \ar[r]^{\alpha} & 4 & \ar[l]_{\rho} 4 \ar[r]^{\beta} & 3 \ar[u]^{\tau}\ar[r]^{\nu} &2}$$ 
This shows that $\bbQ_5$ cannot be a shadow of a GQT algebra; see also Proposition \ref{blockQ5} at the end, which 
generalizes the above observation. \medskip 

Finally, we describe quivers $Q$ with $\bbQ=\bbQ_6$. As for $\bbQ_4$, we have no subquiver allowed by the rule (R4), 
so $E=\emptyset$ and $Q^\circ=\bbQ$ has the form 
$$
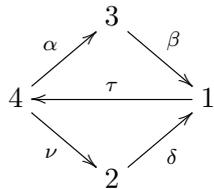
\xymatrix@R=0.6cm{& 3 \ar[rd]^{\beta} & \\ 4\ar[ru]^{\alpha}\ar[rd]_{\nu} && 1\ar[ll]_{\tau}  \\ 
&2\ar[ru]_{\delta}& }$$ 
We claim that $Q=Q^{(6)}$ contains exactly two loops (at vertices $1$ and $4$). Indeed, there must be at least 
one loop at $1$ or $4$, since in case of no loops, the arrow $\tau:1\to 4$ satisfies $1^+=\{\tau\}=4^-$, 
which leads to a contradiction with Lemma \ref{lem:3.2}. If there is only one loop, 
say at $4$, then $p_4^-=p_4^+$ gives the following identity $p_1+p_4=p_2+p_3+p_4$, hence $p_1=p_2+p_3=p_1^-$, 
which contradicts Lemma \ref{lem:3.1}. Consequently, there are two loops $\rho,\sigma$ 
at vertices $1$ and $4$, respectively. Now, it suffices to see that we cannot have more loops.  
If this is not the case, say $\mu:3\to 3$, then $\La$ admits a wild hereditary subcategory  
$$
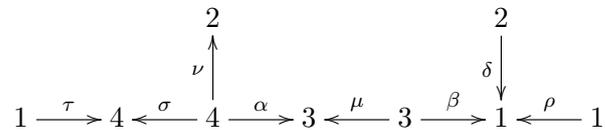
\xymatrix{&& 2 &&& 2\ar[d]_{\delta} & \\ 
1 \ar[r]^{\tau} & 4 & \ar[l]_{\sigma} 4 \ar[u]^{\nu} \ar[r]^{\alpha} & 3 & \ar[l]_{\mu} 3 \ar[r]^{\beta} & 1 & \ar[l]_{\rho} 1}$$ 
Similar wild subcategory arises in case of loop at $2$. \medskip 

Concluding all the above cases, we obtain the following.  

\begin{prop}\label{prop:Ga4} If $\La$ is a GQT algebra $\La=KQ/I$ such that $Q=Q_\La$ has $n=4$ vertices, then $Q$ 
is one of the following quivers $Q^{(1)},\dots,Q^{(6)}$ 
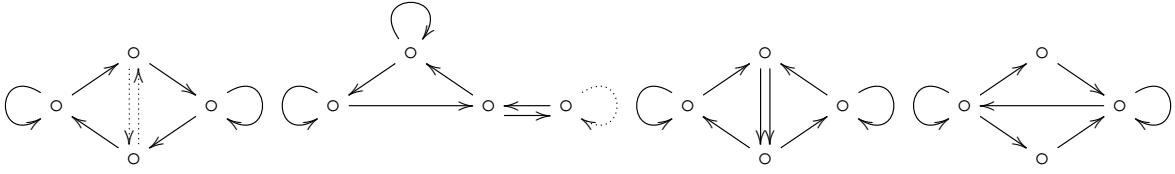
\vspace*{0.5cm} 
$$\xymatrix@R=0.3cm@C=0.6cm{&\circ\ar[rd]\ar@{.>}@<-.37ex>[dd]& \\ 
\ar@(lu,ld)[]\circ\ar[ru] && \circ\ar[ld]\ar@(ru,rd)[]\\ & \circ\ar[lu] \ar@{.>}@<-.37ex>[uu] & } \qquad\quad 
\xymatrix@R=0.3cm@C=0.6cm{& \circ\ar@(lu,ru)[] \ar[ld] && \\ 
\ar@(ul,dl)[] \circ \ar[rr] & & \circ\ar[lu] \ar@<-0.13cm>[r] & \circ \ar[l] \ar@{.>}@(ru,rd)[]} \qquad\quad
\xymatrix@R=0.3cm@C=0.6cm{& \circ \ar@<-.37ex>[dd]\ar@<.37ex>[dd]& \\ 
\ar@(lu,ld)[]\circ\ar[ru] && \circ\ar[lu]\ar@(ru,rd)[]\\ & \circ\ar[lu]\ar[ru]  & } \qquad \quad
\xymatrix@R=0.3cm@C=0.6cm{& \circ \ar[rd] & \\ 
\ar@(lu,ld)[]\circ \ar[ru]\ar[rd] && \circ \ar[ll] \ar@(ru,rd)[]  \\ 
&\circ\ar[ru]& }$$
\end{prop} \medskip 

One can see that the arguments excluding the shadow $\bbQ_5$ can be extended to the following result. 

\begin{prop}\label{blockQ5} If $\La=KQ/I$ is a GQT algebra, $Q=Q_\La$, then its reduced Gabriel quiver $Q^\times$ 
does not admit a block $\Gamma$ of the form $$
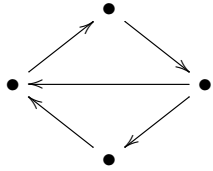
\xymatrix@R=0.6cm{& \bullet \ar[rd] & \\ \bullet \ar[ru]^{} && \ar[ld]^{}\bullet\ar[ll] \\ 
&\bullet\ar[lu]_{}& }$$ \end{prop} 

\begin{proof} Suppose to the contrary that $Q^\times$ contains a block from the statement. Denote by $1,3$ the 
non-regular vertices, and by $2,4$ the $1$-regular vertices of $\Gamma$. One can see from the rule (R4) 
that $Q$ does not admit a $2$-cycle $2\leftrightarrows 4$, since otherwise, we have no loops at vertices 
$2,4$, by Lemma \ref{lem:loops}, and as in the case $\bbQ_5$ for $n=4$, we conclude that $p_4^+=p_4$, which 
contradicts Lemma \ref{lem:3.1}. According to the rule (R3), we obtain that there is no $2$-cycle in $E$ connecting 
vertices from $\Gamma$, and hence $\Gamma$ is a block of $Q^\circ$. Now, observe that $Q$ has no loops at 
vertices $1,3$, due to Corollary \ref{nregloops}. Since $Q$ is connected, we conclude that $\Gamma=Q^\circ$ 
is the whole $Q$, except possible loops at $2$ or $4$, and we obtain a contradiction as in case $\bbQ_5$ for 
$n=4$. \end{proof} 

\subsection{Quivers with 5 vertices}\label{subsec:n=5}

Now assume that $n=5$. Following the list from Example \ref{exm:4.3}, we have to consider $26$ 
essential shadows, identified with quivers $\bbQ_1,\dots,\bbQ_{26}$. We will exclude $16$ of them, 
leaving $10$, from which we next reconstruct all possible Gabriel quivers of GQT algebras with $n=5$ 
vertices. \medskip 

\begin{lemma} If $\La=KQ/I$ is a GQT algebra with $Q=Q_\La$ having $n=5$ vertices, then its shadow 
$\bbQ=Q^\times$ is one of the ten shadows in the following set 
$$\bA=\{\bbQ_4, \bbQ_{11}, \bbQ_{13}, \bbQ_{14}, \bbQ_{16}, 
\bbQ_{17}, \bbQ_{23}, \bbQ_{24}, \bbQ_{25}, \bbQ_{26}\}\subset\bE(5).$$ \end{lemma} 

\begin{proof} Assume that $\La=KQ/I$ is a GQT algebra with shadow $\bbQ=Q^\times$ not in $\bA$. First, we can 
quickly exclude the empty shadow $\bbQ=\bbQ_{15}$, because $n\geqslant 4$ (see Corollary \ref{coro1}). It is also clear that 
$\bbQ$ cannot be $\bbQ_1,\bbQ_2,\bbQ_3$, since then $Q$ admits a forbidden $(1,2)$-vertex (see Corollary \ref{12wD}). 

We shall now present a case-by-case analysis of $Q$, excluding all other shadows from $\bE(5)$ not 
contained in $\bA$. \smallskip 

We claim first that algebras $\La=KQ/I$ with shadow $\bbQ=\bbQ_5,\bbQ_6, \bbQ_{18}$ or $\bbQ_{19}$ are wild. 
Indeed, these shadows contain a vertex $i$ which is a target of three arrows starting at different vertices 
$i_1,i_2,i_3$ ($i=2$ for $\bbQ_5$ and $i=1$ for the remaining cases). Moreover,  $i_1,i_2,i_3$ are $2$-regular in 
$\bbQ=\bbQ_5$ or $\bbQ_6$ and lay on a cycle $i_1 \to i_2 \to i_3 \to i_1$ of length $3$. In this case we have 
the following wild subcategory in covering
$$ 
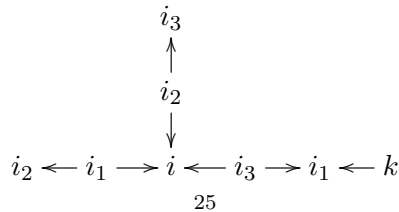
\xymatrix@R=0.45cm@C=0.45cm{&& i_3 &&& \\ && \ar[u] i_2 \ar[d]&&& \\ 
i_2 & \ar[l] i_1 \ar[r] & i & \ar[l] i_3 \ar[r] & i_1 & \ar[l] k }$$ 
for some vertex $k \neq i_3$ in $\bbQ$. Cases $\bbQ=\bbQ_{18}$ or $\bbQ_{19}$ are analogous, we have a wild 
subcategory in covering of the form
$$
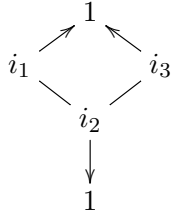
\xymatrix@R=0.18cm@C=0.4cm{& i_3\ar[ld]& &\\ 
1 & & i_2\ar@{-}[ld] \ar@{-}[ul]\ar[r] & 1\\ & i_1\ar[lu] & & }$$
where $\xymatrix@C=0.25cm{i_1 \ar@{-}[r] & i_2 \ar@{-}[r] & i_3 }$ replaces $\xymatrix@C=0.4cm{4 \ar[r] & 2 & 3 \ar[l] }$ 
for $\bbQ_{18}$ and $\xymatrix@C=0.4cm{4  & 3 \ar[r] \ar[l] & 2 }$ for $\bbQ_{19}$. Note that in both cases, 
any path of length $2$ is not involved in a minimal relation of $I$, because of the shape of $Q^\circ$ and 
Lemma \ref{lem:3.3} (see also Remark \ref{noedge}). \smallskip 

Note here that for any $j\neq 19,22$, the quiver $Q$ with  $\bbQ=\bbQ_j\notin\bA$ satisfies $E=\emptyset$, since 
there is no admissible configuration of vertices satisfying the rules (R1)-(R4) in all these cases. In particular, 
it means that then $Q$ is obtained from $Q^\circ=\bbQ$ by attaching loops. \medskip 

Assume now that $\La=KQ/I$ has $\bbQ_7$ as a shadow. It follows from Corollary \ref{nregloops} that there are no 
loops at $1$ and $5$, hence an arrow $1\to 5$, call it $\alpha$, satisfies $1^+=\{\alpha\}=5^-$, which gives 
a contradiction with Lemma \ref{lem:3.2}. The same arguments work for $\bbQ=\bbQ_{20}$. In case $\bbQ_{21}$, 
we have no loops at $1$ and $5$, due to Remark \ref{loops:rk0}. \smallskip 

Next, consider algebras with shadow $\bbQ=\bbQ_8$ or $\bbQ_9$. We show that in this case $Q$ is biserial. 
Indeed, it follows from Corollary \ref{nregloops}, that in both cases there are no loops at vertices $1$ and $3$ 
(i.e. the unique non-regular vertices in $\bbQ$), and there is no loop  in $Q$ at $2$-regular vertex $2$ of 
$\bbQ_8$, by Lemma \ref{lem:loops}. Nevertheless, there is a loop in $Q$ at each $1$-regular vertex of $\bbQ_8$, and 
in case $\bbQ_9$, there is a loop at $5$ and at least one loop at $4$ or $2$. Otherwise, we would get an arrow 
forbidden in Lemma \ref{lem:3.2}. As a result, in both cases the quiver $Q$ is biserial, but some of the 
non-regular vertices does not satisfy the property from Proposition \ref{prop:biser} (i.e. a $(1,2)$-vertex 
has no $1$-regular successors or a $(2,1)$-vertex has no $1$-regular predecessors). \smallskip 

Next, consider $\La$ with shadows $\bbQ_{10}$ or $\bbQ_{12}$. Following Lemma \ref{lem:3.2}, there must be a 
loop in $Q$ at each vertex of $Q^{\circ}=\bbQ_{12}$ (then $p_1=\dots=p_5$), whereas for $Q^{\circ}=\bbQ_{10}$, 
there are no loops in $Q$, by Lemma \ref{lem:loops}. It follows that $Q$ is a $2$-regular quiver in both cases. 
Now, applying \cite[Main Theorem]{AGQT}, we conclude that $\La$ is a weighted surface algebra. In particular, there 
is a permutation $f$ on the set of arrows $Q_1$, such that $(Q,f)$ is a triangulation quiver. In other words, 
$Q$ is obtained by glueing finite number of cycles of length $3$ or loops. But this is impossible, since $\bbQ_{10}$ 
has no loops and $10$ arrows and $\bbQ_{12}$ has a cycle of length $5$. \smallskip 

Eventually, it remains to exclude $\bbQ=\bbQ_{22}$, because it contains a block forbidden in Proposition 
\ref{blockQ5}, and the proof is now complete. \end{proof} \medskip 

Now, we will reconstruct all possible Gabriel quivers of GQT algebras, whose shadows are the 
remaining ten shadows in $\bA$. Namely, we prove the following proposition. 

\begin{prop}\label{prop:Ga5} If $\La$ is a GQT algebra such that $Q=Q_\La$ has $n=5$ vertices, then $Q$ is one of 
the following $19$ quivers (up to permutation of vertices and taking $Q^{\op}$).  
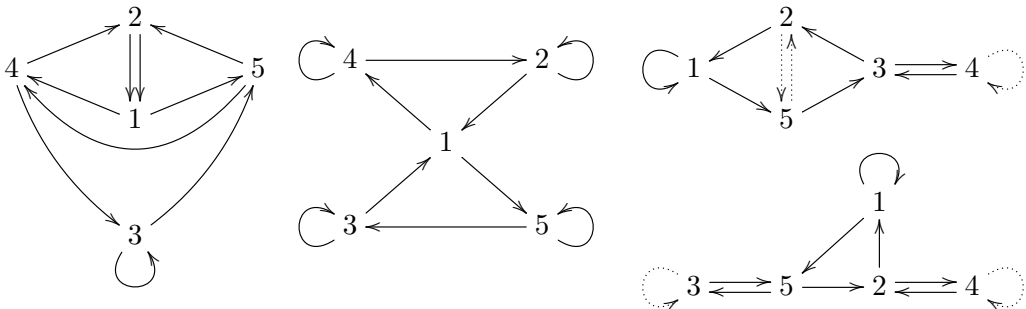
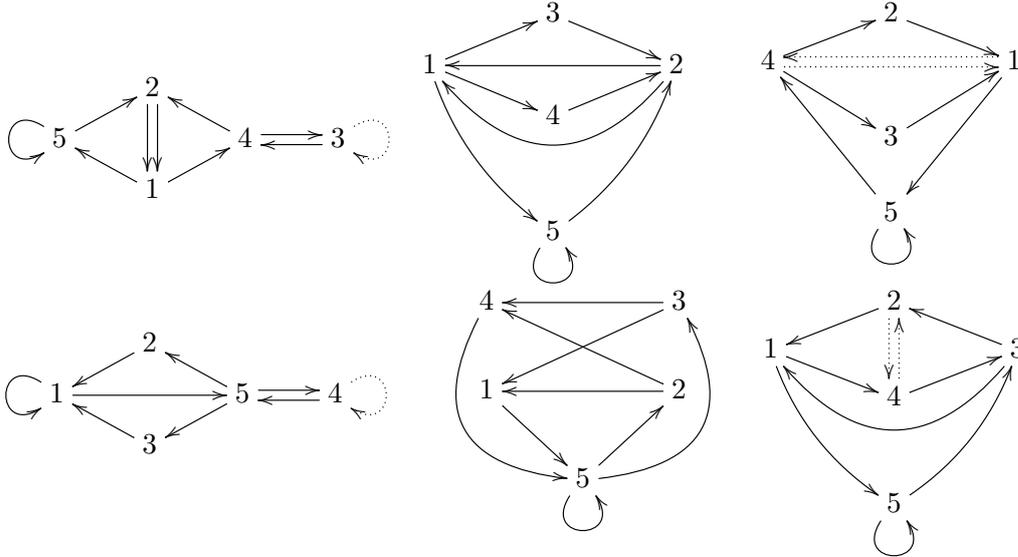
$$\xymatrix@R=0.2cm@C=1.2cm{&2\ar@<-0.4ex>[dd]\ar@<+0.4ex>[dd]& \\ 4\ar[ru] \ar@/_7pt/[rdddd] && 5 \ar[lu] \ar@/^30pt/[ll]\\ 
& 1 \ar[lu] \ar[ru] &  \\ && \\ && \\ & 3 \ar@(ld,rd)[] \ar@/_7pt/[ruuuu] & } \qquad 
\xymatrix@R=0.2cm{&&\\ \ar@(dl,ul)[]4\ar[rr]&& 2\ar[ldd]\ar@(dr,ur)[]\\ && \\ &1\ar[rdd]\ar[luu]& \\ && \\ 
\ar@(dl,ul)[]3\ar[ruu] && 5\ar[ll]\ar@(dr,ur)[] } \qquad \qquad 
\xymatrix@R=0.2cm@C=0.8cm{ &2\ar[ld] \ar@{.>}@<-0.4ex>[dd]&& \\ 
\ar@(ul,dl)[] 1 \ar[rd] &&3\ar[lu]\ar@<+0.4ex>[r] & \ar@<+0.4ex>[l] 4 \ar@{.>}@(ur,dr)[] \\ 
&5 \ar[ru]\ar@{.>}@<-0.4ex>[uu]&& \\ &&&\\ && 1 \ar[ldd]\ar@(lu,ru)[] & \\ &&& \\  
\ar@{.>}@(ul,dl)[] 3 \ar@<+0.4ex>[r] & \ar@<+0.4ex>[l] 5 \ar[r] & 
2 \ar[uu] \ar@<+0.4ex>[r] & \ar@<+0.4ex>[l] 4 \ar@{.>}@(ur,dr)[] \\  }$$ 
$$\xymatrix@R=0.2cm@C=0.8cm{&&& \\ &&& \\ &2 \ar@<-0.4ex>[dd]\ar@<+0.4ex>[dd] && \\ 
\ar@(ul,dl)[] 5 \ar[ru] && 4 \ar[lu]\ar@<+0.4ex>[r] & \ar@<+0.4ex>[l] 3 \ar@{.>}@(ur,dr)[] \\ 
& 1 \ar[lu]\ar[ru]&& } \qquad 
\xymatrix@R=0.2cm@C=1.2cm{&3 \ar[rd]& \\ 
1\ar[rd]\ar[ru] \ar@/_7pt/[rdddd] && 2 \ar@/^30pt/[ll] \ar[ll]\\ 
& 4 \ar[ru] &  \\ && \\ && \\ & 5 \ar@(ld,rd)[] \ar@/_7pt/[ruuuu] & } \qquad 
\xymatrix@R=0.15cm@C=1.2cm{&2 \ar[rd]& \\ 
4\ar[rdd]\ar[ru] \ar@<-0.4ex>@{.>}[rr]  && 1 \ar[ldddd]\ar@<-0.4ex>@{.>}[ll] \\ 
&& \\ & 3 \ar[ruu] &  \\  && \\ & 5 \ar@(ld,rd)[] \ar[luuuu] & }\vspace*{0.3cm}$$ 
$$\xymatrix@R=0.2cm@C=0.8cm{ &&& \\ &2\ar[ld] && \\ 
\ar@(ul,dl)[] 1 \ar[rr] && 5 \ar[lu]\ar@<+0.4ex>[r] \ar[ld] & \ar@<+0.4ex>[l] 4 \ar@{.>}@(ur,dr)[] \\ 
& 3 \ar[lu]&& } \qquad \qquad
\xymatrix@R=0.7cm{ 4 \ar@/_30pt/[rdd] && \ar[ll] 3 \ar[lld] \\ 
1 \ar[rd] && 2 \ar[ll] \ar[llu] \\ 
& 5 \ar@(ld,rd)[] \ar[ru] \ar@/_30pt/[ruu] & } \qquad 
\xymatrix@R=0.16cm@C=1.2cm{& 2 \ar[ld] \ar@{.>}@<-0.4ex>[dd] & \\ 
1 \ar[rd] \ar@/_7pt/[rdddd] && 3 \ar[lu] \ar@/^30pt/[ll]\\ & 4  \ar[ru] \ar@{.>}@<-0.4ex>[uu] &  \\ && 
\\ && \\ & 5 \ar@(ld,rd)[] \ar@/_7pt/[ruuuu] & } \vspace*{0.4cm}$$ 
The dotted loops or $2$-cycles encode few possible quivers, with or without dotted arrows. 
\end{prop} 

\begin{proof} Fix a GQT algebra with $Q=Q_\La$ having $n=5$ vertices, and let $\bbQ=\bfQ_{\bS_\La}$ be 
the quiver identified with its shadow. It follows from the previous lemma that $\bbQ=Q^\times$ 
(up to permutation or taking the opposite quiver) is one of the quivers in $\bA$. We denote by 
$Q^{(1)},\dots,Q^{(19)}$, the 19 quivers obtained from the 10 quivers in the statement by adding 
or deleting dotted loops or $2$-cycles, ordered with respect to the number of added loops and $2$-cycles,  
where quivers with a single loop added come before those enlarged with a single $2$-cycle. \medskip 

In all cases we follow general strategy: we first describe where $2$-cycles can appear, 
i.e. we determine $2$-cycles in $E$ with $Q^\circ=\bbQ\sqcup E$, and then we describe 
which vertices of $Q^\circ$ admit a loop in $Q$ (or may admit). \medskip 

Suppose first that $\bbQ=\bbQ_4$. Then $\bbQ$ is a biregular quiver (with unique $1$-regular vertex 
$3$) of the form 
$$
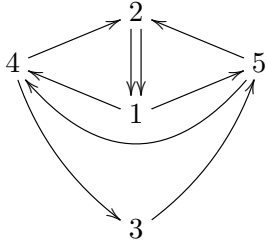
\xymatrix@R=0.2cm@C=1.2cm{&2\ar@<-0.4ex>[dd]\ar@<+0.4ex>[dd]& \\ 4\ar[ru] \ar@/_7pt/[rdddd] && 5 \ar[lu] \ar@/^30pt/[ll]\\ 
& 1 \ar[lu] \ar[ru] &  \\ && \\ && \\ & 3 \ar@/_7pt/[ruuuu] & } $$ 
In this case, it follows from the rules of Theorem \ref{thm:recon}, that $E$ is empty, or 
equivalently, $Q$ has no $2$-cycles. It means that $Q$ is obtained from the quiver $\bbQ=Q^\circ$ 
by adding loops. We know from Lemma \ref{lem:loops} that there are no loops at $2$-vertices 
of $\bbQ$, hence we can add at most one loop at vertex $3$. Now, it is sufficient to see 
that there must be a loop at $3$, because otherwise, $Q$ is a biregular quiver with a 
$1$-vertex which does not belong to a required block (see Theorem \ref{bireg1}). Therefore, 
we proved that there is only one $Q$ with $\bbQ=\bbQ_4$, and it is $Q^{(1)}$. \medskip 

Next, we claim that the unique $Q$ with $\bbQ=\bbQ_{11}$ is the quiver $Q^{(2)}$ obtained 
from $\bbQ=Q^\circ$  
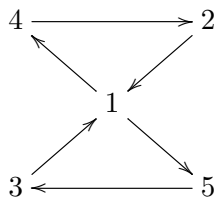
$$\xymatrix@R=0.2cm{ 4\ar[rr]&& 2\ar[ldd] \\ && \\ &1\ar[rdd]\ar[luu]& \\ && \\ 
3\ar[ruu] && 5\ar[ll] }$$ 
by adding four loops at all $1$-vertices. Indeed, it is easy to check that $E=\emptyset$, by the rule (R4), so 
$Q$ is obtained from $Q^\circ=\bbQ$ by adding up to five loops. First, observe that there 
must be a loop at each $1$-vertex of $\bbQ=Q^\circ$. Indeed, all dimension vectors of indecomposable projective 
$\La$-modules are equal $p_1=\dots=p_5$ (because $p_i^-=p_i^+$, for all $1$-vertices $i$ in $Q$), and 
every $1$-vertex $i$ admits $1$ as a predecessor or successor. But then, for all $i\neq 1$, we have 
$\hat{p}_i=p_1+\dots=p_i+\dots \neq p_i$, by Lemma \ref{lem:3.1}, so there must be at least one more 
arrow in $Q$ starting from $i$ and different from an arrow in $\bbQ$. This can be only a loop, which shows the claim. \smallskip 

Now, it remains to see that each successor/predecessor of vertex $1$ is $2$-regular in $Q$ (admits a loop), 
so it follows from Lemma \ref{lem:loops} that there is no loop at vertex $1$.  

\medskip 

Assume now that $\bbQ=\bbQ_{13}$ is of the form  
$$
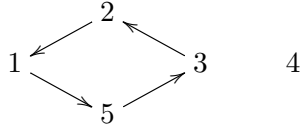
\xymatrix@R=0.2cm@C=0.8cm{ &2\ar[ld] && \\  1 \ar[rd] &&3\ar[lu] &  4  \\ &5 \ar[ru]&&}$$ 
Here, we must have at least one $2$-cycle in $E$ connecting the isolated vertex $4$ with one of the 
remaining $1$-vertices in the square, because $Q$ is connected by the assumption. Up to permutation, 
we can assume there is a $2$-cycle $3\leftrightarrows 4$ in $E$. Moreover, 
there may be at most one additional $2$-cycle in $E$ connecting vertices $2$ and $5$, because this is 
the only possibility allowed by rule (R4). As a result, we proved that $Q^\circ=\bbQ\sqcup E$ has the 
following form 
$$
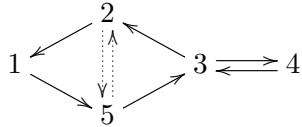
\xymatrix@R=0.2cm@C=0.8cm{ &2\ar[ld] \ar@{.>}@<-0.4ex>[dd]&& \\ 
 1 \ar[rd] &&3\ar[lu]\ar@<+0.4ex>[r] & \ar@<+0.4ex>[l] 4 \\ 
&5 \ar[ru]\ar@{.>}@<-0.4ex>[uu]&&}$$ 
which give two possible quivers, depending on the second $2$-cycle. \smallskip 

If the number $e$ of $2$-cycles in $E$ is $2$, then using Lemma \ref{lem:loops} we infer that there are no 
loops at vertices $2,3$ and $5$. Moreover, then $Q$ is a biregular quiver obtained from $Q^\circ$ by adding 
at most two loops (at vertex $1$ or $4$). Note finally, that there must be a loop at vertex $1$, since 
otherwise $Q$ admits a $1$-vertex not lying in a block (of $Q$) required by Theorem \ref{bireg1} 
($Q$ is biregular). \smallskip 

Further, let $e=1$, that is, $Q^\circ$ contains one $2$-cycle. In this case, we have no loop at 
vertex $3$, by Corollary \ref{rmkloops} (path of length $2$ passing through a $1$-vertex is not involved 
in a minimal relation, by Lemma \ref{lem:3.3}). Hence 
$Q$ is a biregular quiver, and we can use Theorem \ref{bireg1} again, to get a loop at vertex $1$. 
The same argument shows that there cannot be exactly one loop at one of vertices $2$ or $5$. 
Consequently, there are either no loops in $Q$ at vertices $2$ and $5$, or we have two loops. 
The latter is impossible, because then $Q$ is $2$-regular, but not a triangulation quiver 
(see \cite[Theorem 6.1]{AGQT}). Therefore, we have proven that $Q$ has the required form 
$$
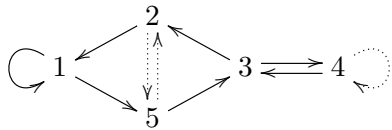
\xymatrix@R=0.2cm@C=0.8cm{ & 2 \ar[ld] \ar@{.>}@<-0.4ex>[dd]&& \\ 
\ar@(ul,dl)[] 1  \ar[rd] && 3 \ar[lu]\ar@<+0.4ex>[r] & \ar@<+0.4ex>[l] 4 \ar@{.>}@(ur,dr)[] \\ 
& 5 \ar[ru]\ar@{.>}@<-0.4ex>[uu]&&}$$ 
which gives the next four quivers $Q^{(3)},\dots,Q^{(6)}$. These are all possible Gabriel quivers 
of GQT algebras with shadow $\bbQ_{13}$.  \medskip 

Consider next the case $\bbQ=\bbQ_{14}$, when $\bbQ$ is a disjoint union of a triangle (supported 
on vertices $1,2,5$) and two isolated vertices $3,4$. Since $Q$ is connected, we conclude from 
Theorem \ref{thm:recon} that $Q^\circ$ is obtained from $\bbQ$ by adding two $2$-cycles connecting the 
isolated vertices, so up to permutation, we can assume that $Q^\circ$ has the form 
$$
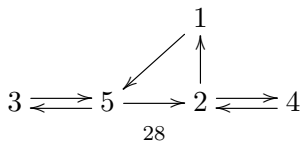
\xymatrix@R=0.2cm@C=0.8cm{&&1\ar[ldd] & \\ &&& \\  
3 \ar@<+0.4ex>[r] & \ar@<+0.4ex>[l] 5 \ar[r] & 
2\ar[uu] \ar@<+0.4ex>[r] & \ar@<+0.4ex>[l] 4 \\  }$$  
It remains to see that there must be a loop at vertex $1$. Otherwise, we have $p_1^+=p_5$, but the identity 
$p_2^-=p_2^+$ gives $p_1=p_5$, so we would get $p_1^+=p_1$, which contradicts Lemma \ref{lem:3.1}. This 
proves that $Q$ is one of the three quivers represented by 
$$
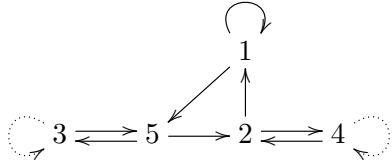
\xymatrix@R=0.2cm@C=0.8cm{&&&\\ && 1\ar[ldd]\ar@(lu,ru)[] & \\ &&& \\  
\ar@{.>}@(ul,dl)[] 3 \ar@<+0.4ex>[r] & \ar@<+0.4ex>[l] 5 \ar[r] & 
2 \ar[uu] \ar@<+0.4ex>[r] & \ar@<+0.4ex>[l] 4 \ar@{.>}@(ur,dr)[] \\  } \vspace*{0.2cm}$$ 
These are the quivers $Q^{(7)},Q^{(8)},Q^{(9)}$ with shadow $\bbQ_{14}$. Note that we do not need to 
consider two quivers with one loop as different cases, since they are opposite to each other. \medskip 

In the next step, suppose $Q$ has shadow $\bbQ=\bbQ_{16}$ 
$$
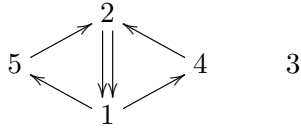
\xymatrix@R=0.2cm@C=0.8cm{ &2 \ar@<-0.4ex>[dd]\ar@<+0.4ex>[dd] && \\ 
 5 \ar[ru] &&4\ar[lu] & 3 \\ 
& 1 \ar[lu]\ar[ru]&& } $$ 
Then the isolated vertex $3$ must be connected to one of $1$-regular vertices, according to the rule (R1), and 
$Q^\circ$ has no more $2$-cycles, by (R4). We may assume that there is a $2$-cycle 
$3\leftrightarrows 4$, and then $Q$ is biregular quiver without loops at vertices $1,2,4$ (see also 
Lemma \ref{lem:loops}). Then, due to Theorem \ref{bireg1}, vertex $5$ cannot be a $1$-vertex (in $Q$), 
hence we have a loop at $5$, and therefore, quiver $Q$ is one of the quivers $Q^{(10)},Q^{(11)}$, given as follows 
$$
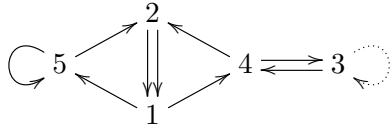
\xymatrix@R=0.2cm@C=0.8cm{ & 2 \ar@<-0.4ex>[dd]\ar@<+0.4ex>[dd] && \\ 
\ar@(ul,dl)[] 5 \ar[ru] && 4 \ar[lu]\ar@<+0.4ex>[r] & \ar@<+0.4ex>[l] 3 \ar@{.>}@(ur,dr)[] \\ 
& 1 \ar[lu]\ar[ru]&& } $$ 

Now, observe that $\bbQ_{17}$ is a glueing of a triangle $(5\to 2\to 1\to 5)$ with a block of type IV 
(consisting of the remaining arrows). It has no block 
of type V$_2$ and $|1^+|=|2^-|=3$, hence $Q$ with $\bbQ=\bbQ_{17}$ is the quiver obtained from $\bbQ=Q^\circ$  
by adding at most three loops at vertices $3,4$ or $5$. Moreover, we cannot have more than one loop, say at $4$ 
and $5$, since then we would get a wild subcategory of the form 
$$
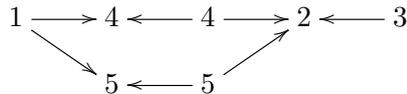
\xymatrix@R=0.4cm{1 \ar[r]\ar[rd] & 4 & \ar[l] 4 \ar[r] & 2 & \ar[l] 3 \\ & 5 & \ar[l] 5 \ar[ru] &&}$$ 
Therefore, $Q$ is obtained from $\bbQ$, by adding at most one loop, say at $5$. It will follow from the proof of 
Proposition \ref{prop:6.1} that there must be a loop, so $Q=Q^{(12)}$, if $Q$ is a Gabriel quiver of a TSP4 algebra, 
hence we skip the arguments here. For GQT algebras we only know that $Q=Q^{(12)}$, or $\bbQ_{17}$, which 
is $Q^{(12)}$ without the loop. \medskip 

Further, consider quivers $Q$ with shadow $\bbQ=\bbQ_{23}$. If $Q^\circ=\bbQ\sqcup E$ contains a $2$-cycle 
connecting the $(1,2)$- and the $(2,1)$-vertex of $\bbQ$, then it is the unique $2$-cycle and $Q^\circ$ is a block with unique 
outlet $\circ=5$, by The Reconstruction Theorem, part (b) (see also rule (R3) after). In this case, $Q=Q^{(14)}$ 
modulo possible loop at vertex $5$. If there is no $2$-cycle between the non-regular vertices in $\bbQ$, then 
$Q^\circ=\bbQ$, i.e. we have no $2$-cycles. Indeed, otherwise Theorem \ref{thm:recon} implies that there is a 
unique $2$-cycle connecting a pair of $1$-vertices forming a block of type V$_1$ in $\bbQ$, say $3 \leftrightarrows 5$. 
In this case, paths of length $2$ passing through vertices $2$ and $3$ are not involved in minimal relations 
(by Lemma \ref{lem:3.3}), hence it follows from Corollary \ref{nregloops}, that we cannot have a loop in $Q$ at 
vertices $1,4$. Applying Lemma \ref{lem:loops}, we conclude that there are no loops at vertices $3,5$, so $Q$ has 
at most one loop at vertex $2$. But then $Q$ is biserial, so both non-regular vertices $1,4$ are of type N, 
due to Proposition \ref{prop:typeN}. This gives the following wild subcategory 
$$\xymatrix@R=0.4cm{&&2&&&2& \\ 3 &\ar[l] 5 \ar[r] & 4 \ar[u]\ar[r] & 3 & \ar[l] 5\ar[r] & 4\ar[u]\ar[r]& 3}$$ 

As a result, if $Q$ has shadow $\bbQ=\bbQ_{23}$, then it is one of the quivers $Q^{(13)}$ or $Q^{(14)}$, 
except the loop. As in case of $Q^{(12)}$, the existence of a loop at vertex $5$ follows from the arguments presented 
in the proofs of Propositions \ref{prop:6.1} and \ref{prop:6.2}. \medskip 

Finally, it remains to consider the quivers $Q$ with shadow $\bbQ=\bbQ_{24},\bbQ_{25}$ or $\bbQ_{26}$. 

Let $\bbQ=\bbQ_{24}$. We have no block of type V$_2$ in $\bbQ$, so by Theorem \ref{thm:recon}(b), we have exactly one 
$2$-cycle in $Q$ connecting the isolated vertex $4$ of $\bbQ$ with one of the remaining vertices. We claim that 
the $2$-cycle connects vertex $4$ with one of the non-regular vertices $1$ or $5$. Indeed, we cannot have a 
$2$-cycle connecting vertex $4$ with one of the $1$-vertices of $\bbQ$, say $4\leftrightarrows 3$, since then 
$Q^\circ$ is biserial, and we have no loops at non-regular vertices of $Q^\circ$, due to Corollary \ref{nregloops} 
($3$ is a $2$-regular predecessor of $1$ and successor of $5$). It follows that $Q$ is obtained from 
$Q^\circ=\bbQ\sqcup\{4\leftrightarrows 3\}$ by adding at most two loops at vertices $4$ or $2$ (no loop at $3$ 
by Corollary \ref{rmkloops}), and then also $Q$ is biserial. As a result, it follows from Proposition \ref{prop:typeN} 
that $1$ and $5$ are vertices of type N, and hence we have the following wild subcategory in covering 
$$\xymatrix@R=0.4cm{& 2 &&&&2\ar[d]&& \\ 1\ar[r] & 5 \ar[u]\ar[r] & 3 & \ar[l] 4 & \ar[l] 3 \ar[r] & 1\ar[r] & 5 \ar[r] & 3}$$ 
Note that the paths $3\to 4\to 3$, $3\to 1\to 5\to 3$ and $2\to 1\to 5 \to 3$ are not involved in minimal relations, 
by Lemmas \ref{lem:3.3} and \cite[Lemma 4.5]{EHS1}, since there are no loops at $3$ and no arrows $3\to 2$. Therefore, 
we have a $2$-cycle $4\leftrightarrows i$, for $i=1$ or $5$. The first case gives an opposite quiver $Q^\circ$ (modulo 
permutation), so we can restrict to the case $i=5$. Then the quiver $Q$ is one of $Q^{(15)}$ or $Q^{(16)}$. To see this, 
it suffices to show that there are no loops at vertices $2$ and $3$, but there is a loop at vertex $1$. Indeed, a loop at vertex 
$3$ induces the following wild subcategory 
$$\xymatrix@R=0.4cm{ &&&& 2&&& \\ 2\ar[r] & 1 & \ar[l] 3\ar[r] & 3 & \ar[l] 5 \ar[u] \ar[r] & 4 \ar[r] & 5 &\ar[l] 1}$$ 
Note that the path $5\to 4\to 5$ is not involved in a minimal relation, by Lemma \ref{lem:3.3}, because there is 
no loop at vertex $5$ (it is not at most $2$-vertex of $Q^\circ$; see Remark \ref{loops:rk0}). Similarly, we can 
construct a wild subcategory of the same type, if there is a loop at vertex $2$. Now, observe that there must be 
a loop at vertex $1$, because otherwise, we get $p_1^+=p_5=p_1$, which contradicts Lemma \ref{lem:3.1}. \medskip 

Observe finally that the unique $Q$ with $\bbQ=\bbQ_{25}$ is the quiver $Q^{(17)}$. Indeed, by The Reconstruction 
Theorem, we cannot have any $2$-cycle in $Q$, so $Q^\circ=\bbQ$. We have no loops at non-regular vertices of 
$Q^\circ$, by Corollary \ref{nregloops}, and there must be a loop at vertex $5$, because otherwise, we 
obtain $p_5^+=p_2+p_3=p_1^-=p_1^+=p_5$, a contradiction with Lemma \ref{lem:3.1}. \medskip 

In the last case, we consider quivers $Q$ having $\bbQ=\bbQ_{26}$. Then we can have at most one $2$-cycle 
in $Q$ connecting a pair of $1$-vertices in a block of type V$_2$. Up to permutation, we can assume that there 
is at most one $2$-cycle of the form $4\leftrightarrows 2$, by the rule (R4). Moreover, vertices $1,3$ do not admit a loop, 
due to Lemma \ref{lem:loops}. If there exists a $2$-cycle, then there are no loops at vertices $4,2$, by  
Lemma \ref{lem:loops} again, and there is a loop at vertex $5$, by Theorem \ref{bireg1} (then $Q$ is biregular). Suppose 
eventually that there is no $2$-cycle in $Q$. Then $Q$ is biregular, and it has at most one loop, because otherwise 
we get a $1$-vertex not contained in a block required in Theorem \ref{bireg1}. We cannot have a loop at vertex $2$, 
because then again vertices $4,5$ are $1$-vertices in configuration not allowed by Theorem \ref{bireg1}. Thus we 
have at most one loop at vertices $4$ or $5$, which gives two isomorphic quivers. Depending on the $2$-cycle, 
we obtain the remaining two quivers $Q=Q^{(18)}$ or $Q^{(19)}$, and the proof is now finished. \end{proof}

\section{Algebra structures}\label{sec:algstr} Now, we will describe possible TSP4 algebras on quivers computed in 
Section \ref{sec:6}. First two subsections deal with cases $n=1$ and $2$, all other cases are covered in the 
last subsection. We will show that, for $n\geqslant 3$, all possible algebras except one case are weighted surface 
algebras or their virtual mutations. For one quiver $Q=Q^{(17)}$ in case $n=5$, the algebra given by $Q$ is the 
generalized weighted surface algebra; see Section \ref{subs:3.2}. \medskip

\subsection{Local TSP4 algebras}\label{subs:6.1} A classification of local tame symmetric algebras is given in 
\cite[Theorem III.1]{E90}, where we have 9 families, denoted by (a),(b),(b'),$\dots$,(e),(e'). \smallskip 

The algebras in (e) and (e') have period four. One can easily construct a minimal projective resolution of the simple module. 

All other algebras are either special biserial, or otherwise are hybrid algebras (described as 'of semidihedral type' in 
\cite{E90}). For these algebras we may apply the constructions in section 3 of \cite{EKM}. Namely, \cite[Lemma 3.5]{EKM} 
applies to a local special biserial algebra and shows that the simple module is not periodic. Then by \cite[Lemma 3.7]{EKM}, 
the simple module of a local hybrid algebra is not periodic.

\subsection{TSP4 algebras with two simple modules}\label{subs:6.2} 

We will show that the TSP4 algebras with two simple modules 
are precisely the algebras of quaternion type, described in Section VII.7 of \cite{E90}. Since the proofs in \cite{E90} 
assumed that the Cartan matrix of the algebra is non-singular, we will give a sketch for the proofs. The answer is then

\begin{cor} The following are equivalent for a symmetric algebra $\La$ with two simple modules:\\
(a) $\La$ is a TSP4 algebra, \\
(b) $\La$ is of generalized quaternion type, and is one of the algebras $Q(2\cA)$ or $Q(2\cB)_i$  ($i=1, 2, 3$). 
\end{cor}

We note that in \cite{Holm} Thorsten  Holm gives a refined parametrisation for these algebras, and he determined the   derived equivalence classification for these algebras.

\subsubsection{The possible quivers}

Assume $\La$ is a TSP4 algebra with two simple modules,  then $\La  = KQ/I$ where $Q$ is a connected quiver with two vertices. There are four possible quivers, they are determined in  IV.2.2 of \cite{E90}.

They are $Q(2\cA)$ and $Q(2\cB)$  which we will display  below. 
The other two are the quivers with
(i) two pairs of double arrows, and 
(ii) one double arrow $1\to 2$ and a single arrow $2\to 1$.
We will exclude these now.

(1)  Consider the  quiver (i) with two pairs of double arrows
$\alpha_1, \alpha_2: 1\to 2$ and $\beta_1, \beta_2: 2\to 1$. 
One may note  that this quiver occurs for a tame symmetric algebra with radical cube zero.
However, it  cannot occur for a TSP4 algebra. Recall that there are natural identifications 
$\Omega(S_i)= {\rm rad} P_i = \alpha_1\La + \alpha_2\La$ and $\Omega^{-1}(S_i)\cong (\beta_1, \beta_2)\La$. 

Assume for a contradiction that there is a TSP4 algebra with this quiver. We consider the exact sequence for the simple module $S_1$,
$$0\to (\beta_1, \beta_2)\La \to P_2\oplus P_2 \to P_2\oplus P_2 \to \alpha_1\La + \alpha_2\La\to 0.$$
This shows that the minimal relations starting with $\alpha_i$ are of the form 
$\alpha_1X_1 + \alpha_2X_2=0$ and 
$X_i\in \La e_2$ (see Proposition 4.3 in [Alg. of generalized quaternion type=17]. Therefore paths of length two $\alpha_i\beta_j$ are $\nprec I$ and $\La$ is wild, a contradiction.

(2) Assume  $\La$ is a TSP4 algebra where  the quiver has  double arrows $\alpha_1, \alpha_2: 1\to 2$ and one arrow $\beta: 2\to 1$. Using again the exact sequence for $S_1$, the argument as
in (1) shows that the paths $\alpha_i\beta$ are $\nprec I$ and the algebra is wild.

\subsubsection{TSP4 Algebras with quiver $Q(2\cB)$ }  

Assume $\La=KQ/I$ is a TSP4 algebra where $Q$ is the quiver $Q(2\cB)$, that is   
\[
% \xymatrix{
%  \bullet \ar@(dl,ul)[]^{\alpha} \ar@<+.5ex>[r]^{\sigma}
%     \save[] +<0pc,3mm> *{1} \restore
%   & \bullet \ar@<+.5ex>[l]^{\gamma} \ar@(ur,dr)[]^{\beta}
%     \save[] +<0pc,3mm> *{2} \restore
% }
 \xymatrix{
  0 \ar@(dl,ul)[]^{\alpha} \ar@<+.5ex>[r]^{\beta}
   & 1 \ar@<+.5ex>[l]^{\gamma} \ar@(ur,dr)[]^{\eta}
 }
\]

We sketch the strategy used in \cite{E90}. It 
 uses the structure of the local algebra $\La_0:= e_0\La e_0$
at vertex $0$.
Theorem VII.7.2 shows that if $\La_0$ is not uniserial then the algebra
is a WSA, with presentation given in $Q(2\cB)_1$. In particular, 
also the local algebra at vertex $1$ is not uniserial.

This leaves  to consider TSP4 algebras with quiver $Q(2\cB)$ where both local algebras are uniserial. In this case it turns out that
there are many commutativity relations, and the algebra is one of $Q(2\cB)_i$ for $i=1$ or $2$. The details for these are done in VII.7 of \cite{E90} . Both VII.7.2 and VII.7.3 do not use the Cartan matrix, in fact the details are
very similar to those used in \cite{EHS3}. That is one starts with identifying minimal relations of paths of length two, 
using the exact sequences for the simple modules,
then identify bases for the radical quotients, and then getting the presentation. We  refer to \cite{E90} for details.

\subsubsection{TSP4 algebras with quiver  $Q(2\cA)$}

Assume $\La = KQ/I$ is a TSP4 algebra where $Q$ is the quiver  $Q(2\cA)$, that is
\[
% \xymatrix{
%  \bullet \ar@(dl,ul)[]^{\alpha} \ar@<+.5ex>[r]^{\sigma}
%     \save[] +<0pc,3mm> *{1} \restore
%   & \bullet \ar@<+.5ex>[l]^{\gamma} \ar@(ur,dr)[]^{\beta}
%     \save[] +<0pc,3mm> *{2} \restore
% }
 \xymatrix{
  1 \ar@(dl,ul)[]^{\alpha} \ar@<+.5ex>[r]^{\beta}
   & 2 \ar@<+.5ex>[l]^{\gamma} 
 }
\]
To classify such algebras, we need again to find the minimal relations involving paths of length $2$ or $3$.
First,  consider the exact sequence for $S_2$, that is
$$0\to \beta\La \to P_1\to P_1\to \gamma\La \to 0$$
This shows that there is a unique minimal relation
$\gamma\vf=0$ with $\vf\in e_1\La e_1$ and $\vf \beta=0$. In particular
we do not have a minimal relation involving $\gamma\beta$.
The following appears in \cite{E90} without details, we include a proof since it is  (indirectly) relevant for 
\cite{EHS3}, see the remark below.

\bigskip

\begin{lemma} (a) If $\gamma\alpha \prec I$ then $\La$ is of finite type.\\
        (b) If $\beta\gamma \prec I$ then $\La$ is of finite type.
\end{lemma}

\begin{proof} \  Part (b) follows from part (a). Namely,
suppose  $\beta\gamma \prec I$. Then by Lemma 4.3, we deduce $\gamma\alpha\prec I$.

(a) Assume now  $\gamma\alpha \prec I$.
Then $\alpha$ must be a term of an element $\vf$ as above, say
$\vf = \alpha + X$
and $X\in e_1J^2e_1$. We may assume that $X=0$, by replacing $\alpha$ by
$\alpha'=\alpha+X$. This means that
$$\gamma\alpha=0, \ \ \alpha\beta =0.
$$
In other words, the algebra $\La$ is special biserial.

 We claim that $\alpha\La \cap \beta\La = {\rm soc}(e_1\La)$: 
One inclusion is clear. For the other direction, let $\beta x = \alpha y$, then $\alpha\beta x=0$ and $\gamma\beta x = \gamma\alpha y =0$, which shows
$J\cdot \beta x=0$ and $\beta x$ is in the left socle, ie in the socle.
We assume $\beta\gamma \prec I$, so we have
$\beta\gamma = \alpha z_1 + \beta z_2$
with $z_2 \in e_2J^2e_1$.
Then
$\beta (\gamma - z_2)$ is in the socle of $e_1\La$. We may replace $\gamma$ and
get $\beta\gamma$ is in the socle.

\bigskip

We can now write down a basis for $e_1\La$ compatible with the radical quotients, that is
$\{ e_1, \alpha, \beta, \alpha^2,  
\ldots, \alpha^m \}$
and $\beta\gamma = \alpha^m$ (possibly after re-scaling).
Recall that the Cartan matrix has entries  $c_{ij}=\dim e_i\La e_j$, and it is
	symmetric.
The basis shows that 
	$c_{12} = 1$ and therefore $c_{21}=1$. It follows that
$e_2\La$ has basis $\{ e_2, \gamma, \gamma\beta\}$.
Moreover $\gamma\beta\gamma=0$ follows.
We have now obtained a presentation for the algebra, as
$KQ/I$ where
$$I = \langle  \alpha\beta, \gamma\alpha, \ \alpha^m-\beta\gamma \rangle$$
As a conclusion, the algebra is special biserial, with no 'bands'.
Hence it is of finite type.
\end{proof}

\begin{rem} \normalfont 
(1) \ This Lemma  is indirectly relevant for  \cite{EHS3}.  When we analyse a block of type
        $V_1$
(with $\ve:x\to y$ and $\eta: y\to x$ where $y$ is a 1-vertex) then we
        say that if $\ve\eta \prec I$ then there must be
        a loop at $x$ and therefore $Q$ has only two vertices. This is
        excluded in  \cite{EHS3}. 
So we should clarify what happens if $\ve\eta\prec I$, this is done by the Lemma.

(2)  The algebra obtained  is one of the symmetric algebras of finite type where
        all simples have period four, as  in \cite{E26}.
\end{rem}

To continue with finding minimal relations, we have the following.
First note that by the previous Lemma we must have $\gamma\alpha \nprec I$ and $\beta\gamma \nprec I$.

\begin{lemma} We have that each of $\beta\gamma\beta$, $\gamma\beta\gamma$ and $\alpha^2$ is $\prec I$. 
\end{lemma}

The proof is similar to proofs in \cite{E90} or  \cite{EHS3}  and we omit details.
To finish finding the presentation,  we refer to \cite{E90}, and we note that this uses similar arguments as in \cite{EHS3}.

\subsection{TSP4 algebras with $n=3,4$ or $5$ simples} In the last part, we will desrcibe TSP4 algebra 
structures for remaining $n=3,4$ or $5$. The case $n=3$ is almost immediate, and for other $n$, most of 
the cases follow from known classifications. The exceptions in cases $n=4$ and $5$ are treated in Propositions 
\ref{prop:6.1}-\ref{prop:6.3}. \medskip 

Fix a TSP4 algebra $\La=KQ/I$ with Gabriel quiver $Q=Q_\La$ having $n\in\{3,4,5\}$ vertices. \smallskip 

First, consider the case $n=3$. Then by Proposition \ref{prop:Ga3}, we know that $Q$ is always 
biregular. It follows from Theorem \ref{class:bireg} that $\La$ is a weighted surface algebra, where 
in case $Q$ is $2$-regular, we have no virtual arrows. \smallskip 

For $n=4$, we know from Proposition \ref{prop:Ga4} that $Q$ is the one of $6$ quivers $Q^{(1)},\dots,Q^{(6)}$. 
All quivers except $Q^{(6)}$ are biregular, so applying Theorem \ref{class:bireg} again, we obtain that 
algebras with $Q\neq Q^{(6)}$ are weighted surface algebras. \smallskip 

Similarly, if $n=5$ then $Q$ is one of the 19 quivers $Q^{(1)},\dots,Q^{(19)}$ presented in Proposition 
\ref{prop:Ga5}, and all cases except $Q^{(12)},\dots,Q^{(17)}$ are biregular. Therefore, by Theorem \ref{class:bireg} 
again, we infer that for $Q$ different from quivers $Q^{(12)},Q^{(13)},\dots,Q^{(17)}$, 
$\La$ is a weighted surface algebra. \medskip 

Hence, we are left with seven quivers: one quiver $Q^{(6)}$ for $n=4$ and six quivers $Q^{(12)},\dots,Q^{(17)}$ 
for $n=5$. 

\begin{prop}\label{prop:6.1} If $Q$ is one of the quivers $Q^{(6)}$ for $n=4$ or $Q^{(12)},Q^{(14)},Q^{(15)},Q^{(16)}$, 
for $n=5$, then $\La$ is a virtual mutation of a weighted surface algebra (up to socle equivalence). \end{prop}

\begin{proof} First of all, observe that in all cases from the statement, quiver $Q$ contains 
the following block (of type IV), denoted by $\Gamma$.   
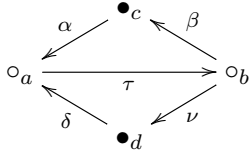
$$\xymatrix@R=0.4cm{ & \ar[ld]_{\alpha} \bullet_c& \\ 
\circ_a \ar[rr]_{\tau} && \circ_b \ar[lu]_{\beta} \ar[ld]^{\nu} \\ 
& \ar[lu]^{\delta} \bullet_d &  }$$ 
Actually, then $Q$ is a glueing of $\Gamma$ with one or two blocks of types I-III (or V1). \smallskip 

We will now use mutation $\mu_c(\La)$ of $\La$ at vertex $c$. Following Section \ref{sec:2}, $\mu_c(\La)$ is the 
endomorphism algebra $\mu_c(\La)=\End_{\cK^b_\La}(T)$ of the tilting complex 
$$T=T_c\oplus T'\in \cK^b_\La,$$  
such that $T_c=(\xymatrix@C=0.6cm{P_c\ar[r]^{\beta} & P_b})$ is concentrated in degrees $-1$ and $0$ and 
$T'=Q$ in degree $0$, where $\La=P_c\oplus Q$ ($\beta:P_c\to P_b$ is a left $\add Q$-approximation of $P_c$, 
because $\beta$ is the unique arrow in $Q$ ending at $c$). In particular, we get a TSP4 algebra 
$\La':=\mu_c(\La)=\End_{\mathcal{K}^b_\La}(T)$ (see Theorems \ref{der_sym}, \ref{der_per}, \ref{der_type}). 
Let $\La'=KQ'/I'$ be a presentation of $\La'$. \medskip  

We will now determine the Gabriel quiver $Q'$ of $\La'$ and show that it is biregular. \smallskip 

We claim first that one of $\tau\beta$ or $\tau\nu$ is involved in a minimal relation of $I$. Indeed, 
in case $Q=Q^{(12)}$, we have double arrows $\tau,\tau':a\to b$. If $\tau\beta,\tau\nu\nprec I$ 
and $\tau'\beta,\tau'\nu\nprec I$, then we have a wild subquiver of type $K_2^*$, therefore one 
of $\tau,\tau'$ satisfies the required condition, after possibly relabelling arrows. For 
$Q\neq Q^{(12)}$, there is an arrow $\bar{\tau}:a\to a'$, $\bar{\tau}\neq\tau$, and an arrow $\tau^*:b'\to b$, 
$\tau^*\neq\tau$, with $\tau^*\beta\nprec I$ and $\tau^*\nu\nprec I$. But then $\tau\beta,\tau\nu\nprec I$ 
induces the following wild subcategory in covering. \smallskip  
$$
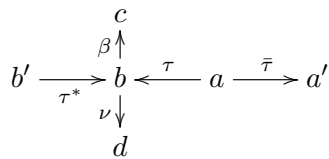
\xymatrix@R=0.4cm{& c && \\ 
b' \ar[r]_{\tau^*} & b \ar[u]^{\beta} \ar[d]_{\nu} & \ar[l]_{\tau} a \ar[r]^{\bar{\tau}} & a' \\ 
&d&&}$$ 
Consequently, we have $\tau\beta\prec I$ or $\tau\nu\prec I$. Up to relabelling of arrows, we may 
assume that $\tau\beta\prec I$. We can then adjust presentation to get $\tau\beta=0$ ($\beta$ is the 
unique arrow ending at $c$, so every relation involving $\tau\beta$ has the form $\tau\beta+z\beta=0$, 
$z\in J^2$ and we can adjust $\tau:=\tau+z$). In this case, we conclude that $\Omega_\La^2(S_c)$ has the 
unique (up to a unit) generator $\tau$, and we have the following exact sequence
$$\xymatrix{0 \ar[r] & S_c \ar[r] & P_c \ar[r]^{\beta} & P_b\ar[r]^{\tau} & P_a \ar[r]^{\alpha} & 
P_c\ar[r] & S_c \ar[r] & 0}.$$ 
In particular, we obtain that also $\alpha\tau=0$. \smallskip 

Next, using Lemma \ref{lem:3.4}, we deduce that also $\beta\alpha\prec I$. Because all paths starting at 
$b$ have the form $\beta\alpha\cdots$ or $\nu\delta\cdots$ or start with an arrow $\sigma\neq\beta,\nu$ 
($\sigma$ is unique in all cases under consideration, and $b^+=\{\beta,\nu,\sigma\}$), we can write relation 
involving $\beta\alpha$ as follows 
$$\beta\alpha+\beta\alpha z_1 + \nu\delta z_2 + A_\sigma,\leqno{(*)}$$ 
where $z_1\in J$, $z_2\in\La$ and $A_\sigma$ is a combination of paths starting with $\sigma$. 
If $z_1\in J$, then after postmultiplying by the inverse $u^{-1}$ of the unit $u=1+z_1$, 
we get a relation of the same form as above, but with $z_1=0$. If not, then $z_1$ is a unit itself, 
hence we can assume $z_1=0$, because this relation must involve $\beta\alpha$. \medskip 

Recall that we have an equivalence $\Hom(T,-): \add T \to \proj \La'$ and arrows $i\to j$ in $Q'$ 
correspond to morphisms $T_j\to T_i$ between indecomposable summands of $T$, which does not factor 
as $gf$, where $f:T_j\to \hat{T}$, $g:\hat{T}\to T_i$, and $\hat{T}$ belongs to $\add T$. \smallskip 

We will now investigate all morphisms between summands of $T$. For a morphism $f$ in $\mathcal{K}^b_{\La}$ 
we denote by $\tilde{f}$ the corresponding homomorphism $\tilde{f}=\Hom(T,f)$ in $\proj\La'$. We will 
find all $f$ such that $\tilde{f}$ corresponds to an arrow of $Q'=Q_{\La'}$. \smallskip 

(1) We claim first that there are no loops at $c$ in $Q'$. Indeed, any morphism $f:T_c\to T_c$ is a pair 
$f=(f^1,f^0)$ with $f^1:P_c\to P_c$ and $f^0:P_b\to P_b$ such that $f^0\beta=\beta f^1$. Suppose $f\neq 0$ 
belongs to $J_{\La'}$. We show that $f^1$ is in $J_\La$. Otherwise $f^1$ is an isomorphism, but $f^0$ is not 
(if also $f^0\notin J_{\La}$, then $f$ is an isomorphism, hence not in $J_{\La'}$), and therefore, 
$f^0\in J_{\La'}$, so $\beta=f^0\beta f_1^{-1}\in J^2_\La$, which is a contradiction, since $\beta$ is 
an arrow of $Q$. 
As a result, we have $f^1\in J_\La$, so there is a homomorphism $s:P_b\to P_c$ such that 
$f^1=s\beta$ ($\beta$ is the unique arrow ending at $c$). With this, one can see that the map 
$h=\beta s-f^0:P_b\to P_b$ satisfies $h\beta=0$, so it factorizes $\cok(\beta)\simeq\ima(\tau)$, and hence, 
there is $h':P_a\to P_b$ such that $h=h'\tau$; more details in (2) below. Thus $\beta s=f^0+h'\tau$, and consequently, 
using $\tau\beta=0$ we get the following commutative diagram in $\proj\La$  
$$\xymatrix{P_c \ar[r]^{\beta} \ar[d]_{f^1} & P_b \ar[d]^{f'=f^0+h'\tau} \\ P_c \ar[r]^{\beta} & P_b}$$ 
such that $s\beta=f^1$ and $\beta s=f'=f^0+h'\tau$. This means that $(f^1,f')=f+(0,h'\tau)$ is 
homotopic to zero, and consequently, we obtain that $f$ is homotopic to $-(0,h'\tau)$, which 
factorizes as $-\beta' h' \alpha'$, where $\alpha'=(0,\tau):T_c\to T_a$ and $\beta'=(0,1_{P_b}):T_b\to T_c$, 
and hence it belongs to $J^2_{\La'}$. Therefore, we have proved that any (non-zero) $f\in e_c J_{\La'} e_c$ belongs 
to $J_{\La'}^2$, and hence, we have no loops at $c$ in $Q'$. \medskip

(2) Now, let $f:T_c\to T_i$ be an arbitrary morphism from $T_c=(\xymatrix@C=0.6cm{P_c\ar[r]^{\beta} & P_b})$ 
to $T_i$ with $i\neq c$. Then $T_i$ is a stalk complex with zero degree $P_i$, and any $f$ is of the form 
$f=(0,h)$, where $h:P_b\to P_i$ satisfies $h\beta=0$. Every such $h$ factorizes through $\cok(\beta)\simeq 
\ima(\tau)=\Omega^{-2}_\La(S_c)$, i.e. $h=u\pi$, for some $u:\Omega^{-2}(S_c)\to P_i$, where $\pi:P_b\to\cok(\beta)$ 
is the cannonical projection. On the other hand, by the construction of injective resolution, also 
$\tau$ factorizes as $\tau=e\pi$, where $e:\Omega^{-2}(S_c)\to P_a$ is an injective envelope of $\Omega^{-2}(S_c)$. 
Now, since $P_i$ is injective, and $e$ is a monomorphism, we conclude that $u=h'e$, for some $h':P_a\to P_i$, 
and all together, we get the following factorization $h=h'\tau$. As a result, every morphism 
$f:T_c\to T_i$, $i\neq c$, admits a factorization $f=f'\circ(0,\tau)$, for some $f':T_a\to T_i$. \medskip 

(3) We claim that 
$f=(0,\tau):T_c\to T_a$ induces homomorphism $\tilde{f}:P_c'\to P_a'$ in $\proj\La'$ identified with an arrow 
$a\to c$ in $Q'$, denoted by $\alpha'$. Indeed, $f$ is clearly non-zero in $J_{\La'}$, so it is sufficient to see that 
$f$ does not belong to $J_{\La'}^2$. It was shown above in (1) that all morphisms $h:T_c\to T_c$ factorize as 
$h=h'f$ (up to homotopy), so including (2), we get that all morphisms $h:T_c\to T_k$ admit such a factorization. 
Consequently, if $f$ belongs to $J^2_{\La'}$, we would get a factorization $f=h_2h_1$, for $h_1:T_c\to X$ and 
$h_2:X\to T_a$ in $J_{\La'}$, and $X\in\add T$, which is impossible, because then $h_1$ factorizes as $h_1=h'_1f$, so 
we obtain $f=h_2h_1=h_2h_1'f$, and hence $(1-h_2h_1')f=0$, for an invertible element $1-h_2h_1'$ of the local 
algebra $e_a\La'e_a$. This proves, that indeed the morphism $(0,\tau):T_c\to T_a$ induces an arrow $\alpha':a\to c$ 
in $Q'$. \smallskip 

It is also the 
unique arrow in $Q'$ ending at $c$. In fact, we have no loops at $c$ and every path $\eta\in e_i J_{\La'} e_c$, $i\neq c$, 
can be written as $\eta=\eta'\alpha'$, for some $\eta'\in e_i \La e_a $, so either $\eta'\in J_{\La'}$, and 
then $\eta\in J_{\La'}^2$, so $\eta$ is not an arrow of $Q'$, or $\eta'\notin J_{\La'}$, which means that 
$i=a$ and $\eta'$ is a unit of the local algebra $e_a\La e_a$. Then $\eta$ is a scalar multiplication of $\alpha'$, 
modulo $J^2_{\La'}$. \smallskip 

In a similar way, one can check that the unique arrow starting from $c$ (and ending at vertex $\neq c$) 
is $\beta':c\to b$, identified with a morphism of the form $\beta'=(0,1_{P_b}):T_b\to T_c$. Obviously, every morphism 
$g:T_i\to T_c$ factorizes as $g=\beta'g'$, because $\beta'$ is given by identity. \smallskip 

Summing up, $c$ is a $1$-vertex of $Q'$ with $c^+=\{\beta'\}$ and $c^-=\{\alpha'\}$. \medskip 

(3') Note also that morphism $f=(\tau):T_b\to T_a$ given by an arrow $\tau:a\to b$ in $Q$ does not induce 
an arrow of $Q'$, because we have a factorization $f=(0,\tau)\circ (0,1_{P_b})=\alpha'\beta'$, so
$f\in J^2_{\La'}$. \medskip 

Using relation $(*)$, we conclude that morphisms given by arrows $\alpha',\beta'$, 
together with all morphisms of the form $\eta'=\tilde{\eta}:P_j'\to P_i'$, given by arrows $\eta:i\to j$ in $Q$ 
different from $\alpha,\beta,\tau$, generate $J_{\La'}$. \smallskip 

(4) Finally, we claim that these arrows exhaust all 
arrows in $Q'$. Indeed, let $\eta$ be a morphism $T_j\to T_i$ corresponding to an arrow 
$\eta\in e_i J_{\La'} e_j\setminus e_i J_{\La'}^2 e_j$ of $Q'$. If one of $i,j$ is $c$, then $\eta$ is either $\beta'$ 
or $\alpha'$, by (3), and we are done. 
If $i,j\neq c$, then $\eta$ is identified with a 
path $\eta\in e_i J_\La e_j=e_iJ_{\La'}e_j$, so it is sufficient to show that $\eta$ is an arrow in $Q_\La$, i.e. 
$\eta\notin J^2$ (then automatically, we get $\eta\neq\tau$, by (3'), and $\eta\neq\alpha,\beta$, because $i,j\neq c$). 
Suppose to the contrary that $\eta\in J^2_\La$, hence we have a factorization $\eta=\eta_1\gamma_1+\eta_2\gamma_2+\dots$, where 
$\eta_k,\gamma_k\in J_\La$. It follows that all summands with $t(\eta_k)\neq c$ belong to $J_{\La'}^2$, thus we can 
assume that all $t(\eta_k)=c=s(\gamma_k)$, so that each $\eta_k\gamma_k=\eta_k'\beta\alpha\gamma_k'$. 
But then, thanks to the relation $(*)$, we can rewrite each summand of $\eta$ as a combination 
of paths passing through vertices $\neq c$, and we get $\eta\in J_{\La'}^2$, a contradiction. \medskip 

It follows that $Q'$ consists of arrows 
$\alpha':a\to c$, $\beta':c\to b$, and all arrows from $Q_1\setminus\{\alpha,\beta,\tau\}$. 
In other words, if $Q$ is a glueing of a block $\Gamma$ with blocks $B_1,B_2,\dots$, then $Q'$ is 
a glueing of the following block 
$$
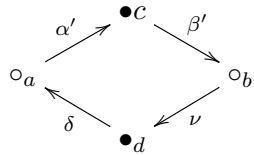
\xymatrix@R=0.4cm{& \bullet c \ar[rd]^{\beta'} & \\ \circ_a\ar[ru]^{\alpha'} && \circ_b \ar[ld]^{\nu}\\ 
&\ar[lu]^{\delta} \bullet_d & }$$ 
with the unchanged blocks $B_1,\dots$ of $Q$ (remaining blocks in $Q'$). It is clear, that after changing 
orientation of arrows $\alpha,\beta$ in $Q$ (and removing $\tau$), we get a biregular quiver $Q'$, 
which is equal (up to permutation of vertices) to $Q^{(18)},Q^{(18)},Q^{(3)}$ or $Q^{(4)}$, if $Q$ is 
$Q^{(12)},Q^{(14)},Q^{(15)}$ or $Q^{(16)}$, respectively. \medskip 

Consequently, using Theorem \ref{class:bireg}, we deduce that $\La'$ is a weighted surface algebra. Eventually, there 
is no loop at $c$ in $Q$, hence applying Theorem \ref{mut:per}, we get that $\mu_c^2(\La)\cong \La$ (up to socle), 
and therefore, algebra $\La=\mu_c(\mu_c(\La))=\mu_c(\La')$ is a mutation of a weighted surface algebra $\La'$ at vertex $c$. 
It means that $\La$ is (socle equivalent to) a virtual mutation of a weighted surface algebra $\La'$, by definition, 
and the proof is finished (see also Section \ref{subs:3.2}). \end{proof}

We have the same result for the quiver $Q^{(13)}$. 

\begin{prop}\label{prop:6.2} If $Q=Q^{(13)}$, then $\La$ is a virtual mutation of a weighted surface algebra. \end{prop}

\begin{proof} Suppose the Gabriel quiver of $\La$ has the form 
$$
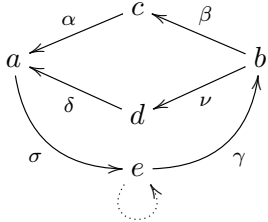
\xymatrix@R=0.25cm@C=1.2cm{ & \ar[ld]_{\alpha} c& \\ 
a \ar@/_15pt/[rdd]_{\sigma} && b \ar[lu]_{\beta} \ar[ld]^{\nu} \\ 
& \ar[lu]^{\delta} d &  \\ & e \ar@(dl,dr)@{..>}[] \ar@/_15pt/[ruu]_{\gamma}&}$$ \vspace*{0.2cm}

In this case, we must have $\sigma\gamma\nu\prec I$ or $\sigma\gamma\beta\prec I$. Indeed, otherwise 
both $\gamma\nu\delta\nprec I$ and $\gamma\beta\alpha\nprec I$, by \cite[Lemma 4.6]{EHS1}. But all paths of 
length $2$ are not involved in minimal relations of $I$, by Lemma \ref{lem:3.3}, hence we obtain the 
following wild subcategory in covering  
$$
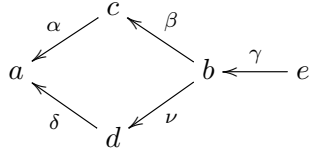
\xymatrix@R=0.4cm{ & \ar[ld]_{\alpha} c& & \\ 
a && b \ar[lu]_{\beta} \ar[ld]^{\nu} & \ar[l]_{\gamma} e\\
& \ar[lu]^{\delta} d & & }$$ 

Without loss of generality we can take $\sigma\gamma\beta\prec I$. Since $c$ is a $1$-vertex and $b^-=\{\gamma\}$, 
every path ending at $c$ has the form $\cdots\gamma\beta$, and hence, we can assume $\sigma\gamma\beta=0$ 
(after possibly adjusting $\sigma$). Consequently, the exact sequence for $S_c$ has the form 
$$\xymatrix{0 \ar[r] & S_c \ar[r] & P_c \ar[r]^{\beta} & P_b\ar[r]^{\sigma\gamma} & P_a \ar[r]^{\alpha} & 
P_c\ar[r] & S_c \ar[r] & 0},$$ 
and in particular, we also have $\alpha\sigma\gamma=0$. \medskip 

Now, we take the mutation $\La'=\mu_c(\La)$ of $\La$ at vertex $c$. This is an endomorphism algebra 
$\End_{\cK^b_\La}(T)$ of the complex $T=T_c\oplus T'=(\xymatrix@C=0.4cm{P_c \ar[r]^{\beta} & P_b})\oplus T'$. As in the proof of 
Proposition \ref{prop:6.1}, we deduce from the results of Section \ref{sec:2} that $\La'$ is again 
a TSP4 algebra. Moreover, one can easily show that the 
Gabriel quiver $Q'=Q_{\La'}$ contains reversed arrows $\alpha':a\to c$ and $\beta':c\to b$ given by 
morphisms $(0,\sigma\gamma):T_c\to T_a$ and $(0,1_{P_b}):T_b\to T_c$, respectively. Similar arguments 
prove that $\alpha'$ is the unique arrow in $Q'$ ending at $c$, and $\beta'$ is the unique arrow in 
$Q'$ starting at $c$. We may also repeat arguments from the proof of \ref{prop:6.1} and show that there 
is no loop at $c$. \medskip 

Consider a morphism $\eta:T_j\to T_i$ associated to an arbitrary arrow $\eta\in 
e_i J_{\La'} e_j\setminus e_i J_{\La'}^2 e_j$ in $Q'$. If $c=j$ or $i$, then $\eta=\alpha'$ or $\beta'$. Otherwise, 
we have $i,j\neq c$, hence $\eta\in J_\La$ is either an an arrow of $Q$ (different from $\alpha,\beta$) 
or $\eta\in J_\La^2$, and then it factorizes as $\eta=\eta_1\gamma_1+\dots$, where all $t(\eta_k)=c$ 
(modulo $J_{\La'}^2$). As before, we have $\eta_k\gamma_k=\eta'_k\beta\alpha\gamma'_k$, for all $k$. 
Observe also that the path $\beta\alpha$ induces a morphism $T_a\to T_b$, which belongs to $J_{\La'}$, 
but not to $J_{\La'}^2$, because it would give $\eta\in J^2_{\La'}$.  
Therefore, we have an arrow $\tau=\wt{\beta\alpha}:b\to a$ induced from the morphism $\beta\alpha:T_a\to T_b$. 
As a result, all $\eta_k',\gamma_k' \notin J_{\La'}$, hence these are units of local algebras, and consequently, 
$\eta=\tau$ in $Q'$. Note also that all arrows $\eta:i\to j$ in $Q_\La$ different from $\alpha,\beta$ (equivalently, 
$i,j\neq c$) induce arrows $\eta$ in $Q'$. Indeed, any such $\eta$ cannot factorize as $\eta=\eta'\gamma'$, 
for $\gamma':P_j\to T_c$ and $\eta':T_c\to T_i$ in $J_{\La'}$, since then $\eta'=\eta''\alpha'$, for 
some $\eta'':P_a\to P_i$, and hence $\eta'=(0,\eta''\sigma\gamma)$, so $\eta\in J^2$, which is impossible 
for an arrow in $Q_\La$. Clearly, $\eta$ cannot factorize through other summands $T_x=P_x$, $x\neq c$, thus 
$\eta\notin J^2_{\La'}$. \medskip 

Summing up, we have proven that $Q'$ consists of arrows $\alpha',\beta',\tau$ and all arrows of 
$Q$ different from $\alpha,\beta$, and hence, it is of the form 
$$
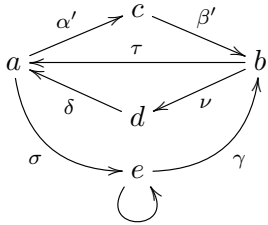
\xymatrix@R=0.25cm@C=1.2cm{ &  c\ar[rd]^{\beta'}& \\ 
a \ar@/_15pt/[rdd]_{\sigma} \ar[ru]^{\alpha'} && b \ar[ll]_{\tau}  \ar[ld]^{\nu} \\ 
& \ar[lu]^{\delta} d &  \\ & e \ar@(dl,dr)[] \ar@/_15pt/[ruu]_{\gamma}&}$$ \vspace*{0.3cm} 
 
It follows that the associated shadow $\bbQ'=\bfQ_{\bS_{\La'}}$ is $\bbQ_{26}$, and 
therefore, we have exactly two Gabriel quivers $Q'=Q^{(18)}$ or $Q'=Q^{(19)}$ of GQT algebras 
obtained from this shadow. In particular, there must be a loop at $e$ in $Q$, since 
there is a loop at $e$ in $Q'$. \medskip 

Finally, as in the previous proof Theorem \ref{class:bireg} implies that $\La'$ is a weighted 
surface algebra, and then, by Theorem \ref{mut:per}, we conclude that $\La=\mu_c^2(\La)=\mu_c(\La')$ 
is a virtual mutation of a weighted surface algebra $\La'$. \end{proof}

The last step is to find all TSP4 algebras with Gabriel quiver $Q=Q^{(17)}$. In this case, we 
have to consider larger class of algebras. 

\begin{prop}\label{prop:6.3} If $Q=Q^{(17)}$, then $\La$ is a generalized weighted surface algebra. \end{prop} 

\begin{proof} Denote vertices and arrows of $Q=Q_\La$ as follows 
$$
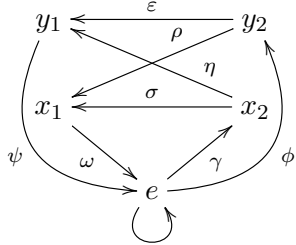
\xymatrix@R=0.7cm{ y_1 \ar@/_30pt/[rdd]_{\psi} && \ar[ll]_{\ve} y_2 \ar[lld]_(.35){\rho} \\ 
x_1 \ar[rd]_{\omega} && x_2 \ar[ll]_{\sigma} \ar[llu]_(.25){\eta} \\ 
& e \ar@(ld,rd)[] \ar[ru]_{\gamma} \ar@/_30pt/[ruu]_{\phi} & } \vspace*{0.3cm}$$ 
Due to Corollary \ref{rmkloops}, we can assume that all paths of length $2$ starting or ending at vertex 
$e$ are involved in minimal relations of $I$. All these paths are in triangles, so using Lemma \ref{lem:3.4} 
(and \ref{lem:3.3}), we get all paths of length $2$ in $Q$ involved in minimal relations 
of $I$. Since $\gamma$ is the unique arrow ending at $x_2$, we can write the relation involving $\omega\gamma\prec I$ 
as $\omega\gamma+z\gamma=0$, so that $(\omega+z)\gamma=0$ in $\La$, and hence, we can assume that $\omega\gamma=0$, 
after adjusting $\omega:=\omega+z$. In the same way, we can adjust $\psi$ to get $\psi\gamma=0$. Similarly, because 
$\omega$ is the unique arrow starting from $x_1$, we obtain $\omega\phi=0$ (possibly adjusting arrow $\phi$). \medskip 

We consider the mutation $\La':=\mu_{x_1}(\La)$ with respect to vertex $x_1$. The exact sequence for 
associated simple module $S_{x_1}$ has the following form 
$$\xymatrix{0 \ar[r] & S_{x_1} \ar[r] & P_{x_1} \ar[r]^(0.4){\vec{\sigma\\ \rho}} & P_{x_2}\oplus P_{y_2} 
\ar[r]^(.6){[\gamma \ \phi] } & P_e \ar[r]^{\omega} & P_{x_1}\ar[r] & S_{x_1} \ar[r] & 0}$$ 
where $\vec{\sigma\\ \rho}:P_{x_1}\to P_{x_2}\oplus P_{y_2}$ is a left $\add Q$-approximation of 
$P_{x_1}$, $\La=P_{x_1}\oplus Q$. Then $\La'$ is the endomorphism algebra 
$\End_{\cK^b_\La}(T)$ of tilting complex $T=T_{x_1}\oplus T'$ with 
$$T_{x_1}=(\xymatrix{P_{x_1} \ar[r]^(0.4){\vec{\sigma\\ \rho}} & P_{x_2}\oplus P_{y_2}})\quad\mbox{and}\quad 
T'=Q\mbox{ (in degree $0$).}$$ 

As in two previous proofs, $\La'$ is a TSP4 algebra of infinite representation type. We will describe the 
quiver $Q'$ and show that this is in fact $Q^{(13)}$. It is easy to check that $Q'$ contains arrows 
$\sigma':x_1\to x_2,\rho':x_1\to y_2$ and $\omega':e\to x_1$, given respectively, by morphisms 
$(0,1_{P_{x_2}}):T_{x_2}\to T_{x_1},(0,1_{P_{y_2}}):T_{y_2}\to T_{x_1}$ and $(0,[\gamma \ \phi]):T_{x_1}\to T_{e}$. 
\smallskip 

Applying arguments as presented in the proofs of Propositions \ref{prop:6.1}-\ref{prop:6.2}, one can show that 
these arrows exhaust all arrows in $Q'$ starting or ending at $x_1$. In a similar way, one can also check 
that there is no loop in $Q'$ at $x_1$, so that $x_1$ is a $(1,2)$-vertex of $Q'$. \medskip 

Now, we will determine all other arrows of $Q'$. Observe first, that any morphism $g:T_{x_1}\to T_i$, $i\neq x_1$, 
is identified with its degree zero part $g:P_{x_2}\oplus P_{y_2}\to P_i$ satisfying $g\circ\vec{\sigma \\ \rho}=0$. Using 
the exact sequence for $S_{x_1}$, we conclude that such $g$ factorizes as $g=h \circ [\gamma \ \phi]$, for some 
$h:P_e\to P_i$, equivalently, $g$ factorizes through the map identified with the arrow $\omega'$. 

Consider arbitrary arrow $\alpha:y\to x$ in $Q$ between vertices $x,y\neq x_1$. It induces a morphism 
$\alpha:T_x\to T_y$, which gives an arrow $\alpha:y\to x$ in $Q'$, provided that $x,y\neq e$. For example,  
arrows $\eta:x_2\to y_1$ and $\ve:y_2\to y_1$ become arrows in $Q'$. \smallskip 

Further, let $\alpha$ be an arrow in 
$Q$ starting from $y=e$, and suppose it belongs to $J_{\La'}^2$. Then the induced map $\alpha:T_x\to T_e$ 
factorizes as $\alpha=gf\in J^2_{\La'}$), for some $f:T_x\to T_{x_1}$ and $g:T_{x_1}\to T_e$. By the above 
property, $g$ factorizes as $g=h\circ[\gamma \ \phi]$, and $h\notin J$, because otherwise $g\in J^2$, 
and $\alpha$ is not an arrow of $Q$. If $x\neq x_2,y_2$, then $f\in J$, and $\alpha\in J^2$, a contradiction. 
This shows, that any arrow $\alpha:e\to x$ in $Q$ induces an arrow $e\to x$ in $Q'$ provided that 
$x$ is different from $x_2,y_2$. Note also that arrows $\gamma:e\to x_2$ and $\phi:e\to y_2$ are not 
arrows in $Q'$, since they factorize as $\gamma=\sigma'\omega'$ and $\phi=\rho'\omega'$ (both belong to $J_{\La'}^2$). \smallskip 

In a similar way, one can show that any arrow $\alpha:y\to e$ induces arrow $y\to e$ in $Q'$, if $y$ is 
different from $x_2,y_2$. Every element $\alpha\in e_{x_2} J e_{e}$ factorizes as $\sigma\omega z_0+\eta\psi z_1$ 
for $z_0,z_1\in \La$. But relations $\omega\gamma=\psi\gamma=0$ imply that the exact the exact sequence 
associated to $S_{x_2}$ is of the form 
$$\xymatrix{0 \ar[r] & S_{x_2} \ar[r] & P_{x_2} \ar[r]^(0.4){\gamma} &  P_e
\ar[r]^(.4){\vec{\omega \\ \psi} } &  P_{x_1}\oplus P_{y_1} \ar[r]^(.6){[\sigma \ \eta]} & P_{x_2}\ar[r] & S_{x_2} \ar[r] & 0},$$
which gives a commutativity relation 
$\sigma\omega+\eta\psi=0$, and hence every $\alpha\in J$ induces a map $\alpha:T_{e}\to T_{x_2}$ which 
belongs to $J^2$ and $J_{\La'}^2$. Similarly, one can show that any path $\alpha\in e_{y_2}J e_e$ 
also induces a map in $J^2$ and $J_{\La'}^2$. \smallskip 

As a result, we conclude that $Q'$ is a glueing of the following block 
$$
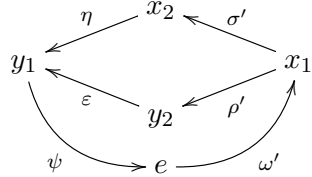
\xymatrix@R=0.25cm@C=1.2cm{ & \ar[ld]_{\eta} x_2 & \\ 
y_1 \ar@/_15pt/[rdd]_{\psi} && x_1 \ar[lu]_{\sigma'} \ar[ld]^{\rho'} \\ 
& \ar[lu]^{\ve} y_2 &  \\ & e \ar@/_15pt/[ruu]_{\omega'}&}$$
with the loop at vertex $e$. In other words, $Q'=Q^{(13)}$, and hence, by Proposition \ref{prop:6.2}, we 
conclude that $\La'$ is a virtual mutation of a weighted surface algebra. Using Theorem \ref{mut:per}, we 
conclude that $\La \cong \mu_{x_1}^2(\La)=\mu_{x_1}(\La')$ is a mutation of $\La'$. By the construction 
(see Section \ref{sec:3}), we obtain that $\La$ is a generalized weighted surface algebra with the Gabriel 
quiver being a glueing of a block of type I and a block of type V. \end{proof}

\end{document}